\documentclass[11pt,onecoulme]{article}
\usepackage{amsfonts}
\usepackage{mathrsfs}
\usepackage{amssymb}
\usepackage{color, amsmath,amssymb, amsfonts, amstext,amsthm, latexsym}

\usepackage{amssymb, epsfig, amssymb, latexsym}
\usepackage{amsmath}
\usepackage{graphicx}
\usepackage{longtable}

\numberwithin{equation}{section}

\allowdisplaybreaks

\newtheorem{definition}{Definition}[section]
\newtheorem{theorem}[definition]{Theorem}
\newtheorem{lemma}[definition]{Lemma}
\newtheorem{remark}[definition]{Remark}

\newtheorem{hyp}[definition]{Hypothesis}

  \renewcommand\appendix{\par
    \setcounter{section}{0}
    \setcounter{subsection}{0}
    \gdef\thesection{ Appendix \Alph{section}}}
\title{
                 Viscosity Solutions to Second Order Elliptic Hamilton-Jacobi-Bellman Equation with infinite delay
}
\date{}
\author{
Jianjun Zhou \thanks{ College of Science,
             Northwest A\&F University, Yangling 712100, Shaanxi, P. R.
             China. Partially supported by  the National Natural Science Foundation of China  (Grant No. 11401474),  the  Natural Science Foundation of Shaanxi Province (Grant No. 2021JM-083) and the Fundamental Research Funds for the Central Universities (Grant No. 2452019075, 2452021063). Email: zhoujianjun@nwsuaf.edu.cn}
                }
\begin{document}
\maketitle
\par\noindent
{\bf Abstract}

\par
This paper introduces a notion of viscosity solutions  for second order elliptic Hamilton-Jacobi-Bellman (HJB) equations with infinite delay associated with infinite-horizon optimal control problems for stochastic differential equations with infinite delay.
We identify the value functional of optimal control problems as  unique viscosity solution to  associated second order elliptic HJB equation with infinite delay. We also show that our notion of viscosity solutions is consistent with the corresponding notion of classical solutions, and satisfies a stability property.

\section{ Introduction}

      Let $\{W(t),t\geq0\}$ be an $n$-dimensional
      standard Wiener process on a  complete probability  space $(\Omega,{\cal {F}}, P)$, and $\{{\mathcal {F}}_{s}\}_{s\geq 0}$  its natural filtration, augmented with the family ${\cal{N}}$ of $P$-null of $\cal{F}$. The process  $u(\cdot)=(u(s))_{s\in [0,\infty)}$ is  ${\mathcal {F}}_{s}$-progressively measurable  and take values in some Polish space $(U,d)$ (subsequently called  $u(\cdot)\in {\cal{U}}_0$).
      ${\cal{C}}_0$  is the totality of all
continuous $\mathbb{R}^d$-valued functions defined on  $(-\infty,0]$ with $\lim_{\theta\rightarrow-\infty}y(\theta)=0$.
    Define a   norm on ${\cal{C}}_0$  as follows:
\begin{eqnarray*}
   |x|_C:=\sup_{-\infty< \theta\leq 0}|x(\theta)|, \ \  \ x\in {\cal{C}}_0.
\end{eqnarray*}
 We assume the coefficients $b:{\cal{C}}_0\times U\rightarrow \mathbb{R}^d$ and  $\sigma:{\cal{C}}_0\times U\rightarrow \mathbb{R}^{d\times n}$   satisfy  Lipschitz condition under $|\cdot|_C$
                         with respect to  the continuous function.

            Let us  consider a controlled  stochastic  differential
                 equation with infinite delay:
\begin{eqnarray}\label{state1}
            \begin{cases}
            dX^{x,u}(s)=
            b(X^{x,u}_s,u(s))ds+\sigma(X^{x,u}_s,u(s))dW(s),  \ \ s>0,\\
            ~~~~~X^{x,u}_0=x\in {\cal{C}}_0,
\end{cases}
\end{eqnarray}
        where
$$
        X^{x,u}_s(\theta)=X^{x,u}(s+\theta),\ \theta\in(-\infty,0].
$$
In the above equation,   the unknown $X^{x,u}(s)$, representing the state of the system, is an $\mathbb{R}^d$-valued process.
    \par
              One tries to minimize a utility  functional of the form:
\begin{eqnarray}\label{cost1}
                     J(x,u(\cdot))=\mathbb{E}\int_{0}^{\infty}e^{-\lambda s}q(X^{x,u}_s,u(s))ds, \  \ x\in {\cal{C}}_0,
\end{eqnarray}
over  ${\cal{U}}_0$.
            Here, the term $q$ is a given real  function on ${\cal{C}}_0\times U$, and the constant $\lambda$ is sufficiently large. 
                           We   introduce  value functional of the  optimal
                  control problem as follows:
\begin{eqnarray}\label{eq3}
                  V(x):=\inf_{u(\cdot)\in{{\mathcal
                  {U}}_0}}J(x,u(\cdot)), \ \ x\in {\cal{C}}_0.
\end{eqnarray}
 We aim at characterizing this value functional $V$.   We   consider the following second order elliptic Hamilton-Jacobi-Bellman
                     (HJB) equation with infinite delay:
  \begin{eqnarray}\label{hjb1}
-\lambda V(x)+\partial_tV(x)+{\mathbf{H}}(x,\partial_xV(x),\partial_{xx}V(x))= 0,\ \ \  x\in  {\cal{C}}_0,
\end{eqnarray}
      where
\begin{eqnarray*}
                                {\mathbf{H}}(x,p,l)=\inf_{u\in{
                                         {U}}}[
                        (p,b(x,u))_{\mathbb{R}^d}+\frac{1}{2}\mbox{tr}[ l \sigma(x,u)\sigma^\top(x,u)]
                        +q(x,u)],  \ (x,p,l)\in  {\cal{C}}_0\times \mathbb{R}^d\times {\cal{S}}(\mathbb{R}^d).
\end{eqnarray*}
Here,   $\sigma^\top$ is the transpose of the matrix $\sigma$,  ${\cal{S}}(\mathbb{R}^d)$  the set of all $(d\times d)$ symmetric matrices, $(\cdot,\cdot)_{\mathbb{R}^d}$ the scalar product of $\mathbb{R}^d$, and
                         $\partial_t$, $\partial_{x}$ and $\partial_{xx}$ the so-called pathwise  (or functional or Dupire;
                          see \cite{dupire1, cotn0, cotn1})  derivatives, where $ \partial_t$  is known as horizontal derivative,
                          while
 $\partial_x$  and $\partial_{xx}$ are first and  second order vertical derivatives, respectively. 
\par
             The type of problem described above arises in many  fields (for an overview on
their applications see  Kolmanovskii and Shaikhet \cite{kol}). We refer to Elsanousi,  ${\O}$ksendal and Sulem \cite{els}, {\O}ksendal and  Sulem  \cite{oks} and {\O}ksendal,  Sulem and Zhang \cite{oks1}
for applications to   mathematical finance, to  Gozzi and Marinell \cite{gozz1} and Gozzi,  Marinelli and Savin \cite{gozz2} for  advertising models with
              delayed  effects and  to Federico \cite{fed1} and Gabay and Grasselli \cite{gab} for  Pension funds.
              \par
             Crandall and Lions \cite{cra1} first
           introduced the notion of viscosity solutions for first order HJB equations in finite dimension, and later Lions \cite{liojia,liojia1,lio} extended the notion to  second order case.
            Reader can also refer to the survey paper of Crandall, Ishii and  Lions \cite{cran2} and  the monographs of Fleming
      and Soner \cite{fle1} and Yong and Zhou \cite{yong} for a detailed account
for the theory of viscosity solutions. For viscosity solutions of second order HJB equations in infinite dimensional space, we refer to
   Fabbri, Gozzi and  \'{S}wi\c{e}ch \cite{fab1},  Gozzi,  Rouy and \'{S}wi\c{e}ch \cite{gozz3},  Lions \cite{lio1, lio2, lio3}  and \'{S}wi\c{e}ch \cite{swi,swi1}. One of the structural assumption is
that the state space has to be a Hilbert space or certain Banach
space with smooth norm, not including the continuous function space.  For the delay  case, the theory of viscosity solutions is more difficult.
  Our paper \cite{zhou3} studied  a class of infinite-horizon optimal control problems for stochastic differential equations with finite delay, in which the term
              $X(\cdot-\tau)$ 
              only occurs in the drift term in the state equation.
 \par
                         In this paper, we want to develop a concept of viscosity solutions to second order elliptic
                         HJB equations with infinite delay. 
We adopt the natural generalization of the well-known Crandall-Lions definition  in terms of test functions and then show that the value functional
                         $V$  defined in  (\ref{eq3}) is the  unique viscosity solution to HJB
                         equation  given in  (\ref{hjb1}) when  coefficients $b,\sigma$ and $q$ only satisfy Lipschitz conditions under $|\cdot|_C$.
 \par
          The standard approach to the treatment of  delay optimal control problems
          is to reformulate them as  optimal control problems of the evolution equation
                in a Hilbert space  by inserting  the delay term in the linear unbounded operator of the evolution equation (see, e.g.,
                Carlier and Tahraoui \cite{car1} and Federico,  Goldys and  Gozzi \cite{fed} for deterministic cases, and  Gozzi and Marinelli \cite{gozz1} for stochastic cases).
                 Since coefficients $b$ and $\sigma$ are genuinely
              nonlinear functions about $X^{x,u}_{\cdot}$,   the standard techniques  are not applicable in
                 our case.
                 By  reformulating
the original control problem as  infinite dimensional stochastic control
problem in a suitable Banach space,  Federico \cite{fed1} studied the optimal control of a stochastic  differential equation with delay
arising in the management of a pension fund with surplus,  in which the delay in the state equation concentrates on a point of the past and appears in a nonlinear way. However,
as the author said  that the uniqueness of viscosity solutions is
a  difficult topic and it is not clear whether their  definition of solution is strong enough
to guarantee such a result or not.
\par
 Our paper \cite{zhou5} introduced a Crandall-Lions  definition of viscosity solutions for second order parabolic path-dependent HJB equations associated with optimal control problems for path-dependent stochastic differential equations (PSDEs), and
      identified the value functional of optimal control problems as unique viscosity solution to associated path-dependent HJB equations when
       coefficients  only satisfy $d_{\infty}$-Lipschitz conditions with respect to
        path function. However,   this path-dependent case does not include our delay case since  the coefficients of  PSDEs depend essentially on time $t$.
 \par

          Our main results are as follows.  For every $m\in \mathbb{N}^+$ and $M\in \mathbb{R}$, we define  functional $\Upsilon^{m,M}:[0,\infty)\times {\cal{C}}_0\times [0,\infty)\times {\cal{C}}_0\rightarrow \mathbb{R}$ by
          $$
          \Upsilon^{m,M}(t,x,s,y)=S_m(t,x,s,y)+M|x(0)-y(0)|^{2m}, \  ~~ (t,x), (s,y)\in [0,\infty)\times {\cal{C}}_0,
          $$
          and
           \begin{eqnarray*}
S_m(t,x,s,y)=\begin{cases}
            \frac{(|x_{(s-t)\vee0}-y_{(t-s)\vee0}|_C^{2m}-|x(0)-y(0)|^{2m})^3}{|x_{(s-t)\vee0}-y_{(t-s)\vee0}|^{4m}_{C}}, \
         ~~ |x_{(s-t)\vee0}-y_{(t-s)\vee0}|_{C}\neq0; \\
0, ~~~~~~~~~~~~~~~~~~~~~~~~~~~~~~~~~~~~~~~~~~~~~ |x_{(s-t)\vee0}-y_{(t-s)\vee0}|_{C}=0,
\end{cases}
\end{eqnarray*}
 where  $x_h\in {\cal{C}}_0$ and
$
  x_h(\theta):=x(0){\mathbf{1}}_{[-h,0]}(\theta)+x(\theta+h){\mathbf{1}}_{(-\infty,-h)}(\theta), \ \ \theta\in (-\infty,0]
$ for every $(h,x)\in [0,\infty)\times {\cal{C}}_0$.
\par
This key functional  is the starting point for the proof of uniqueness result. 
 First, we define
   $$
                 \overline{\Upsilon}^{3,3}(t,x,s,y):=\Upsilon^{3,3}(t,x,s,y)+|s-t|^2, \ \ (t,x), (s,y)\in [0,\infty)\times {\cal{C}}_0.
                 $$
 Unfortunately, $\overline{\Upsilon}^{3,3}$ is not a  gauge-type function in space $[0,\infty)\times {\cal{C}}_0$ with usual norm $|\cdot|_1$, where $|(t,x)|_1=|t|+|x|_C$ for every $(t,x)\in [0,\infty)\times {\cal{C}}_0$ (see Remark \ref{remarks} (iv)). To overcome this difficulty,  we define a metric on $[0,\infty)\times {\cal{C}}_0$  as follows: for any $ 0\leq t\leq s<\infty$ and $(t,x), (s,y)\in [0,\infty)\times {\cal{C}}_0$,
\begin{eqnarray}\label{2.1}
    d_{\infty}((t,x),(s,y))=d_{\infty}((s,y),(t,x)):=|s-t|+|x_{s-t}-y|_C,
\end{eqnarray}
 and show that $\overline{\Upsilon}^{3,3}$ is  a  smooth  gauge-type function in space $([0,\infty)\times {\cal{C}}_0, d_{\infty})$. Then
                 a modification of Borwein-Preiss variational principle (see Theorem 2.5.2 in  Borwein \& Zhu  \cite{bor1}) can be used  to get a maximum of a perturbation of the auxiliary function $\Psi$.
\par
 Second,
 for every fixed $(t,x)\in [0,\infty)\times{\cal{C}}_0$,   define $f:[t,\infty)\times {\cal{C}}_0\rightarrow \mathbb{R}$ by
 $$
 f(s,y):=\Upsilon^{3,3}(t,x,s,y), \ \  (s,y)\in[t,\infty)\times {\cal{C}}_0.
 $$
 We  study its  regularity in the horizontal/vertical sense and show that it satisfies  a functional It\^o formula. Then the test function in our definition of viscosity solutions and the auxiliary function $\Psi$ in the proof of uniqueness  can include $f$. More importantly, $f(s,y)$  is equivalent to $|y-x_{s-t}|_C^6$.  By this,  the uniqueness  result is  established
            when the coefficients only satisfy   Lipschitz assumption under $|\cdot|_C$.
                 \par
                 Finally, notice that the key  functional includes time term $t$, different from the definition of viscosity solutions to
                  classic  elliptic HJB equations (see Definition 2.2 in Crandall,  Ishii  and  Lions \cite{cran2}), our definition of viscosity solutions includes time term $t$ and
                   horizontal derivative $\partial_t$. However, we show that our definition of viscosity solutions  is a natural  extension of viscosity
                    solutions to  classic  elliptic HJB equations.
\par
Regarding existence, we  show that the value functional  $V$  defined in  (\ref{eq3}) is  a viscosity solution to  HJB
                         equation with infinite delay given in  (\ref{hjb1}) by functional It\^o  formula  and dynamic programming principle (DPP).
                         Such a formula was firstly provided  in
                          Dupire  \cite{dupire1}
   (see also Cont and  Fourni\'e \cite{cotn0}, \cite{cotn1}). 
   We point out that the functional $\Upsilon^{1,3}$ is also the key in the proof of existence.
In  infinite-horizon optimal control problem, it is an important problem to find a suitable $\lambda$ such that (\ref{cost1}) is well-defined. Applying functional It\^o formula to $\Upsilon^{1,3}$, there exists a constant  $\Theta>0$, such that (\ref{cost1}) is well-defined when $\lambda> \Theta$ (see Theorem \ref{theoremv}). This constant is smaller than that obtained by It\^o formula.

 \par
  We also mention that a new concept of viscosity solutions for semi-linear path-dependent partial differential equations
was introduced by Ekren, Keller, Touzi
 and Zhang \cite{ekren1} in terms of a nonlinear expectation, and further extended to
 fully nonlinear parabolic
equations by Ekren, Touzi, and Zhang \cite{ekren3, ekren4},  elliptic
equations by Ren \cite{ren}, obstacle problems by Ekren \cite{ekren0}   when the  Hamilton function $\mathbf{H}$ is uniformly nondegenerate, and degenerate second-order
equations by Ren, Touzi, and Zhang \cite{ren1} and Ren and Rosestolato \cite{ren2} when  the nonlinearity  $\mathbf{H}$ is $d_p$-uniformly continuous in the path function.
However, none of the results we know are directly applicable to our situation as in our case the Hamilton function $\mathbf{H}$ may be degenerate and is only required to have continuity properties under supremum norm $|\cdot|_C$.
             \par
                 The outline of this article is  as follows. In the following
              section, we introduce the framework of \cite{cotn1} and \cite{dupire1} and a modification of Borwein-Preiss variational principle.
           We  present  the smooth functionals $S_m$ which are the key to proving the stability and uniqueness results of viscosity solutions in Section 3.
           In Section 4, we introduce  preliminary results on  stochastic delay optimal control problems. 
           In Section 5, we define classical and viscosity solutions to our
             HJB equations and  prove  that the value functional $V$ defined by (\ref{eq3}) is a viscosity solution
               to equation (\ref{hjb1}).   We also show
             the consistency with the notion of classical solutions and the stability result.
                 A maximum principle for delay case is given and the uniqueness of viscosity solutions for  (\ref{hjb1}) is proved in Section 6.

\section{ Preliminaries}
\par
2.1. ${Notations \ and \ Spaces}$.
 We list some notations that are used in this paper.
 For the vectors $x,y\in \mathbb{R}^d$, the scalar product is denoted by $(x,y)_{\mathbb{R}^d}$ and the
         Euclidean norm $(x,x)^{\frac{1}{2}}_{\mathbb{R}^d}$ is denoted by $|x|$ (we use the same symbol $|\cdot|$ to denote the Euclidean norm on ${\mathbb{R}}^k$, for any $k\in \mathbb{N}$). If $A$ is a vector or matrix, its transpose is denoted by $A^\top$; For  a matrix $A$, denote its operator norm and
          Hilbert-Schmidt norm    by $|A|$ and $|A|_2$, respectively.
          Let
         ${\cal{D}}:=D((-\infty,0],\mathbb{R}^d)$,  the set of c\`adl\`ag $\mathbb{R}^d$-functions on $(-\infty,0]$, and ${\cal{C}}:=C((-\infty,0],\mathbb{R}^d)$  the family of continuous functions from $(-\infty,0]$ to $\mathbb{R}^d$.
         Let ${\cal{D}}_0=\{f\in {\cal{D}}: \lim_{\theta\rightarrow-\infty}f(\theta)=0 \}$ and ${\cal{C}}_0=\{f\in {\cal{C}}: \lim_{\theta\rightarrow-\infty}f(\theta)=0 \}$.
  We define a norm on ${\cal{D}}_0$  as
         follows:
   $$
                      |x|_{C}= \sup_{\theta\in
                      (-\infty,0]}|x(\theta)|,\ \ x \in {\cal{D}}_0.
   $$
               Then, $({\cal{D}}_0,|\cdot|_{C})$  and $({\cal{C}}_0,|\cdot|_{C})$ are  Banach spaces.
               Following Dupire \cite{dupire1},  for $ x\in {\cal{D}}_0$, $h\geq0$, we define $x_h\in {\cal{D}}_0$ as
\begin{eqnarray*}
  x_h(\theta):=x(0){\mathbf{1}}_{[-h,0]}(\theta)+x(\theta+h){\mathbf{1}}_{(-\infty,-h)}(\theta), \ \ \theta\in (-\infty,0].
\end{eqnarray*}
               \par
                For each  $t\in[0,\infty)$,
          define
         $\hat{\Lambda}^t:=[t,\infty)\times {\cal{D}}_0$ and  ${\Lambda}^t:=[t,\infty)\times {\cal{C}}_0$. Let $\hat{\Lambda}$ and ${\Lambda}$ denote $\hat{\Lambda}^0$ and $\hat{\Lambda}^0$, respectively.
We define a metric on $\hat{\Lambda}^t$  as follows: for any $ t\leq s\leq l<\infty$ and $(s,x), (l,y)\in \hat{\Lambda}^t$,
\begin{eqnarray}\label{2.1}
    d_{\infty}((s,x),(l,y))=d_{\infty}((l,y),(s,x)):=|l-s|+|x_{l-s}-y|_C.
\end{eqnarray}
Then  $(\hat{\Lambda}^t, d_{\infty})$ and $({\Lambda}^t, d_{\infty})$ are complete metric spaces. 
\par
Now we define the pathwise  derivatives of Dupire \cite{dupire1}.
 \begin{definition}\label{definitionc0} (Pathwise derivatives)
       Let $t\in [0,\infty)$ and  $f:\hat{\Lambda}^t\rightarrow \mathbb{R}$.
\begin{description}
        \item{(i)}  Given $(s,x)\in \hat{\Lambda}^t$, the horizontal derivative of $f$ at $(s,x)$ (if the corresponding limit exists and is finite) is defined as
        \begin{eqnarray}\label{2.3}
               \partial_tf(s,x):=\lim_{h\rightarrow0,h>0}\frac{1}{h}\left[f(s+h,x_h)-f(s,x)\right].
\end{eqnarray}
If the above limit exists and is finite for every $(s,x)\in \hat{\Lambda}^t$, the map $\partial_tf:\hat{\Lambda}^t\rightarrow \mathbb{R}$ is called the horizontal derivative of $f$ with domain $\hat{\Lambda}^t$.
       \par
      \item{(ii)}  Given $(s,x)\in \hat{\Lambda}^t$,  the vertical derivatives of first and second order of $f$ at $(s,x)$
(if the corresponding limit exists and is finite) are defined as
      \begin{eqnarray}\label{2.20jia}
 \partial_{x}f(s,x):=(\partial_{x_1}f(s,x),\partial_{x_2}f(s,x),\ldots, \partial_{x_d}f(s,x)),
\end{eqnarray}
and
\begin{eqnarray}\label{2.20jia}
 \partial_{xx}f(s,x):=(\partial_{x_ix_j}f(s,x))_{i,j=1,2,\ldots,d},
\end{eqnarray}
where
        \begin{eqnarray}\label{2.2}
 \partial_{x_i}f(s,x):=\lim_{h\rightarrow0}\frac{1}{h}\bigg{[}f(s,x+he_i\mathbf{1}_{\{0\}})-f(s,x)\bigg{]},\
 i=1,2,\ldots,d,
 \end{eqnarray}
  \begin{eqnarray}\label{2.20204}
 \partial_{x_ix_j}f(s,x)&:=&\partial_{x_i}(\partial_{x_j}f)(s,x),\
 i,j=1,2,\ldots,d,
\end{eqnarray}
with $e_1,e_2,\ldots,e_d$ is the standard  orthonormal basis of $\mathbb{R}^d$.
If the above limits in (\ref{2.2}) exist and are finite for every $(s,x)\in \hat{\Lambda}^t$, the map $\partial_xf:=(\partial_{x_1}f,\partial_{x_2}f,\ldots, \partial_{x_d}f)^\top:\hat{\Lambda}^t\rightarrow \mathbb{R}^d$
is called the first order vertical  derivative of $f$ with domain $\hat{\Lambda}^t$,
and if the above limits in (\ref{2.20204}) exist and are finite for every $(s,x)\in \hat{\Lambda}^t$, the map  $\partial_{xx}f:=(\partial_{x_ix_j}f)_{i,j=1,2,\ldots,d}:\hat{\Lambda}^t\rightarrow {\cal{S}}(\mathbb{R}^d)$ is called the second order vertical  derivative of $f$ with domain $\hat{\Lambda}^t$,
where ${\cal{S}}(\mathbb{R}^d)$ is the
space of all $d\times d$ symmetric matrices.
\end{description}
\end{definition}
\begin{definition}\label{definitionc}
       Let $t\in[0,\infty)$ and $f:\hat{\Lambda}^t\rightarrow \mathbb{R}$ be given.
\begin{description}
        \item{(i)}
                 We say $f\in C(\hat{\Lambda}^t)$ if $f$ is continuous in $(s,x)$  on $(\hat{\Lambda}^t,d_\infty)$;
\par
       \item{(ii)}  We say
        $f\in C^{1,2}(\hat{\Lambda}^t)\subset C(\hat{\Lambda}^t)$ if   $\partial_{t}f$, $\partial_{x}f$ and $\partial_{xx}f$ exist and are continuous in $(s,x)$  on $(\hat{\Lambda}^t,d_\infty)$;
       \par
       \item{(iii)} We say $f\in C^{1,2}_p(\hat{\Lambda}^t)\subset C^{1,2}(\hat{\Lambda}^t)$ if $f$, $\partial_{t}f$, $\partial_{x}f$ and $\partial_{xx}f$  grow  in a polynomial way.
\end{description}
\end{definition}
 For every $t\in [0,\infty)$,  $f:{\Lambda}^t\rightarrow \mathbb{R}$ and $\hat{f}:\hat{\Lambda}^t\rightarrow \mathbb{R}$ are called consistent
  on ${\Lambda}^t$ if $f$ is the restriction of $\hat{f}$ on ${\Lambda}^t$.
\begin{definition}\label{definitionc2}
      Let $t\in[0,\infty)$ and $f:{\Lambda}^t\rightarrow \mathbb{R}$ be given.
\par
       \item{(i)} We say $f\in C({\Lambda}^t)$ if $f$ is continuous in $(s,x)$ on $({\Lambda}^t,d_\infty)$. 
\par
       \item{(ii)}  We say $f\in C_p^{1,2}({\Lambda}^t)\subset C({\Lambda}^t)$ if
                     there exists $\hat{f}\in C_p^{1,2}(\hat{\Lambda}^t)$ which is consistent with $f$ on ${\Lambda}^t$.
\end{definition}
\begin{definition}\label{1120}
 Let $f: {\cal{C}}_0\rightarrow \mathbb{R}$ be given. Define $\hat{f}:\Lambda\rightarrow \mathbb{R}$ by $\hat{f}(t,x):=f(x),  \  (t,x)\in \Lambda$.
  \begin{description}
 \item{(i)}  We say $f\in C( {\cal{C}}_0)$ if $f$ is continuous in $x$ on $({\cal{C}}_0, |\cdot|_C)$.
 \item{(ii)}  We say $f\in C(\Lambda)$ if $\hat{f}\in C(\Lambda)$. 
  \item{(iii)} We say $f\in C_p^{1,2}( {\cal{C}}_0)$ if $\hat{f}\in C_p^{1,2}(\Lambda)$; and we define  $\partial_tf$, $\partial_{x}f$ and $\partial_{xx}f$ by
  $$\partial_tf(x):=\partial_t\hat{f}(0,x), \ \ \ \partial_xf(x):=\partial_{x}\hat{f}(0,x), \ \ \  \partial_{xx}f(x):=\partial_{xx}\hat{f}(0,x),\ \ \ x\in  {\cal{C}}_0.$$
 \end{description}
\end{definition}
                   We also have
 \begin{lemma}\label{theoremv1202}   $f\in C({\cal{C}}_0)$ if and only if  $f\in C(\Lambda)$.
\end{lemma}
{\bf  Proof}. \ \
For every $(t,x), (s,y)\in \Lambda$, if $t\leq s$,
$$
              |x-y|_C\leq |x_{s-t}-y|_C+|x-x_{s-t}|_C\leq d_\infty((t,x),(s,y))+|x-x_{s-t}|_C;
$$
if $t>s$,
$$
              |x-y|_C\leq |z-y|_C+|x-z|_C\leq d_\infty((t,x),(s,y))+\sup_{\theta\in (-\infty,0]}|x(\theta)-x(\theta+s-t)|,
$$
where $z(\theta)=x(\theta+s-t)$ for all $\theta\in (-\infty,0]$.
Letting $(s,y)\rightarrow(t,x)$ in $(\Lambda,d_\infty)$, we have  $y\rightarrow x$ in $({\cal{C}}_0,|\cdot|_C)$ and $s\rightarrow t$.  Then if $f\in C({\cal{C}}_0)$, we get
$\hat{f}(s,y)=f(y)\rightarrow f(x)=\hat{f}(t,x)$ as $(s,y)\rightarrow(t,x)$ in $(\Lambda,d_\infty)$.
Thus,  $f\in C(\Lambda)$.  It is clear that $f\in C({\cal{C}}_0)$ if  $f\in C(\Lambda)$. The proof is now complete. \ \ $\Box$
\par
 Let $\{W(t),t\geq0\}$ be a $n$-dimensional standard Wiener process defined on  a  complete probability  space
        $(\Omega,{\cal {F}},\mathbb{P})$.  Let $\{{\mathcal {F}}_{t}\}_{t\geq0}$ be the natural filtration of $W(t)$, augmented with the family ${\cal{N}}$ of $\mathbb{P}$-null of $\cal{F}$.
      The filtration
           $\{{\mathcal {F}}_{t}\}_{t\geq0}$  satisfies the usual condition.
   For every $[t,s]\subset[0,\infty)$, we also use the
       notations:
       $$
              {\mathcal{F}}_{t}^{s,0}=\sigma(W(l)-W(t):l\in[t,s]),\  \ \  \ {\mathcal{F}}_{t}^{s}={\mathcal{F}}_{t}^{s,0}\vee \mathcal
             {N}.
       $$
        We also write ${\cal{F}}^{t,0}$ for $\{{\cal{F}}^{s,0}_t \}_{s\geq t}$ and ${\cal{F}}^t$ for $\{{\cal{F}}^s_t\}_{ s\geq t}$.

Next we define several classes of random variables or stochastic processes with values in a Banach space $(K,|\cdot|_K)$.

  $\bullet L^p(\Omega,{\cal{F}}_t;{K})$  defined for all  $t\geq0$ and $p\geq 1$, denotes   the space  of all ${\cal{F}}_t$-measurable maps $\xi: \Omega\rightarrow{K}$ satisfying $\mathbb{E}|\xi|^p_K<\infty$.

 $\bullet L^{2}_{\mathcal{P}}(\Omega\times [t,T];K)$  defined for all $0\leq t<T<\infty$, denotes   the space of equivalence classes of  processes $y\in L^2(\Omega\times [t,T];K)$,  admitting a predictable version.  $ L^{2}_{\mathcal{P}}(\Omega\times [t,T];K)$
  is endowed with the norm
$$|{y}|^{2}=\mathbb{E}\int_{t}^{T}|{y}(s)|_K^{2}ds.$$

$\bullet L^{2}_{\mathcal{P},l}(\Omega\times [t,\infty);K)$  defined for all $t\geq0$, denotes   the space of equivalence classes of  processes $\{y(s),s\geq t\}$,  with values in $K$ such that
 $y|_{[t,T]}\in L^{2}_{\mathcal{P}}(\Omega\times [t,T];K)$ for all $T>t$, where $y|_{[t,T]}$ denotes the restriction of  $y$ to the interval
$[t,T]$.

\vbox{}
2.2. \emph{Functional   It\^o formula}.
Assume that $\vartheta\in L^{2}_{\mathcal{P},l}(\Omega\times [0,\infty);\mathbb{R}^d)$,
 $\varpi\in L^{2}_{\mathcal{P},l}(\Omega\times [0,\infty);\mathbb{R}^{d\times n})$ and $(\tau,x)\in {\Lambda}$, 
then the following process
\begin{eqnarray}\label{formular1}
                 X(s)=x(0)+\int^{s}_{\tau}\vartheta(\sigma)d\sigma+\int^{s}_{\tau}\varpi(\sigma)dW(\sigma),\ s\geq \tau\geq0,
\end{eqnarray}
and  $X(s)=x(s-\tau),\ s\in (-\infty,\tau)$,
is a continuous semi-martingale on $[\tau, \infty)$.
 The following
               functional It\^o formula  is needed to prove the existence  of viscosity solutions.
\begin{lemma}\label{theoremito}
\ \
Suppose  $f\in C_p^{1,2}(\hat{\Lambda}^{t})$
for some $t\in[0,\infty)$. Then, under the above conditions, $\mathbb{P}$-a.s.,  for all $\tau\vee t\leq \hat{t}\leq s<\infty $:
\begin{eqnarray}\label{statesop0}
                 f(s,X_s)&=&f(\hat{t},X_{\hat{t}})+\int_{\hat{t}}^{s}[\partial_tf(\sigma,X_\sigma)+(\partial_xf(\sigma,X_\sigma),\vartheta(\sigma))_{\mathbb{R}^d}
                 +\frac{1}{2}\mbox{tr}(\partial_{xx}f(\sigma,X_\sigma)\varpi(\sigma)\varpi^\top(\sigma))]d\sigma\nonumber\\
                 &&+\int^{s}_{\hat{t}}\partial_xf(\sigma,X_\sigma)\varpi(\sigma)dW(\sigma).
\end{eqnarray}
 Here and in the following, for every $s\in \mathbb{R}$, $X(s)$ denotes  the value of $X$  at
 time $s$, and  $X_s$ the function from $(-\infty,0]$ to $\mathbb{R}^d$ by
 $
                          X_s(\theta)=X(s+\theta),\ \ \theta\in (-\infty,0].
 $
\end{lemma}
The proof is  similar to Theorem 4.1 in  Cont \& Fournie \cite{cotn1} (see also Dupire \cite{dupire1}). Here we omit it.

\par
By the above Lemma, we have the following important results.
\begin{lemma}\label{0815lemma}
             Let $f\in C_p^{1,2}(\Lambda^t)$ and $\hat{f}\in C_p^{1,2}(\hat{\Lambda}^t)$ such that $\hat{f}$ is consistent with $f$ on $\Lambda^t$, then the following definition
             $$
             \partial_tf:=\partial_t\hat{f}, \ \ \ \partial_xf:=\partial_x\hat{f}, \ \ \ \partial_{xx}f:=\partial_{xx}\hat{f} \ \ \mbox{on} \ \Lambda^t
             $$
              is independent of the choice of $\hat{f}$. Namely, if there is another $\hat{f}'\in C_p^{1,2}(\hat{\Lambda}^t)$ such that $\hat{f}'$ is consistent with $f$ on $\Lambda^t$, then the derivatives of $\hat{f}'$
              coincide with those of $\hat{f}$ on $\Lambda^t$.
\end{lemma}
\par
{\bf  Proof}. \ \
 By the definition of the horizontal derivative, it is clear that $\partial_t\hat{f}(l,x)=\partial_t\hat{f}'(l,x)$ for every $(l,x)\in \Lambda^t$.
%
%
  Next,  for every $(l,x)\in \Lambda^t$, let   $\varpi=\mathbf{0}$,  $\tau=l$
   and $\vartheta\equiv h\in \mathbb{R}^d$ in (\ref{formular1}),  by  Lemma \ref{theoremito},
$$
                        \int^{s}_{l}(\partial_x\hat{f}(\sigma,X_\sigma),h)_{\mathbb{R}^d}d\sigma=\int^{s}_{l}(\partial_x\hat{f}'(\sigma, X_\sigma),h)_{\mathbb{R}^d}d\sigma, \ \ s\in [l,\infty),
$$
where $X_\sigma(r)=x((\sigma+r-l)\wedge 0)+(\sigma+r-l)h{\mathbf{1}}_{(l-\sigma,0]}(r), r\in (-\infty,0]$.
Here and
in the sequel, for notational simplicity, we use $\mathbf{0}$ to denote elements, or fucntions  which are
identically equal to zero. 
 By the continuity of $\partial_x\hat{f},\partial_x\hat{f}'$  and the arbitrariness of $h\in {\mathbb{R}^d}$, we  have  $\partial_x\hat{f}(l,x)=\partial_x\hat{f}'(l,x)$ for every $(l,x)\in \Lambda^t$.
 Finally,  let   $\vartheta=\mathbf{0}$,  $\tau=l$ and $\varpi\equiv a\in {\cal{S}}(\mathbb{R}^d)$ in (\ref{formular1}),  by  Lemma \ref{theoremito},
 $$
                        \int^{s}_{l}\mbox{tr}(\partial_{xx}\hat{f}(\sigma,X_\sigma)aa^*)d\sigma =\int^{s}_{l}\mbox{tr}(\partial_{xx}\hat{f}'(\sigma, X_\sigma)aa^*)d\sigma, \ \ s\in [l,\infty).
$$
 By the continuity of $\partial_{xx}\hat{f},\partial_{xx}\hat{f}'$ and the arbitrariness of $a\in {\Gamma}(\mathbb{R}^d)$, we also have  $\partial_{xx}\hat{f}(l,x)=\partial_{xx}\hat{f}'(l,x)$ for every $(l,x)\in \Lambda^t$. \ \ $\Box$

\vbox{}
2.3. \emph{Borwein-Preiss variational principle}.
In this subsection we introduce a modification of  Borwein-Preiss variational principle, which
 plays a crucial role in the proof of the comparison Theorem.
We
firstly recall the definition of gauge-type function for the space $(\Lambda^t,d_\infty)$.
\begin{definition}\label{gaupe}
              Let $t\in [0,\infty)$ be fixed.  We say that a  continuous  functional $\rho:\Lambda^t\times \Lambda^t\rightarrow [0,\infty)$ is a {gauge-type function}
              under $d_\infty$
              provided that:
             \begin{description}
        \item{(i)} $\rho((s,x),(s,x))=0$ for all $(s,x)\in \Lambda^t$,
        \item{(ii)} for any $\varepsilon>0$, there exists $\delta>0$ such that, for all $(s,x), (l,y)\in \Lambda^t$, we have $\rho((s,x),(l,y))\leq \delta$ implies that
        $d_\infty((s,x),(l,y))<\varepsilon$.
        \end{description}
\end{definition}
\begin{lemma}\label{theoremleft} 
Let $t\in [0,\infty)$ be fixed and
 $f:\Lambda^t\rightarrow \mathbb{R}$ be an upper semicontinuous functional 
and  bounded from above.
Suppose that $\rho$ is a gauge-type function  under $d_\infty$
and
 $\{\delta_i\}_{i\geq0}$ is a sequence of positive number, and suppose that $\varepsilon>0$ and $(t_0,x^0)\in \Lambda^t$ satisfy
 $$
f({t_0},x^0)\geq \sup_{(s,x)\in \Lambda^t}f(s,x)-\varepsilon.
 $$
 Then there exist $(\hat{t},\hat{x})\in \Lambda^t$ and a sequence $\{(t_i,x^i)\}_{i\geq1}\subset \Lambda^t$ such that
  \begin{description}
        \item{(i)} $\rho((\hat{t},\hat{x}),({t_0},x^0))\leq \frac{\varepsilon}{\delta_0}$,  $\rho((\hat{t},\hat{x}),({t_i},x^i))\leq \frac{\varepsilon}{2^i\delta_0}$ and $t_i\uparrow \hat{t}$ as $i\rightarrow\infty$,
        \item{(ii)}  $f(\hat{t},\hat{x})-\sum_{i=0}^{\infty}\delta_i\rho((\hat{t},\hat{x}),({t_i},x^i))\geq f({t_0},x^0)$, and
        \item{(iii)}  $f(s,x)-\sum_{i=0}^{\infty}\delta_i\rho((s,x),({t_i},x^i))
            <f(\hat{t},\hat{x})-\sum_{i=0}^{\infty}\delta_i\rho((\hat{t},\hat{x}),({t_i},x^i))$ for all $(s,x)\in \Lambda^{\hat{t}}\setminus \{(\hat{t},\hat{x})\}$.

        \end{description}
\end{lemma}
The proof is completely  similar to Lemma 2.13 in our paper \cite{zhou5} (see also Theorem 2.5.2 in  Borwein \& Zhu  \cite{bor1}). Here we omit it.
\section{Smooth  gauge-type functions.}
In this section we introduce the  functionals $S_m$, which are the key to proving  the uniqueness and  stability  of viscosity solutions.
\par
 For every $m\in \mathbb{N}^+$, we define $S_m:\hat{\Lambda}\times \hat{\Lambda}\rightarrow \mathbb{R}$ as follows: for every $(t,x), (s,y)\in \hat{\Lambda}$,
 \begin{eqnarray*}
S_m(t,x,s,y)=\begin{cases}
            \frac{(|x_{(s-t)\vee0}-y_{(t-s)\vee0}|_C^{2m}-|x(0)-y(0)|^{2m})^3}{|x_{(s-t)\vee0}-y_{(t-s)\vee0}|^{4m}_{C}}, \
         ~~ |x_{(s-t)\vee0}-y_{(t-s)\vee0}|_{C}\neq0; \\
0, ~~~~~~~~~~~~~~~~~~~~~~~~~~~~~~~~~~~~~~~~~~~~~ |x_{(s-t)\vee0}-y_{(t-s)\vee0}|_{C}=0.
\end{cases}
\end{eqnarray*}
For every $M\in \mathbb{R}$, define $\Upsilon^{m,M}$ and $\overline{\Upsilon}^{m,M}$ by
\begin{eqnarray*}
         \Upsilon^{m,M}(t,x,s,y)=S_m(t,x,s,y)+M|x(0)-y(0)|^{2m}, \ \ \ (t,x), (s,y)\in \hat{\Lambda},
\end{eqnarray*}
and
\begin{eqnarray*}
         \overline{\Upsilon}^{m,M}(t,x,s,y)= \Upsilon^{m,M}(t,x,s,y)+|s-t|^2 \ \ \ (t,x), (s,y)\in \hat{\Lambda}.
\end{eqnarray*}
For notational simplicity, we let $S_m(x,y)$, $\Upsilon^{m,M}(x,y)$ and $\overline{\Upsilon}^{m,M}(x,y)$ denote $S_m(t,x,t,y)$,\\
 $\Upsilon^{m,M}(t,x,t,y)$ and $\overline{\Upsilon}^{m,M}(t,x,t,y)$ for all $(t,x,y)\in [0,\infty)\times{\cal{D}}_0\times{\cal{D}}_0$, respectively.
  Moreover, we let  $S_m(x)$ and $\Upsilon^{m,M}(x)$ denote $S_m(x,y)$ and $\Upsilon^{m,M}(x,y)$ respectively when $y(\theta)\equiv{\mathbf{0}}$ for all $\theta\in (-\infty,0]$.  we also let $S$, $\Upsilon$ and $\overline{\Upsilon}$ denote $S_3$, $\Upsilon^{3,3}$ and $\overline{\Upsilon}^{3,3}$, respectively.
\par
Now we study the   regularity  of $S_m$.
\begin{lemma}\label{theoremS}
For every fixed $(\hat{t}, a) \in \hat{\Lambda}$, define   $S_m^{\hat{t},a}:\hat{\Lambda}^{\hat{t}}\rightarrow \mathbb{R}$ by
 $$
 S_m^{\hat{t},a}(t,x):=S_m(t,x,\hat{t},a),\ \ (t,x)\in \hat{\Lambda}^{\hat{t}}.
 $$
 Then
 $S_m^{\hat{t},a}(\cdot,\cdot)\in C^{1,2}_{p}(\hat{\Lambda}^{\hat{t}})$. Moreover, for every $M\geq3$,
\begin{eqnarray}\label{s0}
                |x|_{C}^{2m}\leq   \Upsilon^{m,M}(x)\leq M|x|_{C}^{2m}, \ \ x\in {\cal{D}}_0.
\end{eqnarray}
\end{lemma}
\par
   {\bf  Proof  }. \ \
   First, by the definition of $S^{\hat{t},a}_m$, it is clear that $S_m^{\hat{t},a}(\cdot,\cdot)\in C(\hat{\Lambda}^{\hat{t}})$ and $\partial_tS_m^{\hat{t},a}(t,x)=0$ for
   $(t,x)\in \hat{\Lambda}^{\hat{t}}$.
  Second, we consider $ \partial_{x_i}S_m^{\hat{t}, a}(\cdot,\cdot)$.
  For every $x\in {\cal{D}}_0$, let $|x|_{C^-}:=\sup_{-\infty< s<0}|x(s)|$ 
   and $x_i(0):=(x(0),e_i)_{\mathbb{R}^d},\ i=1,2,\ldots, d$.  
  If $|x(0)-a(0)|<|x-a_{t-\hat{t}}|_{C^-}$,
   \begin{eqnarray}\label{s1}
   &&\partial_{x_i}S_m^{\hat{t}, a}(t,x)
   =\lim_{h\rightarrow0}\frac{S_m(t,x+{h}e_i{\mathbf{1}}_{\{0\}},{{\hat{t}}},a)-S_m(t,x,{{\hat{t}}},a)}{h}\nonumber\\
   &=&\lim_{h\rightarrow0}\frac{{(|x-a_{t-\hat{t}}|^{2m}_{C}-|x(0)+{he_i}-a(0)|^{2m})^3}
   -{(|x-a_{t-\hat{t}}|^{2m}_{C}-|x(0)-a(0)|^{2m})^3}}{h|x-a_{t-\hat{t}}|^{4m}_{C}}\nonumber\\
   &=&-\frac{6m(|x-a_{t-\hat{t}}|^{2m}_{C}-|x(0)-a(0)|^{2m})^2|x(0)-a(0)|^{2m-2}(x_i(0)-a_i(0))}{|x-a_{t-\hat{t}}|^{4m}_{C}};
   \end{eqnarray}
 if  $|x(0)-a(0)|>|x-a_{t-\hat{t}}|_{C^-}$,
\begin{eqnarray}\label{s2}
  \partial_{x_i}S_m^{\hat{t},a}({t},x)=0;
   \end{eqnarray}
if  $|x(0)-a(0)|=|x-a_{t-\hat{t}}|_{C^-}>0$,
since
\begin{eqnarray}\label{jiaxis}
&&|x-a_{t-\hat{t}}|^{2m}_{C}-|x(0)+{he_i}-a(0))|^{2m}\nonumber
\\
&=&
\begin{cases}
0,\ \ \ \ \ \ \ \ \ \  \ \ \ \ \  \ \  \ \ \ \ \ \ \ \ \ \ \ \ ~~~ ~\ \ \ \ \ \  \ \ \  \ \ \ \ \ \ \ \ \ |x(0)+he_i-a(0)|\geq |x(0)-a(0)|,\\
 |x(0)-a(0)|^{2m}-|x(0)+{he_i}-a(0)|^{2m}, \    \ |x(0)+he_i-a(0)|<|x(0)-a(0)|,\end{cases}
\end{eqnarray}
we have
\begin{eqnarray}\label{s3}
 0&\leq&\lim_{h\rightarrow0}\bigg{|}\frac{S_m^{\hat{t},a}(t,x+{h}e_i{\mathbf{1}}_{\{0\}})-S_m^{\hat{t},a}(t,x)}{h}\bigg{|}\nonumber\\
  &\leq&\lim_{h\rightarrow0}\bigg{|}\frac{(|x(0)-a(0)|^{2m}-|x(0)+{he_i}-a(0)|^{2m})^3}{h|x-a_{t-\hat{t}}+he_i\mathbf{1}_{\{0\}}|_{C}^{4m}} \bigg{|}=0;
   \end{eqnarray}
    if $|x(0)-a(0)|=|x-a_{t-\hat{t}}|_{C^-}=0$,
\begin{eqnarray}\label{ss4}
 \partial_{x_i}S_m^{\hat{t},a}({t},x)=0.
   \end{eqnarray}
From (\ref{s1}), (\ref{s2}), (\ref{s3}) and (\ref{ss4}) we obtain that
\begin{eqnarray}\label{0528a}
    \partial_{x_i}S_m^{\hat{t},a}(t,x)=\begin{cases}-\frac{6m(|x-a_{t-\hat{t}}|^{2m}_{C}-|x(0)-a(0)|^{2m})^2|x(0)-a(0)|^{2m-2}(x_i(0)-a_i(0))}{|x-a_{t-\hat{t}}|^{4m}_{C}}, \  \ \  |x-a_{t-\hat{t}}|_{C}\neq0,\\
    0, ~~~~~~~~~~~~~~~~~~~~~~~~~~~~~~~~~~~~~~~~~~~~~~~~~~~~~~~~~~~~~~~~~~~~~~|x-a_{t-\hat{t}}|_{C}=0.
    \end{cases}
\end{eqnarray}
It is clear that $\partial_{x_i}S_m^{\hat{t},a}(\cdot,\cdot)\in C(\hat{\Lambda}^{\hat{t}})$.
\par
We now consider $\partial_{x_jx_i}S_m^{\hat{t},a}(\cdot,\cdot)$.
If $|x(0)-a(0)|<|x-a_{t-\hat{t}}|_{C^-}$,
\begin{eqnarray}\label{s5}
   &&\partial_{x_jx_i}S_m^{\hat{t},a}(t,x)\nonumber\\
   &=&\lim_{h\rightarrow0}\bigg{[}\frac{{-6m(|x-a_{t-\hat{t}}|_{C}^{2m}-|x(0)+he_j-a(0)|^{2m})^2|x(0)+he_j-a(0)|^{2m-2}}
   }{h{|x-a_{t-\hat{t}}|_{C}^{4m}}}\nonumber\\
    &&~~~~~~~\times
   (x_i(0)-a_i(0)+h\mathbf{1}_{\{i=j\}})\nonumber\\
   &&~~~~~~~+\frac{
   {6m(|x-a_{t-\hat{t}}|_{C}^{2m}-|x(0)-a(0)|^{2m})^2|x(0)-a(0)|^{2m-2}(x_i(0)-a_i(0))}}{h{|x-a_{t-\hat{t}}|_{C}^{4m}}}\bigg{]}\nonumber\\
   &=&\frac{24m^2(|x-a_{t-\hat{t}}|_{C}^{2m}-|x(0)-a(0)|^{2m})|x(0)-a(0)|^{4m-4}(x_i(0)-a_i(0))(x_j(0)-a_j(0))}
   {{|x-a_{t-\hat{t}}|_{C}^{4m}}}\nonumber\\
   &&-\frac{12m(m-1)(|x-a_{t-\hat{t}}|_{C}^{2m}-|x(0)-a(0)|^{2m})^2|x(0)-a(0)|^{2m-4}
   (x_i(0)-a_i(0))}{{|x-a_{t-\hat{t}}|_{C}^{4m}}}\nonumber\\
   &&\times (x_j(0)-a_j(0))\nonumber\\
   &&-\frac{6m(|x-a_{t-\hat{t}}|_{C}^{2m}-|x(0)-a(0)|^{2m})^2|x(0)-a(0)|^{2m-2}{\mathbf{1}}_{\{i=j\}}}{{|x-a_{t-\hat{t}}|_{C}^{4m}}};
   \end{eqnarray}
   if $|x(0)-a(0)|>|x-a_{t-\hat{t}}|_{C^-}$,
\begin{eqnarray}\label{s4}
   \partial_{x_jx_i}S_m^{\hat{t},a}(t,x)=0;
   \end{eqnarray}
 if $|x(0)-a(0)|=|x-a_{t-\hat{t}}|_{C^-}>0$, by (\ref{jiaxis}),
we have
\begin{eqnarray}\label{s6666}
 0&\leq&\lim_{h\rightarrow0}\bigg{|}\frac{\partial_{x_i}S_m^{\hat{t},a}(t,x+{h}e_j{\mathbf{1}}_{\{0\}})-\partial_{x_i}S_m^{\hat{t},a}(t,x)}{h}\bigg{|}\nonumber\\
  &\leq&\lim_{h\rightarrow0}6m\bigg{|}\frac{(|x(0)-a(0)|^{2m}-|x(0)+{he_i}-a(0)|^{2m})^2|x(0)+{he_j}-a(0)|^{2m-2}}{h|x-a_{t-\hat{t}}+he_j\mathbf{1}_{\{0\}}|_{C}^{4m}} \nonumber\\
  &&\times (x_i(0)-a_i(0)+h{\mathbf{1}}_{\{i=j\}})\bigg{|}=0;
   \end{eqnarray}
    if $|x(0)-a(0)|=|x-a_{t-\hat{t}}|_{C^-}=0$,
\begin{eqnarray}\label{ss42}
\partial_{x_jx_i}S_m^{\hat{t},a}(t,x)=0.
   \end{eqnarray}
%
%
   Combining (\ref{s5}), (\ref{s4}), (\ref{s6666}) and (\ref{ss42}) we obtain
   \begin{eqnarray}\label{0528b}
    \partial_{x_jx_i}S_m^{\hat{t},a}(t,x)=\begin{cases}\frac{24m^2(|x-a_{t-\hat{t}}|_{C}^{2m}-|x(0)-a(0)|^{2m})|x(0)-a(0)|^{4m-4}(x_i(0)-a_i(0))(x_j(0)-a_j(0))}
   {{|x-a_{t-\hat{t}}|_{C}^{4m}}}\\
   -\frac{12m(m-1)(|x-a_{t-\hat{t}}|_{C}^{2m}-|x(0)-a(0)|^{2m})^2|x(0)-a(0)|^{2m-4}
   (x_i(0)-a_i(0))(x_j(0)-a_j(0))}{{|x-a_{t-\hat{t}}|_{C}^{4m}}}\\
   -\frac{6m(|x-a_{t-\hat{t}}|_{C}^{2m}-|x(0)-a(0)|^{2m})^2|x(0)-a(0)|^{2m-2}{\mathbf{1}}_{\{i=j\}}}{{|x-a_{t-\hat{t}}|_{C}^{4m}}}, ~~  |x-a_{t-\hat{t}}|_{C}\neq0,\\
    0, ~~~~~~~~~~~~~~~~~~~~~~~~~~~~~~~~~~~~~~~~~~~~~~~~~~~~~~~~~~~~~~|x-a_{t-\hat{t}}|_{C}=0.
    \end{cases}
\end{eqnarray}
It is clear that $\partial_{x_jx_i}S_m^{\hat{t},a}(\cdot,\cdot)\in C(\hat{\Lambda}^{\hat{t}})$.
  By  simple calculation, we can see that  $S_m^{\hat{t},a}(\cdot,\cdot)$ and all of its derivatives grow  in a polynomial way.
 Thus, we have show that $S_m^{\hat{t},a}(\cdot,\cdot)\in C^{1,2}_{p}(\hat{\Lambda}^{\hat{t}})$.
 %
%
%
\par
Now we prove (\ref{s0}).  If $|x|_{C}=0$, it is clear that (\ref{s0}) holds. Then we may assume that
       $|x|_{C}\neq0$.
Letting $\alpha:=|x(0)|^{2m}$, we have
 \begin{eqnarray*}
    \Upsilon^{m,M}(x)=\frac{(|x|_C^{2m}-|x(0)|^{2m})^3}{|x|_C^{4m}}+M|x(0)|^{2m}:=f(\alpha)=\frac{(|x|_C^{2m}-\alpha)^3}{|x|_C^{4m}}+M\alpha.
\end{eqnarray*}
By $$
              f'(\alpha)=-3\frac{(|x|_C^{2m}-\alpha)^2}{|x|_C^{4m}}+M\geq0, \ \ \ M\geq3, \ \  0\leq \alpha\leq |x|_C^{2m},
$$
we get that
$$
                                 |x|_C^{2m} =f(0)\leq\Upsilon^{m,M}(x)=f(\alpha)\leq f(|x|_C^{2m})=M|x|_C^{2m},\ \ \  M\geq3, \ \ x\in {\cal{D}}_0.
$$
Thus, we  have (\ref{s0}) holds true.
 The proof is now complete. \ \ $\Box$
 \begin{remark}\label{remarks}
 \begin{description}
   \item{(i)}   
 For every fixed $m\in {\mathbb{N}}^+$ and $a_0 \in \mathbb{R}^d$, define   $f:{\hat{\Lambda}}\rightarrow \mathbb{R}$ by
 $$f(t,x):=|x(0)-a_0|^{2m},\ \ (t,x)\in {\hat{\Lambda}}.$$
 Notice that
\begin{eqnarray}\label{0612a}
\partial_tf(t,x)=0;
\end{eqnarray}
\begin{eqnarray}\label{0612b}
\partial_{x}f(t,x)=2m|x(0)-a_0|^{2m-2}(x(0)-a_0);
\end{eqnarray}
\begin{eqnarray}\label{0612c}
\partial_{xx}f(t,x)=2m|x(0)-a_0|^{2m-2}I+4m(m-1)|x(0)-a_0|^{2m-4}(x(0)-a_0)(x(0)-a_0)^\top. \
\end{eqnarray}
Then $f\in C_p^{1,2}(\hat{\Lambda})$, and  by the above lemma,  $\Upsilon^{m,M}(\cdot,\cdot,\hat{t},{{a}})\in C_p^{1,2}(\hat{\Lambda}^{\hat{t}})$ for all $m\in {\mathbb{N}}^+$, $M\in {\mathbb{R}}$
and $(\hat{t}, {{a}}) \in \hat{\Lambda}_{\hat{t}}$.
\item{(ii)}      Since $|\cdot|_{C}^6$ does not belong to $C^{1,2}_p(\hat{\Lambda})$, then, for every $(\hat{t},a)\in \Lambda$, $|x-a_{t-\hat{t}}|_C^6$ cannot  appear  as an auxiliary functional in the proof of the
    uniqueness and stability of viscosity solutions. However, by the above lemma and (i) of this remark, we can replace $|x-a_{t-\hat{t}}|_C^6$
   with its equivalent functional  $\Upsilon(t,x,\hat{t},a)$.
 \item{(iii)}
  It follows from (\ref{s0}) that, for all $(l,x),(s,y)\in \hat{\Lambda}$,
 \begin{eqnarray}\label{0612d}
         \overline{\Upsilon}(\gamma_t,\eta_s)={\Upsilon}(\gamma_{t,t\vee s}-\eta_{s,t\vee s})+|s-t|^2\geq |x_{(s-l)\vee0}-y_{(l-s)\vee0}|^6_C+|s-t|^2.
\end{eqnarray}
 Thus $\overline{\Upsilon}$ is a gauge-type function.  We can apply it to Lemma \ref{theoremleft} to
get a maximum of a perturbation of the auxiliary function in the proof of uniqueness.
 %
%
%
\item{(iv)}
        $\overline{\Upsilon}$ is not a  gauge-type function under usual norm $|\cdot|_1$, where $|(t,x)|_1=|t|+|x|_C$ for every $(t,x)\in \Lambda$. As a matter of fact, consider the following example. Take $(t_n,x^n), (s_n,y^n)\in\Lambda$ as follows:
       $$
             t_n=0, \ s_n=\frac{1}{n}, \  \ x^n(\theta)=(1+n\theta)\mathbf{1}_{[-\frac{1}{n},0]}(\theta), \ \ y^n(\theta)=[(2+n\theta)\wedge 1]\mathbf{1}_{[-\frac{2}{n},0]}(\theta),\  \ \theta\in (-\infty,0].
       $$
       It is clear that $\overline{\Upsilon}(t_n,x^n,s_n,y^n)=\frac{1}{n}\rightarrow0$ as $n\rightarrow\infty$. However,
       $$
       |(t_n,x^n)-(s_n,y^n)|_1=|t_n-s_n|+|x^n-y^n|_C=\frac{1}{n}+1\geq1.
       $$
       This shows that $\overline{\Upsilon}$ does not satisfy  condition (ii) of Definition \ref{gaupe} with  norm $|\cdot|_1$.
\end{description}
\end{remark}
\par
In the proof of uniqueness of viscosity solutions, in order to apply Theorem 8.3 in \cite{cran2}, we also need the following lemma.
Its proof is completely similar to the path-dependent  case (see Lemma 3.3 in \cite{zhou5}). Here we omit it.
\begin{lemma}\label{theoremS000} For $m\in \mathbb{N}^+$ and $M\geq3$, we have
\begin{eqnarray}\label{up}
\Upsilon^{m,M}(x+y)\leq 2^{2m-1}( \Upsilon^{m,M}(x)+ \Upsilon^{m,M}(y)) \ \ x,y\in {\cal{D}}_0.
\end{eqnarray}
\end{lemma}
\section{ A DPP for optimal control problems.}

\par
In this section, we consider the controlled state
             equation (\ref{state1}) and cost functional  (\ref{cost1}).
We introduce the admissible control. Let $t$ be a deterministic time, $0\leq t<\infty$.
\begin{definition}
                An admissible control process $u(\cdot)=\{u(r),  r\in [t,\infty)\}$ on $[t,\infty)$  is an ${\cal{F}}_t^r$-progressing measurable process taking values in some Polish space $(U,d)$. The set of all admissible controls on $[t,\infty)$ is denoted by ${\cal{U}}_t$. We identify two processes $u(\cdot)$ and $\tilde{u}(\cdot)$ in ${\cal{U}}_t$
                and write $u(\cdot)\equiv\tilde{u}(\cdot)$ on $[t,\infty)$, if $\mathbb{P}(u(\cdot)=\tilde{u}(\cdot) \ a.e. \ \mbox{in}\ [t,\infty))=1$.
\end{definition}
 We make the following assumption.
\begin{hyp}\label{hypstate}
$ b: {\cal{C}}_0 \times U\rightarrow \mathbb{R}^d$, $\sigma: {\cal{C}}_0 \times U\rightarrow \mathbb{R}^{d\times n}$ and $
        q: {\cal{C}}_0\times U\rightarrow \mathbb{R}$  are continuous, and
                  there exists a  constant
                   $L>0$ 
                   such that, for all $(x,u)$,  $ (y,u) \in {\cal{C}}_0\times U$,
      \begin{eqnarray}\label{1207}
                &&|b(x,u)|^2\vee|\sigma(x,u)|_2^2\vee|q(x,u)|^2\leq
                 L^2(1+|x|_C^2),\\
                 &&|b(x,u)-b(y,u)|\vee|\sigma(x,u)-\sigma(y,u)|_2\vee|q(x,u)-q(y,u)|\leq
                 L|x-y|_C.\nonumber
\end{eqnarray}
\end{hyp}
For given $t\in[0,\infty)$, ${\cal{F}}_t$-measurable map $\xi:\Omega\rightarrow {\cal{C}}_0$ and admissible control $u(\cdot)\in {\cal{U}}_0$, consider the following stochastic differential equation (SDE) with infinite delay:
\begin{eqnarray}\label{state}
            \begin{cases}{dX^{t,\xi,u}(s)}=
            b{(}X^{t,\xi,u}_s,u(s){)}ds+\sigma(X^{t,\xi,u}_s,u(s))dW(s), \  s\in [t,\infty),\\
            ~~~~~X^{t,\xi,u}_t=\xi.
           \end{cases}
\end{eqnarray}
We first recall a result on the solvability of (\ref{state}) on a bounded interval.
\begin{lemma}\label{lemmaexist}
\ \ Take $p\geq2$ and assume that Hypothesis \ref{hypstate}  holds. Then for every $T>0$ and 
$(t,\xi,u(\cdot))\in [0,T)\times L^p(\Omega,{\cal{F}}_t;{\cal{C}}_0)\times {\cal{U}}_0$, equation (\ref{state}) admits a
unique strong  solution $X^{t,\xi,u}(\cdot)$ on $[t,T]$. 
 Furthermore, let  $X^{t,\xi',u}(\cdot)$  be the solution of equation(\ref{state})
 corresponding $(t,\xi',u(\cdot))\in [0,T)\times L^p(\Omega,{\cal{F}}_t;{\cal{C}}_0)\times {\cal{U}}_0$. Then the following estimates hold:
 \begin{eqnarray}\label{fbjia1}
               \mathbb{E}\left[\sup_{t\leq s\leq T}|X^{t,\xi,u}(s)-X^{t,\xi',u}(s)|^p\right]\leq C_p\mathbb{E}|\xi-\xi'|_C^p;
\end{eqnarray}
\begin{eqnarray}\label{fbjia2}
               \mathbb{E}\left[\sup_{t\leq s\leq T}|X^{t,\xi,u}(s)|^p\right]\leq C_p(1+\mathbb{E}|\xi|_C^p).
                \end{eqnarray}
             The constant $C_p$ depending only on  $p$, $T$ and $L$.
\end{lemma}
{\bf  Proof}. \ \
                   In the case  of $p=2$, we refer to Theorem 3.1 in  \cite{wei} for existence and uniqueness of the solution to equation (\ref{state}) on $[t,T]$ and to Lemma 3.2 in   \cite{wei} for the Linear growth
                    of this solution on the initial datum. By the similar proving process of Lemma 3.2 in   \cite{wei}, we obtain (\ref{fbjia1}). 
                  The proof in the case of $p>2$ can be performed in a similar way.
                   \ \ $\Box$\\
                   Now we study equation (\ref{state}) and consider certain continuities for the solution $X^{t,\xi,u}$.
                      These
properties will be used in the proof of Theorem \ref{theoremv}.
                  \begin{theorem}\label{theoremexists1013} Take $p\geq2$ and suppose that Hypothesis \ref{hypstate} is satisfied,
                    then for every $(t,\xi,u(\cdot))\in [0,\infty)\times L^p(\Omega,{\cal{F}}_t;{\cal{C}}_0)\times {\cal{U}}_0$ and $\beta<-(\frac{5}{2}L^2+L)$,
                  equation (\ref{state}) has a unique  solution  $X^{t,\xi,u}(\cdot)$. Moreover, if we let  $X^{t,\eta,v}(\cdot)$  be the solutions of  (\ref{state})
 corresponding $(t,\eta,v(\cdot))\in [0,\infty)\times L^{p}(\Omega,{\cal{F}}_t;{\cal{C}}_0)\times {\cal{U}}_0$. Then the following estimates hold:
\begin{eqnarray}\label{linear}
\sup_{s\geq t}e^{2\beta s}\mathbb{E}\left[|X^{t,\xi,u}_s|_C^2|{\cal{F}}_t\right]+\int_t^{\infty}e^{2\beta l}\mathbb{E}[|X^{t,\xi,u}_l|_C^2|{\cal{F}}_t]dl
\le C(1+|\xi|_C^2);
\end{eqnarray}
and
\begin{eqnarray}\label{lipsss}
&&\sup_{s\geq t}e^{2\beta s}\mathbb{E}\left[|X^{t,\xi,u}_s-X^{t,\eta,v}_s|_C^2|{\cal{F}}_t\right]+\int_t^{\infty}e^{2\beta l}\mathbb{E}[|X^{t,\xi,u}_l-X^{t,\eta,v}_l|_C^2|{\cal{F}}_t]dl\nonumber\\
&\leq& C|\xi-\eta|_C^2+C\int^{\infty}_{t}e^{2\beta l}\mathbb{E}[|b(X^{t,\eta,v}_l,u(l))-b(X^{t,\eta,v}_l,v(l))|^2|{\cal{F}}_t]dl\nonumber\\
&&+C\int^{\infty}_{t}e^{2\beta l}\mathbb{E}[|\sigma(X^{t,\eta,v}_l,u(l))-\sigma(X^{t,\eta,v}_l,v(l))|_2^2|{\cal{F}}_t]dl.
\end{eqnarray}
The constant $C$ depends only on $\beta$ and $L$. Moreover, for all $s\geq t$, there exists some constant $C_0>0$ depending  only on $\beta$ and $L$ such that
\begin{eqnarray}\label{12071}
                   \mathbb{E}\left|X^{t,\xi,u}_s-\xi_{s-t}\right|_C^2
                   \leq C_0(1+\mathbb{E}|\xi|_C^2)e^{-2\beta s}((s-t)+1)(s-t).
\end{eqnarray}
\end{theorem}
\begin{proof}
Existence and uniqueness are satisfied by Lemma \ref{lemmaexist}. We only need to prove that (\ref{linear}), (\ref{lipsss}) and (\ref{12071})
hold true.
	Applying It\^{o} formula  to $3e^{2\beta l}|X^{t,\xi,u}(l)|^2$ and functional It\^{o} formula (\ref{statesop0}) to $e^{2\beta l}S_{1}(X^{t,\xi,u}_l)$ on $[t,s]$, we obtain
	\begin{flalign*}
	&\quad e^{2\beta s}\Upsilon^{1,3}(X^{t,\xi,u}_s)-e^{2\beta t}\Upsilon^{1,3}(\xi)\\
	&= 2\beta\int_t^se^{2\beta l}\Upsilon^{1,3}(X^{t,\xi,u}_l)dl+\int_t^se^{2\beta l}(\partial_x\Upsilon^{1,3}(X^{t,\xi,u}_l),b(X^{t,\xi,u}_l,u(l)))_{\mathbb{R}^d} dl\\
	&\quad+\int_t^se^{2\beta l}(\partial_x\Upsilon^{1,3}(X^{t,\xi,u}_l),\sigma(X^{t,\xi,u}_l,u(l))dW(l))_{\mathbb{R}^d}\\
&\quad+\frac{1}{2}\int_t^se^{2\beta l}\mbox{tr}[\partial_{xx}\Upsilon^{1,3}(X^{t,\xi,u}_l)\sigma(X^{t,\xi,u}_l,u(l))\sigma^\top(X^{t,\xi,u}_l,u(l))]dl.
	\end{flalign*}
Taking conditional expectation with respect to ${\cal{F}}_t$, by  Hypothesis \ref{hypstate} and Lemma \ref{theoremS}, for every $\beta<0$ and $\varepsilon>0$,
\begin{flalign*}
	&\quad e^{2\beta s}\mathbb{E}\left[|X^{t,\xi,u}_s|_C^2|{\cal{F}}_t\right]\leq e^{2\beta s}\mathbb{E}\left[\Upsilon^{1,3}(X^{t,\xi,u}_s)|_C^2|{\cal{F}}_t\right]\\
	&\leq 3e^{2\beta t}|\xi|_C^2+6\beta\int_t^se^{2\beta l}\mathbb{E}[|X^{t,\xi,u}_l|_C^2|{\cal{F}}_t]dl\\
&\quad+6L\int_t^se^{2\beta l}\mathbb{E}[(|X^{t,\xi,u}_l|_C+|X^{t,\xi,u}_l|_C^2)|{\cal{F}}_t] dl
+15L^2\int_t^se^{2\beta l}(1+\mathbb{E}[|X^{t,\xi,u}_l|_C^2|{\cal{F}}_t])dl\\
&\leq 3|\xi|_C^2+(15L^2+6L+6\beta+\varepsilon)\int_t^se^{2\beta l}\mathbb{E}[|X^{t,\xi,u}_l|_C^2|{\cal{F}}_t]dl-\frac{9L^2}{2\beta\varepsilon}-\frac{15L^2}{2\beta}.
	\end{flalign*}
For every $\beta<-(\frac{5}{2}L^2+L)$, we can let $\varepsilon$ be small enough such that $15L^2+6L+6\beta+\varepsilon<0$, then there exists $C>0$ depending only on $\beta, L$ such that
$$
e^{2\beta s}\mathbb{E}\left[|X^{t,\xi,u}_s|_C^2|{\cal{F}}_t\right]+\int_t^se^{2\beta l}\mathbb{E}[|X^{t,\xi,u}_l|_C^2|{\cal{F}}_t]dl
\le C(1+|\xi|_C^2).
$$
Taking the supremum over  $s\in [t,\infty)$, we obtain (\ref{linear}). By the similar  process, we can show (\ref{lipsss})  holds true.
Now let us prove (\ref{12071}). 
By (\ref{1207}) and (\ref{linear}), we
                obtain the following result:
\begin{eqnarray*}
                   \mathbb{E}\left|X^{t,\xi,u}_s-\xi_{s-t}\right|_C^2
                   &\leq&2L^2\mathbb{E}\left(\int_t^s(1+|X^{t,\xi,u}_l|_C) dl\right)^2+8L^2\mathbb{E}\int_t^s(1+|X^{t,\xi,u}_l|_C^2) dl\\
               &\leq& 4L^2((s-t)+2)(s-t)(1+\mathbb{E}|X^{t,\xi,u}_s|_C^2)\\
               &\leq& 4L^2((s-t)+2)(s-t)(1+Ce^{-2\beta s}(1+\mathbb{E}|\xi|_C^2)).
\end{eqnarray*}
Then there exists a suitable constant $C_0>0$ depending  only on $\beta$ and $L$ such that (\ref{12071}) holds true.
 The proof is now complete.
\end{proof}
For the particular case of a deterministic $\xi$, i.e. $\xi=x\in {\cal{C}}_0$, we let $X^{t,x,u}(\cdot)$ denote the solution of equation (\ref{state}) corresponding $(t,x,u(\cdot))\in [0,\infty)\times {\cal{C}}_0\times {\cal{U}}_0$.
 For simplicity, if $t=0$ we denote $X^{0,x,u}(\cdot)$ by $X^{x,u}(\cdot)$.
Our first result for the value functional defined in (\ref{eq3}) includes   the local boundedness  and
                the
                continuity.
\begin{theorem}\label{theoremv}  Suppose that Hypothesis \ref{hypstate}   holds. Then, for every $\lambda> \Theta:=\frac{5}{2}L^2+L$, there exists a constant $C_1>0$  such that
\begin{eqnarray}\label{valuep1}
                         |V(x)|\leq C_1(1+|x|_{{C}}), \ \ x, y\in {\cal{C}}_0,
\end{eqnarray}
and
\begin{eqnarray}\label{valuep2}
                        |V(x)-V(y)|
                        \leq
                        C_1|x-y|_{{C}}, \  \ x, y\in {\cal{C}}_0.
\end{eqnarray}
\end{theorem}
{\bf  Proof}. \ \
                   From the definition of $J$ and from Hypothesis \ref{hypstate}, we know that, for all $u(\cdot)\in {\cal{U}}_0$,
\begin{eqnarray*}
                        |J(x,u(\cdot))-J(y,u(\cdot))|&\leq&\int_{0}^{\infty}e^{-\lambda l}{\mathbb{E}}|q(X^{x,u}_l,u(l))-q(X^{y,u}_l,u(l))|dl\\
                        &\leq&L\int_{0}^{\infty}e^{-\lambda
                        l}\mathbb{E}|X^{x,u}_l-X^{y,u}_l|_Cdl.
\end{eqnarray*}
                   According to   Theorem \ref{theoremexists1013}, we obtain that for
                   a constant $C_1>0$,
\begin{eqnarray*}
                        |V(x)-V(y)|\leq \sup_{u(\cdot)\in {\cal{U}}_0}|J(x,u(\cdot))-J(y,u(\cdot))|
                        \leq C_1|x-y|_C.
\end{eqnarray*}
 By a similar procedure, we can show
                   that (\ref{valuep1})  holds true, which completes the proof. \ \ $\Box$
\par
                        Now we present the following result, which is called the  dynamic programming principle  (DPP) for  optimal
                       control problems (\ref{state1}) and (\ref{cost1}).
   \begin{theorem}\label{theoremvalue}
                      Assume that Hypothesis \ref{hypstate}   holds true. Then, for every   $\lambda> \Theta$, we know that
\begin{eqnarray}\label{dpp}
             V(x)=\inf_{u(\cdot)\in {\cal{U}}_0}\bigg{[}\int_{0}^{t}e^{-\lambda l}\mathbb{E}q(X^{x,u}_l,u(l))dl
                   +e^{-\lambda t}\mathbb{E}V(X^{x,u}_t)\bigg{]}, \ \ (t,x)\in \Lambda.
\end{eqnarray}
\end{theorem}
The proof is completely  similar to Theorem 2.31  in Fabbri, Gozzi and \'{S}wi\c{e}ch \cite{fab1}. The only difference being
that here the value functional $V$ is defined on a  separable metric space $({\cal{C}}_0,|\cdot|_C)$, while  the latter is defined on a  separable Hilbert space. Here we omit it.

             In Theorem \ref{theoremv} we have already seen that the
                       value functional $V(x)$ is Lipschitz
                       continuous in $x$. With the
                       help of Theorem \ref{theoremvalue}  now show that another  continuity
                       property of $V(x)$.
\begin{theorem}\label{theorem3.9}
                          Under Hypotheses \ref{hypstate},  for every $\delta\in [0,\infty), x\in {\cal{C}}_0$ and  $\lambda> \Theta$,
there is a constant $C'>0$ such that 
\begin{eqnarray}\label{hold}
                 |V(x)-V(x_\delta)|\leq
                                    C'(1+|x|_C)(\delta+\delta^{\frac{1}{2}}+1-e^{-\lambda \delta}).
\end{eqnarray}
\end{theorem}
{\bf  Proof}. \ \
                 Let $x\in{\cal{C}}_0$ and $\delta\geq0$ be
                 arbitrarily given.
                 From Theorem \ref{theoremvalue} it follows that for any
                 $\varepsilon>0$ there exists an admissible control
                 $u^\varepsilon(\cdot)\in{\mathcal {U}}_0$ such that
                 $$
                 V(x)\geq\int_{0}^{\delta}e^{-\lambda l}\mathbb{E}q(X^{x,u^\varepsilon}_l,u^\varepsilon(l))dl
                   +e^{-\lambda \delta}\mathbb{E}V(X^{x,u^\varepsilon}_\delta)-\varepsilon.
                 $$
Therefore,
\begin{eqnarray}\label{3.20}
                          |V(x)-V(x_\delta)|\leq
                          |I^1_{\delta}|+|I^2_{\delta}|+\varepsilon,
\end{eqnarray}
                        where
$$
                             I^1_{\delta}=\int_{0}^{\delta}e^{-\lambda l}\mathbb{E}q(X^{x,u^\varepsilon}_l,u^\varepsilon(l))dl,
$$
$$
                                  I^2_{\delta}=e^{-\lambda \delta}\mathbb{E}V(X^{x,u^\varepsilon}_\delta)-
                                                V(x_\delta).
$$
                                 By (\ref{linear}), we
                                 have that
\begin{eqnarray*}
                                |I^1_{\delta}|&\leq&L\delta(1+\sup_{0\leq l\leq \delta}e^{-\lambda l}\mathbb{E}|X^{x,u^\varepsilon}_l|_C)\leq L\delta(1+C^{\frac{1}{2}}(1+|x|_C)).
\end{eqnarray*}
         Now we consider
         second term $I^2_{\delta}$. By (\ref{12071}) and Theorem \ref{theoremv},
\begin{eqnarray*}
                                   |I^2_{\delta}|&\leq& e^{-\lambda \delta}|\mathbb{E}V(X^{x,u^\varepsilon}_\delta)-
                                                V(x_\delta)|+|(e^{-\lambda \delta}-1)V(x_\delta)|\\
                                   &\leq& C_1e^{-\lambda \delta}\mathbb{E}|X^{x,u^\varepsilon}_\delta-x_\delta|_C+C_1(1-e^{-\lambda \delta})(1+|x|_C)\\
                                    &\leq& C_1C_0^{\frac{1}{2}}(1+|x|_C)(1+\delta)^{\frac{1}{2}}\delta^{\frac{1}{2}}+C_1(1-e^{-\lambda \delta})(1+|x|_C).
\end{eqnarray*}
                                   Hence, from (\ref{3.20}), we then have,  for some suitable constant
                                 $C'$ independent of the control
                                 $u^\varepsilon$,
$$
                                   |V(x)-V(x_\delta)|\leq
                                   C'(1+|x|_C)(\delta+\delta^{\frac{1}{2}}+1-e^{-\lambda \delta})
                                   +\varepsilon,
$$
                                   and letting $\varepsilon\downarrow
                                   0$ we get (\ref{hold}) holds true.
                                  The proof is complete.\ \
                                  $\square$

\section{ Viscosity solutions to  HJB equations: Existence theorem.}

\par
                                   In this section, we consider second order elliptic HJB  equation with infinite delay (\ref{hjb1}). As usual, we start with classical solutions.

\par
\begin{definition}\label{definitionccc}     (Classical solution)
              A functional $v\in C_p^{1,2}({\cal{C}}_0)$       is called a classical solution to   equation (\ref{hjb1}) if it satisfies
                equation (\ref{hjb1}) pointwise.
 \end{definition}
                      \par
        We will prove that the value functional $V$ defined by (\ref{eq3}) is a viscosity solution of  equation (\ref{hjb1}).
                     We  give the following definition for the viscosity solutions.
 For every $(t,x)\in \Lambda$ and $w\in C({\cal{C}}_0)$, define
$$
             \mathcal{A}^+(t,x,w):=\bigg{\{}\varphi\in C_p^{1,2}(\Lambda^t):  0={w}(x)-\varphi({t,x})=\sup_{(s,y)\in \Lambda^t}
                         ({w}(y)- \varphi(s,y
                         ))\bigg{\}},
$$
and
$$
             \mathcal{A}^-(t,x,w):=\left\{\varphi\in C_p^{1,2}(\Lambda^t):  0={w}(x)+\varphi({t,x})=\inf_{(s,y)\in \Lambda^t}
                         ({w}(y)+\varphi(
                         s,y))\right\}.
$$
\begin{definition}\label{definition4.1} \ \
 $w\in C({\cal{C}}_0)$ is called a
                             viscosity subsolution (resp.,  supersolution)
                             to  equation (\ref{hjb1}) if whenever  $\varphi\in {\cal{A}}^+(s,x,w)$ (resp.,  $\varphi\in {\cal{A}}^-(s,x,w)$)  with $(s,x)\in \Lambda$,  we have
  %
%
%
\begin{eqnarray*}
                          -\lambda w(x)+\partial_t{\varphi}(s,x)
                           +{\mathbf{H}}(x,\partial_x\varphi(s,x),\partial_{xx}\varphi(s,x))\geq0,
\end{eqnarray*}
\begin{eqnarray*}
                          (\mbox{resp.},\ -\lambda w(x)-\partial_t{\varphi}(s,x)
                           +{\mathbf{H}}(x,-\partial_x\varphi(s,x),-\partial_{xx}\varphi(s,x))
                          \leq0).
\end{eqnarray*}
                                $w\in C({\cal{C}}_0)$ is said to be a
                             viscosity solution to equation (\ref{hjb1}) if it is
                             both a viscosity subsolution and a viscosity
                             supersolution.
\end{definition}
\begin{remark}\label{remarkv}
  Assume that   coefficients $b(x,u)=\overline{b}(x(0),u),
                     \ \sigma(x,u)=\overline{\sigma}(x(0),u)$,  $ q(x,u)=\overline{q}(x(0),u)$ for all
                     $(x,u) \in {\cal{C}}_0\times U$. Then there exists a function $\overline{V}: \mathbb{R}^d\rightarrow \mathbb{R}$ such that  $V(x)=\overline{V}(x(0))$ for all
                     $x \in {\cal{C}}_0$, and    equation (\ref{hjb1}) reduces to the following HJB equation:
 \begin{eqnarray}\label{hjb3}
-\lambda\overline{V}(x)+\overline{{\mathbf{H}}}(x,\nabla_x\overline{V}(x),\nabla_{xx}\overline{V}(x))= 0,\ \ \  x\in
                                \mathbb{R}^d,
\end{eqnarray}
where
 \begin{eqnarray*}
                                \overline{{\mathbf{H}}}(x,p,l)=\sup_{u\in{
                                         {U}}}[(p,\overline{b}(x,u))_{\mathbb{R}^d} +\frac{1}{2}\mbox{tr}[ l \sigma(x,u)\sigma^\top(x,u)]+\bar{q}(x,u)],  \ (x,p,l)\in \mathbb{R}^d\times  \mathbb{R}^d\times {\cal{S}}(\mathbb{R}^d).
\end{eqnarray*}
Here  and in the sequel, 
$\nabla_x$ and $\nabla_{xx}$ denote the standard  first and second order derivatives
with respect to $x$. 
 %
%
%
%
\end{remark}
\par
 The following theorem show that our definition of viscosity solutions to  equation (\ref{hjb1}) is a natural  extension of classical viscosity solution to equation  (\ref{hjb3}).
 \begin{theorem}\label{theoremnatural} \ \ Consider the setting in Remark \ref{remarkv}.
                          Assume that $V$ is a viscosity solution of   equation (\ref{hjb1}) in the sense of Definition \ref{definition4.1}. Then $\overline{V}$ is a
                          viscosity solution of equation (\ref{hjb3}) in the standard sense (see Definition 2.2 in  \cite{cran2}).
\end{theorem}
 {\bf  Proof}. \ \ Without loss of generality, we shall only prove the viscosity subsolution property.
Let $\overline{\varphi}\in C^{2}(\mathbb{R}^d)$ and $x\in  \mathbb{R}^d$
 such that $$
                         0=(\overline{V}- {\overline{\varphi}})(x)=\sup_{y\in \mathbb{R}^d}
                         (\overline{V}- {\overline{\varphi}})(y).
$$
We can modify $\overline{\varphi}$  such that $\overline{\varphi}$, $\nabla_{x}\overline{\varphi}$ and $\nabla_{xx}\overline{\varphi}$  grow  in a polynomial way.
 Define $\varphi:\hat{\Lambda}\rightarrow \mathbb{R}$ by
 $$
 \varphi(s,\gamma)=\overline{\varphi}(\gamma(0)),\  (s, \gamma)\in \hat{\Lambda},
 $$
 and define $\hat{\gamma}\in {\cal{C}}_0$ by
 $$
   \hat{\gamma}(\theta)= e^{\theta}x,\ \ \theta\in (-\infty,0].$$
 It is clear that,
 $$\partial_t\varphi(s,\gamma)=0,\ \ \partial_x\varphi(s,\gamma)=\nabla_x\overline{\varphi}(s,\gamma(0)),\ \
 \partial_{xx}\varphi(s,\gamma)=\nabla_{xx}\overline{\varphi}(s,\gamma(0)),\ \ (s,\gamma)\in \hat{\Lambda}.
 $$
 Thus we have  $\varphi\in C^{1,2}_p(\Lambda)$.
  Moreover, by the definitions of $V$ and $\varphi$, for every fixed $t\geq0$,
 $$
 0=V(\hat{\gamma})-{{\varphi}}(t,\hat{\gamma})=(\overline{V}-\overline{\varphi})(x)=\sup_{y\in \mathbb{R}^d}
                         (\overline{V}- {\overline{\varphi}})(
                         y)=\sup_{(s,\gamma)\in \Lambda^t}
                         (V(\gamma)-{{\varphi}}(
                         s,\gamma)).
 $$
 Therefore,    $ \varphi \in {\cal{A}}^+(t,\hat{\gamma},V)$ with $(t,\hat{\gamma})\in \Lambda$.  
 Since $V$ is a viscosity subsolution of  equation (\ref{hjb1}), we have
 $$
-\lambda V(\hat{\gamma})+\partial_t{\varphi}(t, \hat{\gamma})
                           +{\mathbf{H}}(\hat{\gamma},\partial_x{\varphi}(t,\hat{\gamma}),\partial_{xx}{\varphi}(t,\hat{\gamma}))]\geq0.
$$
Thus,
$$-\lambda\overline{V}(x)+\overline{{\mathbf{H}}}(x,\nabla_x\overline{\varphi}(x),\nabla_{xx}\overline{\varphi}(x))\geq 0.$$
 By the arbitrariness of   $\overline{\varphi}\in C^{2}(\mathbb{R}^d)$, we see that $\overline{V}$ is a viscosity subsolution of  equation (\ref{hjb3}), and thus completes the proof. \ \ $\Box$
 \par
We are now in a  position  to give  the existence and consistency results 
for the viscosity solutions.
\begin{theorem}\label{theoremvexist} \ \
                          Suppose that Hypothesis \ref{hypstate}  holds. Then, for every $\lambda>\Theta$, the value
                          functional $V$ defined by (\ref{eq3}) is a
                          viscosity solution to equation  (\ref{hjb1}).
\end{theorem}
 \begin{theorem}\label{theorem3.2}
  Suppose Hypothesis \ref{hypstate}   holds, $\lambda>\Theta$ and $v\in C_p^{1,2}({\cal{C}}_0)$. Then
                 $v$ is a classical solution of  equation (\ref{hjb1}) if and only if it is a viscosity solution.
\end{theorem}
The proof of Theorems  \ref{theoremvexist} and \ref{theorem3.2} is rather standard.   For the sake of the completeness of the article and the convenience of readers, we give their proof in the appendix A.
               \par
We conclude this section with   the stability of viscosity solutions.
\begin{theorem}\label{theoremstability}
                      Let $b,\sigma,q$ satisfy  Hypothesis \ref{hypstate}, and $v\in C( {\cal{C}}_0)$. Assume
                       \item{(i)}      for any $\varepsilon>0$, there exist $b^\varepsilon, \sigma^\varepsilon, q^\varepsilon$ and $v^\varepsilon\in C({\cal{C}}_0)$ such that  $b^\varepsilon, \sigma^\varepsilon, q^\varepsilon$ satisfy  Hypothesis \ref{hypstate} and $v^\varepsilon$ is a viscosity subsolution (resp., supersolution) of  equation (\ref{hjb1}) with generators $b^\varepsilon, \sigma^\varepsilon, q^\varepsilon$;
                           \item{(ii)} as $\varepsilon\rightarrow0$, $(b^\varepsilon, \sigma^\varepsilon, q^\varepsilon, v^\varepsilon)$ converge to
                           $(b, \sigma, q,  v)$  uniformly in the following sense: 
\begin{eqnarray}\label{sss}
                          \lim_{\varepsilon\rightarrow0}\sup_{(x,u)\in  {\cal{C}}_0 \times U}[(|b^\varepsilon-b|+|\sigma^\varepsilon-\sigma|_2+|q^\varepsilon-q|)(x,u)
                         +|v^\varepsilon-v|(x)]=0.
\end{eqnarray}
                   Then $v$ is a viscosity subsoluiton (resp., supersolution) of equation  (\ref{hjb1}) with generators $b,\sigma,q$.
\end{theorem}
{\bf  Proof }. \ \ Without loss of generality, we shall only prove the viscosity subsolution property.
   Let   $\varphi\in {\cal{A}}^+(\hat{t},\hat{x}, v)$ with
  $(\hat{t},\hat{x})\in \Lambda$. By (\ref{sss}), there exists a constant $\delta>0$ such that for all $\varepsilon\in (0,\delta)$,
  $$\sup_{(t,x)\in \Lambda^{\hat{t}}}(v^\varepsilon(x)-\varphi(t,x))\leq 1.$$
 Denote $\varphi_{1}(t,x):=\varphi(t,x)+\overline{\Upsilon}(t,x,\hat{t}, {\hat{x}})$
 for all
 $(t,x)\in \Lambda$. 
 For every $\varepsilon\in (0,\delta)$, by Lemma \ref{theoremv1202}, it is clear that $v^{\varepsilon}-{{\varphi_{1}}}$ is an  upper semicontinuous functional and bounded from above on $\Lambda^{\hat{t}}$. Define a sequence of positive numbers $\{\delta_i\}_{i\geq0}$  by 
        $\delta_i=1$ for all $i\geq0$. Since $\overline{\Upsilon}$ is a gauge-type function, from Lemma \ref{theoremleft} it follows that,
  for every $\varepsilon\in (0,\delta)$ and  $(t_0,x^0)\in \Lambda^{\hat{t}}$ satisfy
 $$
v^{\varepsilon}(x^0)-{{\varphi_{1}}}({t_0},x^0)\geq \sup_{(s,y)\in \Lambda^{\hat{t}}}(v^{\varepsilon}(y)-{{\varphi_{1}}}(s,y))-\varepsilon,\
\    \mbox{and} \ \ v^{\varepsilon}(x^0)-{{\varphi_{1}}}({t_0},x^0)\geq v^{\varepsilon}(\hat{x})-{{\varphi_{1}}}({\hat{t}},\hat{x}),
 $$
  there exist $({t_{\varepsilon}},{x}^{\varepsilon})\in \Lambda^{\hat{t}}$ and a sequence $\{({t_i},x^i)\}_{i\geq1}\subset \Lambda^{\hat{t}}$ such that
  \begin{description}
        \item{(i)} $\overline{\Upsilon}({t_0},x^0,{t_{\varepsilon}},{x}^{\varepsilon})\leq {\varepsilon}$,  $\overline{\Upsilon}({t_i},x^i,{t_{\varepsilon}},{x}^{\varepsilon})\leq \frac{\varepsilon}{2^i}$ and $t_i\uparrow t_{\varepsilon}$ as $i\rightarrow\infty$,
        \item{(ii)}  $v^{\varepsilon}({x}^{\varepsilon})-{{\varphi_{1}}}({t_{\varepsilon}},{x}^{\varepsilon})-\sum_{i=0}^{\infty}\overline{\Upsilon}({t_i},x^i,{t_{\varepsilon}},{x}^{\varepsilon})\geq v^{\varepsilon}(x^0)-{{\varphi_{1}}}({t_0},x^0)$, and
        \item{(iii)}  $v^{\varepsilon}(y)-{{\varphi_{1}}}(s,y)-\sum_{i=0}^{\infty}\overline{\Upsilon}({t_i},x^i,s,y)
            <v^{\varepsilon}({x}^{\varepsilon})-{{\varphi_{1}}}({t_{\varepsilon}},{x}^{\varepsilon})-\sum_{i=0}^{\infty}\overline{\Upsilon}({t_i},x^i,{t_{\varepsilon}},{x}^{\varepsilon})$ for all $(s,y)\in \Lambda^{t_{\varepsilon}}\setminus \{(t_{\varepsilon},{x}^{\varepsilon})\}$.

        \end{description}
  %
%
%
%
We claim that
\begin{eqnarray}\label{gamma}
|{{t}_{\varepsilon}}-{\hat{t}}|+|{x}^{\varepsilon}-\hat{x}|_C\leq |{{t}_{\varepsilon}}-{\hat{t}}|+|{x}^{\varepsilon}-\hat{x}_{t_{\varepsilon}-\hat{t}}|_C
+|\hat{x}-\hat{x}_{t_{\varepsilon}-\hat{t}}|_C\rightarrow0  \ \ \mbox{as} \ \ \varepsilon\rightarrow0.
\end{eqnarray}
 Indeed, if not,  by (\ref{s0}), we can assume that there exists an $\nu_0>0$
 such
                    that
$$
  \overline{\Upsilon}({{t}_{\varepsilon}},{x}^{\varepsilon},{\hat{t}},\hat{x})
  \geq\nu_0.
$$
 Thus,  we obtain that
\begin{eqnarray*}
   &&0=v(\hat{x})- \varphi({\hat{t}},\hat{x})= \lim_{\varepsilon\rightarrow0}v^\varepsilon(\hat{x})-{{\varphi_{1}}}({\hat{t}},\hat{x})\\
   &\leq & \limsup_{\varepsilon\rightarrow0}\bigg{[}v^{\varepsilon}({x}^{\varepsilon})-{{\varphi_{1}}}({t_{\varepsilon}},{x}^{\varepsilon})-\sum_{i=0}^{\infty}\overline{\Upsilon}({t_i},x^i,{t_{\varepsilon}},{x}^{\varepsilon})\bigg{]}\\
   &\leq&\limsup_{\varepsilon\rightarrow0}\bigg{[}v({x}^{\varepsilon})-\varphi({{t}_{\varepsilon}},{x}^{\varepsilon})+(v^\varepsilon-v)({x}^{\varepsilon})
     -\sum_{i=0}^{\infty}\overline{\Upsilon}({t_i},x^i,{t_{\varepsilon}},{x}^{\varepsilon})\bigg{]}-\nu_0\\
     &\leq& v(\hat{x})-\varphi({\hat{t}},\hat{x})-\nu_0=-\nu_0,
\end{eqnarray*}
 contradicting $\nu_0>0$.  We notice that,  We notice that, by (\ref{0528a}), (\ref{0528b}), (\ref{0612b}), (\ref{0612c}), the definition of $\Upsilon$ and  the property (i) of $({t_{\varepsilon}},{x}^{\varepsilon})$, there exists a generic constant $C>0$ such that
  \begin{eqnarray*}
  2\sum_{i=0}^{\infty}({t_{\varepsilon}}-{t}_{i})
  \leq2\sum_{i=0}^{\infty}\bigg{(}\frac{\varepsilon}{2^i}\bigg{)}^{\frac{1}{2}}\leq C\varepsilon^{\frac{1}{2}};
    \end{eqnarray*}
    \begin{eqnarray*}
    |\partial_x{\Upsilon}({x}^{\varepsilon}-\hat{x}_{t_{\varepsilon}-\hat{t}})|\leq C|\hat{x}(0)-{x}^{\varepsilon}(0)|^5, \ \
    |\partial_{xx}{\Upsilon}({x}^{\varepsilon}-\hat{x}_{t_{\varepsilon}-\hat{t}})|\leq C|\hat{x}(0)-{x}^{\varepsilon}(0)|^4;
    \end{eqnarray*}
     \begin{eqnarray*}
  \bigg{|}\partial_x\sum_{i=0}^{\infty}
                      \Upsilon({x}^{\varepsilon}-x^i_{t_{\varepsilon}-t_i})
                      \bigg{|}
                      \leq18\sum_{i=0}^{\infty}|x^i(0)-{x}^{\varepsilon}(0)|^5
                      \leq18\sum_{i=0}^{\infty}\bigg{(}\frac{\varepsilon}{2^i}\bigg{)}^{\frac{5}{6}}\leq C{\varepsilon}^{\frac{5}{6}};
                        \end{eqnarray*}
                        and
                         \begin{eqnarray*}
  &&\bigg{|}\partial_{xx}\sum_{i=0}^{\infty}
                      \Upsilon({x}^{\varepsilon}-x^i_{t_{\varepsilon}-t_i})\bigg{|}
                      \leq306\sum_{i=0}^{\infty}|x^i(0)-{x}^{\varepsilon}(0)|^4
                      \leq306\sum_{i=0}^{\infty}\bigg{(}\frac{\varepsilon}{2^i}\bigg{)}^{\frac{2}{3}}\leq C\varepsilon^{\frac{2}{3}}.
  \end{eqnarray*}
 Then for any $\varrho>0$, by (\ref{sss}) and (\ref{gamma}), there exists $\varepsilon>0$ small enough such that
$$
             2|{t}_{\varepsilon}-\hat{t}|+2\sum_{i=0}^{\infty}({t_{\varepsilon}}-{t}_{i})\leq \frac{\varrho}{4}, $$
             and
             $$
                     \lambda|v^\varepsilon({x}^{\varepsilon})-v(\hat{x})|+|\partial_t{\varphi}({{t}_{\varepsilon}},{x}^{\varepsilon})-\partial_t{\varphi}({\hat{t}},\hat{x})|\leq \frac{\varrho}{4}, \ |I|\leq \frac{\varrho}{4}, \ |II|\leq \frac{\varrho}{4},
$$
where
\begin{eqnarray*}
I&=&{\mathbf{H}}^{\varepsilon}({x}^{\varepsilon},
                           \partial_x{\varphi_{2}}({{t}_{\varepsilon}},{x}^{\varepsilon}),\partial_{xx}{\varphi_{2}}({{t}_{\varepsilon}},{x}^{\varepsilon}))
                           -{\mathbf{H}}({x}^{\varepsilon},
                           \partial_x{\varphi_{2}}({{t}_{\varepsilon}},{x}^{\varepsilon}),\partial_{xx}{\varphi_{2}}({{t}_{\varepsilon}},{x}^{\varepsilon})),\\
II&=&{\mathbf{H}}({x}^{\varepsilon},
                           \partial_x{\varphi_{2}}({{t}_{\varepsilon}},{x}^{\varepsilon}),\partial_{xx}{\varphi_{2}}({{t}_{\varepsilon}},{x}^{\varepsilon}))
                       -{\mathbf{H}}(\hat{x},\partial_x{\varphi}({\hat{t}},\hat{x}),\partial_{xx}{\varphi}({\hat{t}},\hat{x})),\\
{\varphi_{2}}({{t}_{\varepsilon}},{x}^{\varepsilon})&=&{\varphi_{1}}({{t}_{\varepsilon}},{x}^{\varepsilon})+\sum_{i=0}^{\infty}\overline{\Upsilon}({t_i},x^i,{t_{\varepsilon}},{x}^{\varepsilon}),
\end{eqnarray*}
and
\begin{eqnarray*}
                                {\mathbf{H}}^{\varepsilon}(x,p,l)=\sup_{u\in{
                                         {U}}}[
                        (p,b^{\varepsilon}(x,u))_{\mathbb{R}^d}+\frac{1}{2}\mbox{tr}[ l \sigma^{\varepsilon}(x,u){\sigma^{\varepsilon}}^\top(x,u)]
                         +q^{\varepsilon}(x,u)],   \ (x,p,l)\in  {\cal{C}}_0\times  \mathbb{R}^d\times {\cal{S}}(\mathbb{R}^d).
\end{eqnarray*}
 Since $v^{\varepsilon}$ is a viscosity subsolution of equation (\ref{hjb1}) with generators $b^{\varepsilon}, \sigma^{\varepsilon},  q^{\varepsilon}$, we have
$$
                          -\lambda v^\varepsilon({x}^{\varepsilon})+\partial_t\varphi({{t}_{\varepsilon}},{x}^{\varepsilon})+2({t}_{\varepsilon}-\hat{t})+2\sum_{i=0}^{\infty}({t_{\varepsilon}}-{t}_{i})
                           +{\mathbf{H}}^{\varepsilon}({x}^{\varepsilon},
                           \partial_x{\varphi_{2}}({{t}_{\varepsilon}},{x}^{\varepsilon}),\partial_{xx}{\varphi_{2}}({{t}_{\varepsilon}},{x}^{\varepsilon}))\geq0.
$$
Thus
\begin{eqnarray*}
                       0&\leq& -\lambda v^\varepsilon({x}^{\varepsilon})+ \partial_t{\varphi}({{t}_{\varepsilon}},{x}^{\varepsilon})
                       +2({t}_{\varepsilon}-\hat{t})+2\sum_{i=0}^{\infty}({t_{\varepsilon}}-{t}_{i})\\
                       &&
                       +{\mathbf{H}}(\hat{x},\partial_x\varphi({\hat{t}},\hat{x}),
                       \partial_{xx}\varphi({\hat{t}},\hat{x}))+I+II\\
                       &\leq& -\lambda v(\hat{x})+\partial_t{\varphi}({\hat{t}},\hat{x})+{\mathbf{H}}(\hat{x},\partial_x\varphi({\hat{t}},\hat{x}),
                       \partial_{xx}\varphi({\hat{t}},\hat{x}))+\varrho.
\end{eqnarray*}
Letting $\varrho\downarrow 0$, we show that
$$
-\lambda v(\hat{x})+\partial_t{\varphi}({\hat{t}},\hat{x})+{\mathbf{H}}(\hat{x},\partial_x\varphi({\hat{t}},\hat{x}),
                       \partial_{xx}\varphi({\hat{t}},\hat{x}))\geq0.
$$
Since  $\varphi\in {\cal{A}}^+(\hat{t},\hat{x}, v)$ is arbitrary, we see that $v$ is a viscosity subsolution of  equation (\ref{hjb1}) with generators $b,\sigma,q$, and thus completes the proof.
\ \ $\Box$

\section{Viscosity solution to  HJB equation: Uniqueness theorem.}
\par
 \vbox{}
6.1. \emph{Maximum principle}. In this subsection we extend Crandall-Ishii maximum principle to infinite delay case. 
It is the cornerstone of the theory of viscosity solutions, and will be used to prove  comparison theorem in next subsection.
\begin{definition}\label{definition0513}  Let $(\hat{t},\hat{x}_0,T)\in (0,\infty)\times \mathbb{R}^d\times (\hat{t},\infty)$ and $f:[0,\infty)\times \mathbb{R}^d\rightarrow \mathbb{R}$ be an upper semicontinuous function bounded from above. We say $f\in \Phi(\hat{t},\hat{x}_0,T)$ if there is a constant $r>0$ such that
 for every $L>0$ and $\varphi\in C^{1,2}([0,T]\times \mathbb{R}^{d})$ be a function such that
            $f(s,y_0)-\varphi(s,y_0)$  has a  maximum over $[0,T]\times \mathbb{R}^d$ at a point $(t,x_0)\in (0,T)\times \mathbb{R}^d$, 
            there is a constant $C>0$ such that
\begin{eqnarray}\label{05281}
                       &&  {\varphi}_{t}(t,x_0) \geq C \ \mbox{whenever}\\
                   && |t-\hat{t}|+|x_0-\hat{x}_0|<r,\ \ \ \
                    |f(t,x_0)|+|\nabla_x\varphi(t,x_0)|
                    +|\nabla^2_x\varphi(t,x_0)|\leq L.\nonumber
\end{eqnarray}
\end{definition}
\begin{definition}\label{definition0607}
 Let $\hat{t}\in [0,\infty)$ be fixed and  $w:\Lambda\rightarrow \mathbb{R}$ be an upper semicontinuous function bounded from above.
Define,  for $(t,x_0)\in [0,\infty)\times \mathbb{R}^{d}$,
\begin{eqnarray*}
                             &&\tilde{w}^{\hat{t}}(t,x_0):=\sup_{\xi\in  {\cal{C}}_0,\xi(0)=x_0}
                             {[}w(t,\xi){]}, \ \   t\in [\hat{t},\infty);\ \
                             \tilde{w}^{\hat{t}}(t,x_0):=\tilde{w}^{\hat{t}}(\hat{t},x_0)-(\hat{t}-t)^{\frac{1}{2}}, \ \   t\in[0,\hat{t}).
\end{eqnarray*}
Let $\tilde{w}^{\hat{t},*}$ be the upper
semicontinuous envelope of $\tilde{w}^{\hat{t}}$ (see Definition D.10 in  
 \cite{fab1}), i.e.,
$$
\tilde{w}^{\hat{t},*}(t,x_0)=\limsup_{(s,y_0)\in [0,\infty)\times \mathbb{R}^d, (s,y_0)\rightarrow(t,x_0)}\tilde{w}^{\hat{t}}(s,y_0).
$$
\end{definition}
In what follows, by a $modulus \ of \ continuity$, we mean a continuous function $\rho_1:[0,\infty)\rightarrow[0,\infty)$, with $\rho_1(0)=0$ and subadditive: $\rho_1(t+s)\leq \rho_1(t)+\rho_1(s)$, for all $t,s>0$; by a $local\ modulus\ of $ $ continuity$, we mean  a continuous function $\rho_1:[0,\infty)\times[0,\infty) \rightarrow[0,\infty)$, with the properties that for each $r>0$, $t\rightarrow \rho_1(t,r)$ is a modulus of continuity and $\rho_1$ is increasing in second variable.
\begin{theorem}\label{theorem0513} (Crandall-Ishii maximum principle)\ \ Let $\kappa>0$. Let $w_1,w_2:\Lambda\rightarrow \mathbb{R}$ be upper semicontinuous functions bounded from above and such that
\begin{eqnarray}\label{05131}
                     \limsup_{t+|x|_C\rightarrow\infty}\frac{w_1(t,x)}{t+|x|_C}<0;\ \ \  \limsup_{t+|x|_C\rightarrow\infty}\frac{w_2(t,x)}{t+|x|_C}<0.
\end{eqnarray}
Let $\varphi\in C^2( \mathbb{R}^{d}\times \mathbb{R}^{d})$ be such that
$$
                w_1(t,x)+w_2(t,y)-\varphi(x(0),y(0))
$$
has a 
 maximum over $[\hat{t},\infty)\times {\cal{C}}_0\times {\cal{C}}_0$ at a point $(\hat{t}, \hat{x}, \hat{y})$. 
 Assume, moreover, $\tilde{w}_{1}^{\hat{t},*}\in \Phi(\hat{t},\hat{x}(0),T)$ and $\tilde{w}_{2}^{\hat{t},*}\in \Phi(\hat{t},\hat{y}(0),T)$ for some $T\in (\hat{t},\infty)$, and there exists a local modulus of continuity  $\rho_1$  such that, for all  $\hat{t}\leq t\leq s\leq T, \ x\in {\cal{C}}_0$,
\begin{eqnarray}\label{0608a}
w_1(t,x)-w_1(s,x_{s-t})\leq \rho_1(|s-t|,|x|_C),\ \  w_2(t,x)-w_2(s,x_{s-t})\leq \rho_1(|s-t|,|x|_C).
\end{eqnarray}
  Then there exist the
sequences  $(t_{k},x^{k}), (s_{k},y^{k})\in \Lambda^{\hat{t}}$ and
 the sequences of functionals $\varphi_k\in C_p^{1,2}(\Lambda^{t_k})$, $\psi_k\in C_p^{1,2}(\Lambda^{s_k})$ 
   such that $\varphi_k$, $\partial_t\varphi_k$, $\partial_x\varphi_k$, $\partial_{xx}\varphi_k$, $\psi_k,\partial_t\psi_k$, $\partial_x\psi_k,\partial_{xx}\psi_k$ are bounded and uniformly continuous, and such that
$$
 w_{1}(t,x)-\varphi_k(t,x)
$$
has a strict global maximum $0$ at  $(t_k,x^{k})$ over $\Lambda^{t_k}$,
$$
w_{2}(t,y)-\psi_k(t,y)
$$
has a strict global maximum $0$ at $(s_k,y^{k})$ over $\Lambda^{s_k}$, and
\begin{eqnarray}\label{0608v}
       &&\left(t_{k}, x^{k}(0), w_1(t_{k},x^k),\partial_t\varphi_k(t_{k},x^k),\partial_x\varphi_k(t_{k},x^k),\partial_{xx}\varphi_k(t_{k},x^k)\right)\nonumber\\
       &&\underrightarrow{k\rightarrow\infty}\left({\hat{t}},\hat{x}(0), w_1(\hat{t},\hat{x}),b_1, \nabla_{x_1}\varphi(\hat{x}(0),\hat{y}(0)), X\right),
\end{eqnarray}
\begin{eqnarray}\label{0608vw}
       &&\left(s_{k}, y^k(0), w_2(s_k,y^k),\partial_t\psi_k(s_k,y^k),\partial_x\psi_k(s_k,y^k),\partial_{xx}\psi_k(s_k,y^k)\right)\nonumber\\
       &&\underrightarrow{k\rightarrow\infty}\left({\hat{t}},\hat{y}(0), w_2(\hat{t}, \hat{y}),b_2, \nabla_{x_2}\varphi(\hat{x}(0),\hat{y}(0)), Y\right),
\end{eqnarray}
 where $b_{1}+b_{2}=0$ and $X,Y\in \mathcal{S}(\mathbb{R}^{d})$ satisfy the following inequality:
\begin{eqnarray}\label{II0615}
                              -\left(\frac{1}{\kappa}+|A|\right)I\leq \left(\begin{array}{cc}
                                    X&0\\
                                    0&Y
                                    \end{array}\right)\leq A+\kappa A^2,
\end{eqnarray}
and  $A=\nabla^2_{x}\varphi(\hat{x}(0),\hat{y}(0))$. Here $\nabla_{x_1}\varphi$ and $\nabla_{x_2}\varphi$ denote  the standard  first order derivative of $\varphi$
with respect to  the first variable and  the second variable, respectively.
\end{theorem}
\par
   {\bf  Proof  }. \ \
By the following Lemma \ref{lemma4.30615}, we have that
\begin{eqnarray}\label{05211}
\tilde{w}_{1}^{\hat{t},*}(t,x_0)+ \tilde{w}_{2}^{\hat{t},*}(t,y_0)-\varphi(x_0,y_0) \ \mbox{has a  maximum over}\  [0,\infty)\times \mathbb{R}^{d}\times \mathbb{R}^{d} \
\mbox{ at}\ (\hat{t},\hat{x}(0),\hat{y}(0)).
\end{eqnarray}
 Moreover, we have
  $\tilde{w}_{1}^{\hat{t},*}(\hat{t},\hat{x}(0))={w}_{1}(\hat{t}, \hat{x})$, $\tilde{w}_{2}^{\hat{t},*}(\hat{t},\hat{y}(0))={w}_{2}(\hat{t},\hat{y})$.
  Then,  by $\tilde{w}_{1}^{\hat{t},*}\in \Phi(\hat{t},\hat{x}(0),T)$, $\tilde{w}_{2}^{\hat{t},*}\in \Phi(\hat{t},\hat{y}(0),T)$ for some $T\in (\hat{t},\infty)$ and Remark \ref{remarkv0528}, the Theorem 8.3 in \cite{cran2} can be used  to obtain  sequences of bounded  functions
  $\tilde{\varphi}_k,\tilde{\psi}_k\in C^{1,2}([0,T]\times \mathbb{R}^{d})$ with bounded and uniformly continuous derivatives  such that
$\tilde{w}_{1}^{\hat{t},*}(t, x_0)-\tilde{\varphi}_k(t,x_0)$ has a strict global maximum $0$ at some point $(t_k,x_0^{k})\in (0,T)\times \mathbb{R}^d$ over {$[0,T]\times \mathbb{R}^d$}, $\tilde{w}_{2}^{\hat{t},*}(s,y_0)-\tilde{\psi}_k(s,y_0)$ has a strict global maximum $0$ at some point $(s_k,y_0^{k})\in (0,T)\times \mathbb{R}^d$ over { $[0,T]\times \mathbb{R}^d$}, and such that
\begin{eqnarray}\label{0607a}
       &&\left(t_{k}, x_0^k, \tilde{w}_{1}^{\hat{t},*}(t_{k}, x_0^k), (\tilde{\varphi}_k)_t(t_{k}, x_0^k),\nabla_x\tilde{\varphi}_k(t_{k}, x_0^k),\nabla^2_x\tilde{\varphi}_k(t_{k}, x_0^k)\right)\nonumber\\
       &&\underrightarrow{k\rightarrow\infty}\left({\hat{t}},\hat{x}(0), w_1(\hat{t},\hat{x}),b_1, \nabla_{x_1}\varphi(\hat{x}(0),\hat{y}(0)), X\right),
\end{eqnarray}
\begin{eqnarray}\label{0607b}
       &&\left(s_{k}, y_0^k, \tilde{w}_{2}^{\hat{t},*}(s_{k}, y_0^k),(\tilde{\psi}_k)_t(s_{k}, y_0^k),\nabla_x\tilde{\psi}_k(s_{k}, y_0^k),\nabla^2_x\tilde{\psi}_k(s_{k}, y_0^k)\right)\nonumber\\
       &&\underrightarrow{k\rightarrow\infty}\left({\hat{t}},\hat{y}(0), w_2(\hat{t},\hat{y}),b_2, \nabla_{x_2}\varphi(\hat{x}(0),\hat{y}(0)), Y\right),
\end{eqnarray}
  where $b_{1}+b_{2}=0$ and (\ref{II0615}) is satisfied.
  \\
  We claim that we can assume the  sequences $\{t_{k}\}_{k\geq1}\in [\hat{t},T]$ and $\{s_{k}\}_{k\geq1}\in [\hat{t},T]$. Indeed, if not, for example, there exists a subsequence of $\{t_{k}\}_{k\geq1}$ still denoted by itself such that $t_{k}<\hat{t}$ for all $k\geq0$.
Since $\tilde{w}_{1}^{\hat{t},*}(t, x_0)-\tilde{\varphi}_k(t,x_0)$
            has a maximum at $(t_{k},x_0^{k})$ on $[0,T]\times \mathbb{R}^{d}$, we obtain that
            $$
           (\tilde{\varphi}_k)_t(t_{k},x_0^{k})={\frac{1}{2}}(\hat{t}-t_{k})^{-\frac{1}{2}}\rightarrow\infty,\ \mbox{as}\ k\rightarrow\infty,
           $$
            which  contradicts   that $(\tilde{\varphi}_k)_t(t_{k},x_0^{k})\rightarrow b_{1}$, $(\tilde{\psi}_k)_t(s_{k},y_0^{k})\rightarrow b_{2}$ and $b_{1}+b_{2}=0$.
               \par
 We can modify $\tilde{\varphi}_k,\tilde{\psi}_k$  such that $\tilde{\varphi}_k,\tilde{\psi}_k\in C^{1,2}([0,\infty)\times \mathbb{R}^{d})$ with bounded and uniformly continuous derivatives,
 $\tilde{w}_{1}^{\hat{t},*}(t,x_0)+\tilde{w}_{2}^{\hat{t},*}(s,y_0)-\tilde{\varphi}_k(t,x_0)-\tilde{\psi}_k(s,y_0)$ has a strict global maximum $0$
 at
 $(t_k,x_0^{k},s_k,y_0^{k})$  on $[0,\infty)\times\mathbb{R}^d\times [0,\infty)\times\mathbb{R}^d$,
 and (\ref{0607a}) and (\ref{0607b})  hold true.
  Now we consider the functional,
              for $(t,x), (s,y)\in \Lambda^{\hat{t}}$,
\begin{eqnarray}\label{4.1111}
                 \Gamma_{k}(t,x,s,y)= w_{1}(t,x)+w_{2}(s,y)
                -\tilde{\varphi}_k(t,x(0))-\tilde{\psi}_k(s,y(0)).
\end{eqnarray}
 It is clear that $\Gamma_{k}$ is an upper semicontinuous functional  bounded from above on ${\Lambda}^{\hat{t}}\times{\Lambda}^{\hat{t}}$.  
 Define a sequence of positive numbers $\{\delta_i\}_{i\geq0}$  by 
        $\delta_i=\frac{1}{2^i}$ for all $i\geq0$.  For every $k$ and $j>0$,
           from Lemma \ref{theoremleft} it follows that,
  for every  $(\check{t}_{0},\check{x}^{0}), (\check{s}_{0},\check{y}^{0})\in   \Lambda^{\hat{t}}$ satisfying
\begin{eqnarray}\label{20210509a}
\Gamma_{k}(\check{t}_{0},\check{x}^{0},\check{s}_{0},\check{y}^{0})\geq \sup_{(t,\gamma_t),(s,\eta_s)\in  \Lambda^{\hat{t}}}\Gamma_{k}(t,x,s,y)-\frac{1}{j},\
\end{eqnarray}
  there exist $(t_{k,j},x^{k,j}), (s_{k,j},y^{k,j})\in [\hat{t},\infty)\times \Lambda^{\hat{t}}$ and two sequences $\{(\check{t}_{i},\check{x}^{i})\}_{i\geq1}, \{(\check{s}_{i},\check{y}^{i})\}_{i\geq1}\subset
  [\hat{t},\infty)\times \Lambda^{\hat{t}}$ such that
  \begin{description}
        \item{(i)} $\overline{\Upsilon}(\check{t}_{0},\check{x}^{0},t_{k,j},x^{k,j})+\overline{\Upsilon}(\check{s}_0,\check{y}^0,s_{k,j},y^{k,j})\leq \frac{1}{j}$,
         $\overline{\Upsilon}(\check{t}_{i},\check{x}^{i},t_{k,j},x^{k,j})
         +\overline{\Upsilon}(\check{s}_{i},\check{y}^{i},s_{k,j},y^{k,j})\leq \frac{1}{2^ij}$
          and $\check{t}_{i}\uparrow t_{k,j}$, $\check{s}_{i}\uparrow s_{k,j}$ as $i\rightarrow\infty$,
        \item{(ii)}  $\Gamma_k(t_{k,j},x^{k,j},s_{k,j},y^{k,j})
            -\sum_{i=0}^{\infty}\frac{1}{2^i}[\overline{\Upsilon}(\check{t}_{i},\check{x}^{i},t_{k,j},x^{k,j})
         +\overline{\Upsilon}(\check{s}_{i},\check{y}^{i},s_{k,j},y^{k,j})]\geq \Gamma_{k}(\check{t}_{0},\check{x}^{0},\check{s}_{0},\check{y}^{0})$, and
        \item{(iii)}    for all $(t,x,s,y)\in  \Lambda^{t_{k,j}}\times \Lambda^{s_{k,j}}\setminus \{(t_{k,j},x^{k,j}, s_{k,j},y^{k,j})\}$,
        \begin{eqnarray*}
       && \Gamma_{k}(t,x,s,y)
        -\sum_{i=0}^{\infty}
        \frac{1}{2^i}[\overline{\Upsilon}(\check{t}_{i},\check{x}^{i},t,x)+\overline{\Upsilon}(\check{s}_{i},\check{y}^{i},s,y)]\\
           & <&\Gamma_k(t_{k,j},x^{k,j},s_{k,j},y^{k,j})
            -\sum_{i=0}^{\infty}\frac{1}{2^i}[\overline{\Upsilon}(\check{t}_{i},\check{x}^{i},t_{k,j},x^{k,j})
         +\overline{\Upsilon}(\check{s}_{i},\check{y}^{i},s_{k,j},y^{k,j})].
        \end{eqnarray*}
        \end{description}
  By the following Lemma \ref{lemma4.40615}, we have
\begin{eqnarray}\label{4.22}
        (t_{k,j}, x^{k,j}(0))\rightarrow (t_k,x_0^k),\ (s_{k,j}, y^{k,j}(0))\rightarrow (s_k,y_0^k) \ \mbox{as}\ j\rightarrow\infty,
       \end{eqnarray}
\begin{eqnarray}\label{05231}
        \tilde{w}_{1}^{\hat{t},*}(t_{k,j}, x^{k,j}(0))\rightarrow  \tilde{w}_{1}^{\hat{t},*}(t_k,x_0^k),\ \tilde{w}_{2}^{\hat{t},*}(s_{k,j},  y^{k,j}(0))\rightarrow  \tilde{w}_{2}^{\hat{t},*}(s_k,y_0^k) \ \mbox{as}\ j\rightarrow\infty,
       \end{eqnarray}
        and 
\begin{eqnarray}\label{05232}
        w_1(t_{k,j},{x}^{k,j})\rightarrow  \tilde{w}_{1}^{\hat{t},*}(t_k,x_0^k),\ w_2(s_{k,j},y^{k,j})\rightarrow  \tilde{w}_{2}^{\hat{t},*}(s_k,y_0^k) \ \mbox{as}\ j\rightarrow\infty.
       \end{eqnarray}
        Using these and (\ref{0607a}) and (\ref{0607b}) we can therefore select a subsequence $j_k$ such that
\begin{eqnarray*}
       &&\left(t_{k,j_k}, {x}^{k,j_k}(0), w_1(t_{k,j_k},{x}^{k,j_k}),((\tilde{\varphi}_k)_t,\nabla_x\tilde{\varphi}_k,\nabla^2_x\tilde{\varphi}_k)(t_{k,j_k}, {x}^{k,j_k}(0))\right)\\
       &&\underrightarrow{k\rightarrow\infty}\left({\hat{t}},\hat{x}(0), w_1(\hat{t},\hat{x}),(b_1, \nabla_{x_1}\varphi(\hat{x}(0),\hat{y}(0)), X)\right),
\end{eqnarray*}
\begin{eqnarray*}
       &&\left(s_{k,j_k}, {y}^{k,j_k}(0), w_2(s_{k,j_k},{y}^{k,j_k}),((\tilde{\psi}_k)_t,\nabla_x\tilde{\psi}_k,\nabla^2_x\tilde{\psi}_k)(s_{k,j_k}, {y}^{k,j_k}(0))\right)\\
       &&\underrightarrow{k\rightarrow\infty}\left({\hat{t}},\hat{y}(0), w_2(\hat{t},\hat{y}),(b_2, \nabla_{x_2}\varphi(\hat{x}(0),\hat{y}(0)), Y)\right).
\end{eqnarray*}
 We notice that, by   (\ref{0528a}), (\ref{0528b}), (\ref{0612b}), (\ref{0612c}), the definition of $\Upsilon$ and the property (i) of $(t_{k,j}, x^{k,j}$, $s_{k,j}, y^{k,j})$, there exists a generic constant $C>0$ such that
  \begin{eqnarray*}
  2\sum_{i=0}^{\infty}\frac{1}{2^i}[(s_{k,j_k}-\check{s}_{i})+(t_{k,j_k}-\check{t}_{i})]
  \leq Cj_k^{-\frac{1}{2}};
    \end{eqnarray*}
     \begin{eqnarray*}
  \bigg{|}\partial_x\left[\sum_{i=0}^{\infty}\frac{1}{2^i}
                      \Upsilon({x}^{k,j_k}-\check{x}^{i}_{t_{k,j_k}-\check{t}_{i}})\right]\bigg{|}
                      +\bigg{|}\partial_x\left[\sum_{i=0}^{\infty}\frac{1}{2^i}
                      \Upsilon({y}^{k,j_k}-\check{y}^{i}_{s_{k,j_k}-\check{s}_{i}})\right]\bigg{|}
\leq Cj_k^{-\frac{5}{6}};
                        \end{eqnarray*}
                        and
\begin{eqnarray*}
                         \bigg{|}\partial_{xx}\left[\sum_{i=0}^{\infty}\frac{1}{2^i}
                      \Upsilon({x}^{k,j_k}-\check{x}^{i}_{t_{k,j_k}-\check{t}_{i}})\right]\bigg{|}
                      +\bigg{|}\partial_{xx}\left[\sum_{i=0}^{\infty}\frac{1}{2^i}
                      \Upsilon({y}^{k,j_k}-\check{y}^{i}_{s_{k,j_k}-\check{s}_{i}})\right]\bigg{|}
\leq Cj_k^{-\frac{2}{3}}.
  \end{eqnarray*}
Therefore the lemma holds with $\varphi_k(t,x):=\tilde{\varphi}_k(t,x(0))+\sum_{i=0}^{\infty}\frac{1}{2^i}\overline{\Upsilon}(\check{t}_{i},\check{x}^{i},t,x)
        $, $\psi_k(s,y):=\tilde{\psi}_k(s,y(0))+\sum_{i=0}^{\infty}\frac{1}{2^i}\overline{\Upsilon}(\check{s}_{i},\check{y}^{i},s,y)$ and $t_{k}:=t_{k,j_k}, {x}^{k}:={x}^{k,j_k}, s_{k}:=s_{k,j_k}, {y}^{k}:={y}^{k,j_k}$.\ \ $\Box$
        \begin{remark}\label{remarkv0528}
As mentioned in Remark 6.1 in Chapter V of \cite{fle1},  Condition (\ref{05281}) is stated with reverse inequality in Theorem 8.3 of \cite{cran2}. However, we  immediately obtain  results (\ref{0608v})-(\ref{II0615}) from Theorem
8.3 of \cite{cran2} by considering the functions $u_1(t,x) :=\tilde{w}_{1}^{\hat{t},*}(T-t,x)$ and
$u_2(t,x) := \tilde{w}_{2}^{\hat{t},*}(T-t,x)$.
\end{remark}

        \par
 To complete the  proof of Theorem \ref{theorem0513}, it remains to state and prove the following  two lemmas.
\begin{lemma}\label{lemma4.30615}\ \
               $\tilde{w}_{1}^{\hat{t},*}(t,x_0)+ \tilde{w}_{2}^{\hat{t},*}(t,y_0)-\varphi(x_0,y_0)$ has a  maximum over $  [0,\infty)\times \mathbb{R}^{d}\times \mathbb{R}^{d} $
 at $(\hat{t},\hat{x}(0),\hat{y}(0))$.
 Moreover, we have
 \begin{eqnarray}\label{2020020201}
  \tilde{w}_{1}^{\hat{t},*}(\hat{t},\hat{x}(0))={w}_{1}(\hat{t},\hat{x}), \ \ \tilde{w}_{2}^{\hat{t},*}(\hat{t},\hat{y}(0))={w}_{2}(\hat{t},\hat{y}).
  \end{eqnarray}
\end{lemma}
\par
   {\bf  Proof  }. \ \
For every $\hat{t}\leq t\leq s<\infty$ and $x_0\in  \mathbb{R}^{d}$, 
from the definition of $\tilde{w}^{\hat{t}}_1$ and (\ref{05131})
 there exists a constant $C_{x_0}>0$ depending only on   $x_0$ such that
\begin{eqnarray*}
                             \tilde{w}^{\hat{t}}_1(t,x_0)-\tilde{w}^{\hat{t}}_1(s,x_0)
                             &=&\sup_{\xi\in {\cal{C}}_0,\xi(0)=x_0}
                             {[}w_{1}(t,\xi)
                             {]}-\sup_{\eta\in {\cal{C}}_0,\eta_s(0)=x_0}
                             {[}w_{1}(s,\eta)
                            {]}\nonumber\\
                             &=&\sup_{\xi\in {\cal{C}}_0,|\xi|_C\leq C_{x_0},\xi(0)=x_0}
                             {[}w_{1}(t,\xi)
                             {]}-\sup_{\eta\in {\cal{C}}_0,\eta_s(0)=x_0}
                             {[}w_{1}(s,\eta)
                            {]}\nonumber\\
                            &\leq& \sup_{\xi\in {\cal{C}}_0,|\xi|_C\leq C_{x_0},\xi(0)=x_0}
                            {[}w_{1}(t,\xi)
                            -w_{1}(s,\xi_{s-t})
                            {]}.
\end{eqnarray*}
By (\ref{0608a}) we have that,
\begin{eqnarray}\label{202105083}
\tilde{w}^{\hat{t}}_1(t,x_0)-\tilde{w}^{\hat{t}}_1(s,x_0)\leq \sup_{\xi\in {\cal{C}}_0,|\xi|_C\leq C_{x_0},\xi(0)=x_0}\rho_1(|s-t|,|\xi|_C)
                             \leq \rho_1(|s-t|,C_{x_0}).
\end{eqnarray}
Clearly, if $0\leq t\leq s\leq \hat{t}$, we have
\begin{eqnarray}\label{202105084}
                             \tilde{w}^{\hat{t}}_1(t,x_0)-\tilde{w}^{\hat{t}}_1(s,x_0)=-(\hat{t}-t)^{\frac{1}{2}}+(\hat{t}-s)^{\frac{1}{2}}\leq 0,
\end{eqnarray}
and, if $0\leq t\leq  \hat{t}\leq s<\infty$ , we have
\begin{eqnarray}\label{202105085}
                             \tilde{w}^{\hat{t}}_1(t,x_0)-\tilde{w}^{\hat{t}}_1(s,x_0)\leq\tilde{w}^{\hat{t}}_1(\hat{t},x_0)-\tilde{w}^{\hat{t}}_1(s,x_0) \leq \rho_1(|s-\hat{t}|,C_{x_0}).
\end{eqnarray}
On the other hand, 
 for every $(t,x_0,y_0)\in [0,\infty)\times \mathbb{R}^{d}\times \mathbb{R}^{d}$,
  by the definitions of $\tilde{w}_{1}^{\hat{t},*}(t,x_0)$ and $ \tilde{w}_{2}^{\hat{t},*}(t,y_0)$, there exist  sequences  $(l_i,x_i), (\tau_i,y_i)\in [0,\infty)\times \mathbb{R}^d$ such that
   $(l_{i},x_i)\rightarrow (t,x_0)$, $(\tau_{i},y_i)\rightarrow (t,y_0)$
                 as $i\rightarrow\infty$ and
   \begin{eqnarray}\label{202105081}
       \tilde{w}_{1}^{\hat{t},*}(t,x_0)=\lim_{i\rightarrow\infty}\tilde{w}^{\hat{t}}_{1}(l_i,x_i), \ \ \ \tilde{w}_{2}^{\hat{t},*}(t,y_0)=\lim_{i\rightarrow\infty}\tilde{w}^{\hat{t}}_2(\tau_i,y_i).
\end{eqnarray}
Without loss of generality, we may assume  $l_i\leq \tau_i$ for all $i>0$.
 By (\ref{202105083})-(\ref{202105085}), we have
 \begin{eqnarray}\label{202105082}
       \tilde{w}_{1}^{\hat{t},*}(t,x_0)=\lim_{i\rightarrow\infty}\tilde{w}^{\hat{t}}_{1}(l_i,x_i)\leq \liminf_{i\rightarrow\infty}[\tilde{w}^{\hat{t}}_{1}(\tau_i,x_i)+\rho_1(|\tau_i-l_i|,C_{x_i})].
\end{eqnarray}
We claim that we can assume that there exists a constant $M_1>0$ such that $C_{x_i}\leq M_1$ for all $i\geq 1$. Indeed, if not, for every $n$, there exists  $i_n$ such that
\begin{eqnarray}
\tilde{w}^{\hat{t}}_{1}(l_{i_n},x_{i_n})
=\begin{cases} \sup_{\xi\in {\cal{C}}_0,|\xi|_C> n,\xi(0)=x_{i_n}}
                             {[}w_{1}(l_{i_n},\xi){]}, \ \ \ \ \ \ \ \ \ \ \ \ \ \  \ i_n\geq \hat{t};\\
                             \sup_{\xi\in {\cal{C}}_0,|\xi|_C> n,\xi(0)=x_{i_n}}
                             {[}w_{1}(\hat{t},\xi){]}-({\hat{t}}-l_{i_n})^{\frac{1}{2}}, \ \ i_n< \hat{t}.
\end{cases}
\end{eqnarray}
Letting $n\rightarrow\infty$, by (\ref{05131}), we get that
$$
\tilde{w}^{\hat{t}}_{1}(l_{i_n},x_{i_n})\rightarrow-\infty \ \mbox{as}\ n\rightarrow\infty,
$$
 which contradicts the convergence that $ \tilde{w}_{1}^{\hat{t},*}(t,x)=\lim_{i\rightarrow\infty}\tilde{w}^{\hat{t}}_{1}(l_i,x_i)$.  Then, by  (\ref{202105082}),
 \begin{eqnarray}\label{20210704a}
                           \tilde{w}_{1}^{\hat{t},*}(t,x)
                           \leq \liminf_{i\rightarrow\infty}[\tilde{w}^{\hat{t}}_{1}(\tau_i,x_i)+\rho_1(|\tau_i-l_i|,M_1)] = \liminf_{i\rightarrow\infty}\tilde{w}^{\hat{t}}_{1}(\tau_i,x_i).
\end{eqnarray}
Therefore,by (\ref{202105081}), (\ref{202105082}) and the definitions of $\tilde{w}^{\hat{t}}_{1}$ and $\tilde{w}^{\hat{t}}_{2}$,
\begin{eqnarray}\label{20210508b}
                            &&\tilde{w}_{1}^{\hat{t},*}(t,x_0)+\tilde{w}_{2}^{\hat{t},*}(t,y_0)-\varphi(x_0,y_0)\nonumber\\
                          &\leq&\liminf_{i\rightarrow\infty}[\tilde{w}^{\hat{t}}_{1}(\tau_i,x_i)+\tilde{w}^{\hat{t}}_{2}(\tau_i,y_i)-\varphi(x_i,y_i)]\nonumber\\
                           &\leq&\sup_{(l,x,y)\in [0,\infty)\times \mathbb{R}^{d}\times \mathbb{R}^{d}}[\tilde{w}^{\hat{t}}_{1}(l,x)+\tilde{w}^{\hat{t}}_{2}(l,y)-\varphi(x,y)]\nonumber\\
                            &=&\sup_{(l,x,y)\in [\hat{t},\infty)\times \mathbb{R}^{d}\times \mathbb{R}^{d}}[\tilde{w}^{\hat{t}}_{1}(l,x)+\tilde{w}^{\hat{t}}_{2}(l,y)-\varphi(x,y)].
\end{eqnarray}
We also have, for $(l,x,y)\in [\hat{t},\infty)\times \mathbb{R}^d\times \mathbb{R}^d$,
\begin{eqnarray}\label{0608aa}
                         &&\tilde{w}^{\hat{t}}_{1}(l,x)+  \tilde{w}^{\hat{t}}_{2}(l,y)-\varphi(x,y)\nonumber   \\
                          &=&\sup_{
                          \gamma,\eta\in{\cal{C}}_0,\gamma(0)=x,
                          \eta(0)=y}
                          \left[w_{1}(l,\gamma)
                          +w_{2}(l,\eta)
                          -\varphi(\gamma(0),\eta(0))\right]\nonumber \\
                          &\leq&w_{1}(\hat{t},\hat{x})+w_{2}(\hat{t},\hat{y})-\varphi(\hat{x}(0),\hat{y}(0)),
\end{eqnarray}
                       where the  inequality becomes equality if  $l={\hat{t}}$ 
                       and $x=\hat{x}(0),y=\hat{y}(0)$.
                       Combining  (\ref{20210508b}) and (\ref{0608aa}), we obtain that
\begin{eqnarray}\label{0608a1}
                            &&\tilde{w}_{1}^{\hat{t},*}(t,x_0)+\tilde{w}_{2}^{\hat{t},*}(t,y_0)-\varphi(x_0,y_0)\leq w_{1}(\hat{t},\hat{x})+w_{2}(\hat{t},\hat{y})-\varphi(\hat{x}(0),\hat{y}(0)).
\end{eqnarray}
By the definitions of $\tilde{w}_1^{\hat{t},*}$ and $\tilde{w}_2^{\hat{t},*}$, we have $\tilde{w}_1^{\hat{t},*}(t,x_0)\geq \tilde{w}^{\hat{t}}_1(t,x_0), \tilde{w}_{2}^{\hat{t},*}(t,y_0) \geq \tilde{w}_{2}(t,y_0)$.  Then by also (\ref{0608aa}) and (\ref{0608a1}), for every $(t,x_0,y_0)\in [0,\infty)\times \mathbb{R}^{d}\times \mathbb{R}^{d}$,
\begin{eqnarray}\label{0608abc}
                            &&\tilde{w}_{1}^{\hat{t},*}(t,x_0)+\tilde{w}_{2}^{\hat{t},*}(t,y_0)-\varphi(x_0,y_0)\leq w_{1}(\hat{t},\hat{x})+w_{2}(\hat{t},\hat{y})-\varphi(\hat{x}(0),\hat{y}(0))\nonumber\\
                            &=&\tilde{w}^{\hat{t}}_{1}(\hat{t},\hat{x}(0))+  \tilde{w}^{\hat{t}}_{2}(\hat{t},\hat{y}(0))-\varphi(\hat{x}(0),\hat{y}(0))\nonumber\\
                            &\leq&\tilde{w}_{1}^{\hat{t},*}(\hat{t},\hat{x}(0))+  \tilde{w}_{2}^{\hat{t},*}(\hat{t},\hat{y}(0))-\varphi(\hat{x}(0),\hat{y}(0)).
\end{eqnarray}
Thus
we obtain that (\ref{2020020201}) holds true, and  $ \tilde{w}_{1}^{\hat{t},*}(t,x_0)+\tilde{w}_{2}^{\hat{t},*}(t,y_0)-\varphi(x_0, y_0)$ has a maximum at $({\hat{t}},\hat{x}(0),
\hat{y}(0))$
 on $[0,\infty)\times \mathbb{R}^{d}\times \mathbb{R}^{d}$.
The proof is now complete. \ \ $\Box$
\begin{lemma}\label{lemma4.40615}\ \
The maximum points $(t_{k,j},{x}^{k,j}, s_{k,j},{y}^{k,j})$
of $\Gamma_{k}(t,x,s,y)
        -\sum_{i=0}^{\infty}
        \frac{1}{2^i}[\overline{\Upsilon}(\check{t}_{i},\check{x}^{i},t,x)+\overline{\Upsilon}(\check{s}_{i},\check{y},s,y)]$ 
        satisfy  conditions (\ref{4.22}), (\ref{05231}) and (\ref{05232}).
\end{lemma}
\par
   {\bf  Proof  }. \ \
   Recall that $\tilde{w}_{1}^{\hat{t},*}\geq \tilde{w}^{\hat{t}}_{1}, \tilde{w}_{2}^{\hat{t},*} \geq \tilde{w}^{\hat{t}}_{2}$,
by  
the definitions of $\tilde{w}^{\hat{t}}_1$ and $\tilde{w}^{\hat{t}}_2$, we get that
 \begin{eqnarray*}
                           && \tilde{w}_{1}^{\hat{t},*}(t_{k,j},{x}^{k,j}(0))+\tilde{w}_{2}^{\hat{t},*}(s_{k,j},{y}^{k,j}(0))
                           -\tilde{\varphi}_k(t_{k,j},{x}^{k,j}(0))-\tilde{\psi}_k(s_{k,j},{y}^{k,j}(0))\\
                           &\geq&w_1(t_{k,j},{x}^{k,j})+w_2(s_{k,j},{y}^{k,j})-\tilde{\varphi}_k(t_{k,j},{x}^{k,j}(0))-\tilde{\psi}_k(s_{k,j},{y}^{k,j}(0))
                           =\Gamma_{k}(t_{k,j},{x}^{k,j},s_{k,j},{y}^{k,j}).
  \end{eqnarray*}
We notice that, from (\ref{20210509a}) and the property (ii) of $(t_{k,j},{x}^{k,j}, s_{k,j},{y}^{k,j})$,
\begin{eqnarray*}
                 \Gamma_{k}(t_{k,j},{x}^{k,j},s_{k,j},{y}^{k,j})\geq\Gamma_{k}(\check{t}_{0},\check{x}^{0},\check{s}_{0},\check{y}^{0})\geq \sup_{(t,x),(s,y)\in \Lambda^{\hat{t}}}\Gamma_{k}(t,x,s,y)-\frac{1}{j},
\end{eqnarray*}
and by  
the  definitions of $\tilde{w}_{1}^{\hat{t},*}$ and $\tilde{w}_{2}^{\hat{t},*}$,
$$
\sup_{(t,x),(s,y)\in  \Lambda^{\hat{t}}}\Gamma_{k}(t,x,s,y)\geq\tilde{w}_{1}^{\hat{t},*}(t_k,x_0^k)+\tilde{w}_{2}^{\hat{t},*}(s_k,y_0^k) -\tilde{\varphi}_k(t_k,x_0^k)-\tilde{\psi}_k(s_k,y_0^k).
$$
Therefore,
\begin{eqnarray}\label{0525b}
&&\tilde{w}_{1}^{\hat{t},*}(t_{k,j},{x}^{k,j}(0))+\tilde{w}_{2}^{\hat{t},*}(s_{k,j},{y}^{k,j}(0))-\tilde{\varphi}_k(t_{k,j},{x}^{k,j}(0))-\tilde{\psi}_k(s_{k,j},{y}^{k,j}(0))\nonumber\\
                  &\geq&\Gamma_{k}(t_{k,j},{x}^{k,j},s_{k,j},{y}^{k,j})\geq
                  \tilde{w}_{1}^{\hat{t},*}(t_k,x_0^k)+\tilde{w}_{2}^{\hat{t},*}(s_k,y_0^k) -\tilde{\varphi}_k(t_k,x_0^k)-\tilde{\psi}_k(s_k,y_0^k)-\frac{1}{j}.
\end{eqnarray}
By (\ref{05131}) and $\tilde{\varphi}_k, \tilde{\psi}_k$ are bounded, 
 there exists a constant  ${M}_2>0$  that is sufficiently  large   that
 $
           \Gamma_k(t,x,s,y)<\sup_{(l,z^1),(r,z^2)\in  \Lambda^{\hat{t}}}\Gamma_{k}(l,z^1,r,z^2)-1
           $ for all  $(t+|x|_C)\vee (s+|y|_C)\geq M_2$. Thus, we have $(t_{k,j}+|{x}^{k,j}|_C)\vee
           (s_{k,j}+|{y}^{k,j}|_C)<M_2$. In particular, $t_{k,j}\vee s_{k,j}\vee|{x}^{k,j}(0)|\vee|{y}^{k,j}(0)|<M_2$. We note that $M_2$ is independent of $j$.
 Then letting $j\rightarrow\infty$ in (\ref{0525b}), 
   we obtain (\ref{4.22}).  Indeed, if not, we may assume that  there exist $(\grave{t},\grave{x},\grave{s},\grave{y})\in [0, \infty)\times \mathbb{R}^{d}\times [0, \infty)\times \mathbb{R}^{d}$ and  a subsequence of $(t_{k,j},{x}^{k,j}(0),s_{k,j},{y}^{k,j}(0))$ still denoted by itself  such that
       $$
       (t_{k,j},{x}^{k,j}(0),s_{k,j},{y}^{k,j}(0))\rightarrow  (\grave{t},\grave{x},\grave{s},\grave{y})\neq (t_k,x_0^k,s_k,y_0^k).
       $$
Letting $j\rightarrow\infty$ in (\ref{0525b}), by the upper semicontinuity of $\tilde{w}_{1}^{\hat{t},*}+\tilde{w}_{2}^{\hat{t},*}-\tilde{\varphi}_k-\tilde{\psi}_k$, we have
$$
           \tilde{w}_{1}^{\hat{t},*}(\grave{t},\grave{x})+\tilde{w}_{2}^{\hat{t},*}(\grave{s},\grave{y})-\tilde{\varphi}_k(\grave{t},\grave{x})-\tilde{\psi}_k(\grave{s},\grave{y})
                 \geq\tilde{w}_{1}^{\hat{t},*}(t_k,{x}_0^k)+\tilde{w}_{2}^{\hat{t},*}(s_k,{y}_0^k) -\tilde{\varphi}_k(t_k,{x}_0^k)-\tilde{\psi}_k(s_k,{y}_0^k),
$$
       which  contradicts that 
       $(t_k,x_0^{k},s_k,y_0^{k})$  is the  strict  maximum point of $\tilde{w}_{1}^{\hat{t},*}(t,x_0)+\tilde{w}_{2}^{\hat{t},*}(s,y_0)-\tilde{\varphi}_k(t,x_0)-\tilde{\psi}_k(s,y_0)$ on $[0,\infty)\times\mathbb{R}^d\times [0,\infty)\times\mathbb{R}^d$.\par
       By (\ref{4.22}), the  upper semicontinuity of $\tilde{w}_{1}^{\hat{t},*}$ and $\tilde{w}_{2}^{\hat{t},*}$ and  the continuity of $\tilde{\varphi}_k$ and $\tilde{\psi}_k$, letting $j\rightarrow\infty$ in (\ref{0525b}), we obtian
  (\ref{05231}), and then also (\ref{05232}).
  The proof is now complete. \ \ $\Box$

\par
 \vbox{}
6.2. \emph{Uniqueness}.
             This subsection is devoted to a  proof of uniqueness of  viscosity
                   solutions to equation  (\ref{hjb1}). This result, together with
                  the results from  the previous section, will be used to characterize
                   the value functional defined by (\ref{eq3}).
                   \par
We  now state the main result of this subsection.
\begin{theorem}\label{theoremhjbm}  Suppose Hypothesis \ref{hypstate}   holds and $\lambda\geq (12+15L)L$.
                         Let $W_1\in C({\cal{C}}_0)$ $(\mbox{resp}., W_2\in C({\cal{C}}_0))$ be  a viscosity subsolution (resp., supersolution) to equation  (\ref{hjb1}) and  let  there exist a constant $\hat{C}>0$
                        such that for $h\in [0,\infty),  x,y\in{\cal{C}}_0$,
\begin{eqnarray}\label{w}
                                 |W_1(x)|\vee|W_2(x)|\leq \hat{C} (1+|x|_C); \ \ \
                                       |W_1(x)-W_1(y)|\vee |W_2(x)-W_2(y)|
                                    \leq
                                 \hat{C}|x-y|_C;
                                  \end{eqnarray}
                                  and
\begin{eqnarray}\label{w12072}
                                  |W_1(x)-W_1(x_h)|\vee |W_2(x)-W_2(x_h)|
                                    \leq \hat{C}(1+|x|_C)(h+h^{\frac{1}{2}}+1-e^{-\lambda h}).
\end{eqnarray}
                   Then  $W_1\leq W_2$.
\end{theorem}
\par
                      Theorems    \ref{theoremvexist} and \ref{theoremhjbm} lead to the result (given below) that the viscosity solution to the  HJB equation with infinite delay  given in (\ref{hjb1})
                      corresponds to the value functional  $V$ of our optimal control problem given in (\ref{state1}) and (\ref{cost1}).
\begin{theorem}\label{theorem52}\ \
                 Let Hypothesis \ref{hypstate}  hold and $\lambda\geq (12+15L)L$. Then the value
                          functional $V$ defined by (\ref{eq3}) is the unique viscosity
                          solution to equation (\ref{hjb1}) in the class of functionals satisfying (\ref{valuep1}) and(\ref{hold}).
\end{theorem}
\par
   {\bf  Proof  }. \ \   Theorem \ref{theoremvexist} shows that $V$ is a viscosity solution to (\ref{hjb1}).  Thus, our conclusion follows from
    Theorems  \ref{theoremv}, \ref{theorem3.9} and
    \ref{theoremhjbm}.  \ \ $\Box$
\par
 We are now in a position to  prove Theorem \ref{theoremhjbm}.   
\par
   {\bf  Proof of Theorem \ref{theoremhjbm} } \ \   The proof of this theorem  is rather long. Thus, we split it into several
        steps.
\par
            $Step\  1.$ Definitions of auxiliary functionals.
            \par
To prove the theorem, we  assume the converse result that $\tilde{x}\in
      {\cal{C}}_0$ exists  such that
        $\tilde{m}:=W_1(\tilde{x})-W_2(\tilde{x})>0$.
\par
         Consider that  $\varepsilon >0$ is  a small number such that
 $$
 W_1(\tilde{x})-W_2(\tilde{x})-2\varepsilon \Upsilon^{1,3}(\tilde{x})
 >\frac{\tilde{m}}{2},
 $$
      and
\begin{eqnarray}\label{5.3}
                          {\varepsilon}\leq\frac{\tilde{m}}{16}.
\end{eqnarray}
Next,  we define for any  $(t,x,y)\in [0,\infty)\times{\cal{C}}_0\times{\cal{C}}_0$,
\begin{eqnarray*}
                 \Psi(t,x,y)=W_1(x)-W_2(y)-{\beta}\Upsilon(x,y)-\beta^{\frac{1}{3}}|x(0)-y(0)|^2-\varepsilon(\Upsilon^{1,3}(x)+\Upsilon^{1,3}(y)).
\end{eqnarray*}
     By (\ref{s0}) and (\ref{w}), it is clear that $\Psi$ is bounded from above on $[0,\infty)\times{\cal{C}}_0\times {\cal{C}}_0$.
     Moreover,  by Lemma \ref{theoremv1202}, $\Psi$ is an  upper semicontinuous  functional on $([0,\infty)\times{\cal{C}}_0\times{\cal{C}}_0,d_{1,\infty})$,
     where $d_{1,\infty}((t,x^1,x^2),(s,y^1,y^2))=|s-t|+|x^1_{(s-t)\vee0}-y^1_{(t-s)\vee0}|_C+|x^2_{(s-t)\vee0}-y^2_{(t-s)\vee0}|_C$
       for all $(t,x^1,x^2)$, $(s,y^1,y^2)\in [0,\infty)\times{\cal{C}}_0\times{\cal{C}}_0$.
        Define a sequence of positive numbers $\{\delta_i\}_{i\geq0}$  by 
        $\delta_i=\frac{1}{2^i}$ for all $i\geq0$.
           Since 
             ${\Upsilon}$ is a gauge-type function, from Lemma \ref{theoremleft} it follows that,
  for every  $(1,x^0,y^0)\in [0,\infty)\times{\cal{C}}_0\times {\cal{C}}_0$ satisfy
$$
\Psi(1,x^0,y^0)\geq \sup_{(t,x,y)\in [1,\infty)\times {\cal{C}}_0\times {\cal{C}}_0}\Psi(t,x,y)-\frac{1}{\beta},\
\    \mbox{and} \ \ \Psi(1,x^0,y^0)\geq \Psi(1,\tilde{x},\tilde{x}) >\frac{\tilde{m}}{2},
 $$
  there exist $(\hat{t}, \hat{x},\hat{y})\in [1,\infty)\times {\cal{C}}_0\times {\cal{C}}_0$ and a sequence $\{(t_i,x^i,y^i)\}_{i\geq1}\subset
   [1,\infty)\times{\cal{C}}_0\times {\cal{C}}_0$ such that
  \begin{description}
        \item{(i)} $\Upsilon(1, x^0,\hat{t}, \hat{x})+\Upsilon(1, y^0,\hat{t}, \hat{y})+|\hat{t}-1|^2\leq \frac{1}{\beta}$,
         $\Upsilon(t_i, x^i,\hat{t}, \hat{x})+\Upsilon(t_i, y^i,\hat{t}, \hat{y})+|\hat{t}-t_i|^2
         \leq \frac{1}{\beta2^i}$ and $t_i\uparrow \hat{t}$ as $i\rightarrow\infty$,
        \item{(ii)}  $\Psi_1(\hat{t},\hat{x}, \hat{y})
       \geq \Psi(1,x^0,y^0)$, and
        \item{(iii)}    for all $(s,x,y)\in [\hat{t},\infty)\times {\cal{C}}_0\times {\cal{C}}_0\setminus \{(\hat{t}, \hat{x}, \hat{y})\}$,
        \begin{eqnarray}\label{iii4}
        \Psi_1(s,x,y)
         <\Psi_1(\hat{t},\hat{x}, \hat{y}),
        \end{eqnarray}
        where we define
        $$
        \Psi_1(s,x,y)
        =\Psi(s,x,y)-\sum_{i=0}^{\infty}
        \frac{1}{2^i}[\Upsilon(t_i, x^i,s,x)+\Upsilon(t_i, y^i,s,y)+|{s}-t_i|^2], \ \ (s,x,y)\in [0,\infty)\times {\cal{C}}_0\times {\cal{C}}_0.
        $$
        \end{description}
             We should note that the point
             $(\hat{t}, \hat{x},\hat{y})$ depends on $\beta$ and
              $\varepsilon$.
\par
$Step\ 2.$
There exists ${{M}_0}>0$ independent of $\beta$
    such that
                   \begin{eqnarray}\label{5.10jiajiaaaa}| \hat{x}|_C\vee|\hat{y}|_C<M_0,
                   \end{eqnarray} and
 the following result  holds true:
 \begin{eqnarray}\label{5.10}
                      \beta^{\frac{1}{3}} |\hat{x}-\hat{y}|_{C}^2
                         +\beta^{\frac{1}{2}}|\hat{x}(0)-\hat{y}(0)|^2
                         \rightarrow0 \ \mbox{as} \ \beta\rightarrow\infty.
 \end{eqnarray}
  Let us show the above. First,   noting $\varepsilon$ is independent of  $\beta$, by the definition of  ${\Psi}$,
 there exists an ${M}_0>0$  that is sufficiently  large   that
           $
           \Psi(t, x, y)<0
           $ for all $t\geq0$ and $|x|_C\vee|y|_C\geq M_0$. Thus, we have $|\hat{x}|_C\vee|\hat{y}|_C\vee
           |x^{0}|_C\vee|y^0|_C<M_0$.\\
   Second, by (\ref{iii4}), we have
 \begin{eqnarray}\label{5.56789}
                        2\Psi_1(\hat{t},\hat{x}, \hat{y})
            \geq\Psi_1(\hat{t},\hat{x}, \hat{x})
            +\Psi_1(\hat{t},\hat{y}, \hat{y}).
 \end{eqnarray}
 This implies that
 \begin{eqnarray}\label{5.6}
                         &&2{\beta}\Upsilon(\hat{x},\hat{y})+
                         2{\beta}^{\frac{1}{3}}|\hat{x}(0)-\hat{y}(0)|^2\nonumber\\
                         &
                         \leq&|W_1(\hat{x})-W_1(\hat{y})|
                                   +|W_2(\hat{x})-W_2(\hat{y})|+
                                   \sum_{i=0}^{\infty}\frac{1}{2^i}[\Upsilon(t_i,y^i,\hat{t},\hat{x})+\Upsilon(t_i, x^i,\hat{t},\hat{y})].
 \end{eqnarray}
 On the other hand, by Lemma \ref{theoremS000} and the property (i) of $(\hat{t},\hat{x},\hat{y})$,
 \begin{eqnarray}\label{4.7jiajia130}
 \sum_{i=0}^{\infty}\frac{1}{2^i}[\Upsilon(t_i,y^i,\hat{t},\hat{x})+\Upsilon(t_i, x^i,\hat{t},\hat{y})]
 &\leq&2^5\sum_{i=0}^{\infty}\frac{1}{2^i}[\Upsilon(t_i,y^i,\hat{t},\hat{y})
 +\Upsilon(t_i,x^i,\hat{t},\hat{x})+2\Upsilon(\hat{x},\hat{y})]\nonumber\\
 &\leq&\frac{2^6}{\beta}+{2^7}\Upsilon(\hat{x},\hat{y}).
 \end{eqnarray}
Combining (\ref{5.6}) and (\ref{4.7jiajia130}),  from  (\ref{w}) and (\ref{5.10jiajiaaaa})  we have
 \begin{eqnarray}\label{5.jia6}
                         &&(2{\beta}-2^7)\Upsilon(\hat{x},\hat{y})+
                         2{\beta}^{\frac{1}{3}}|\hat{x}(0)-\hat{y}(0)|^2
                         \leq |W_1(\hat{x})-W_1(\hat{y})|
                                   +|W_2(\hat{x})-W_2(\hat{y})|+\frac{2^6}{\beta}\nonumber\\
                                   &\leq& 2L(2+|\hat{x}|_C+|\hat{y}|_C)+\frac{2^6}{\beta}
                                   \leq 4L(1+M_0)+\frac{2^6}{\beta}.
 \end{eqnarray}
   Letting $\beta\rightarrow\infty$, we get
    \begin{eqnarray*}
                \Upsilon(\hat{x},\hat{y})
                \leq \frac{1}{2{\beta}-2^7}\left[4L(1+M_0)+\frac{2^6}{\beta}\right]\rightarrow0\ 
                          \mbox{as} \ \beta\rightarrow\infty.
                         \end{eqnarray*}
From (\ref{s0}) it follows that
\begin{eqnarray}\label{5.66666123}
|\hat{x}-\hat{y}|_C \rightarrow0\ 
                          \mbox{as} \ \beta\rightarrow\infty.
 \end{eqnarray}
                   Combining (\ref{s0}),  (\ref{w}), (\ref{5.6}), (\ref{4.7jiajia130}) and (\ref{5.66666123}), we see
                           that
                            \begin{eqnarray}\label{5.10112345}
                        &&{\beta}|\hat{x}-\hat{y}|_C^6+\beta^{\frac{1}{3}}|\hat{x}(0)-\hat{y}(0)|^2
                        \leq{\beta}\Upsilon(\hat{x},\hat{y})
                        +\beta^{\frac{1}{3}}|\hat{x}(0)-\hat{y}(0)|^2\nonumber\\
                        &\leq&\frac{1}{2}[|W_1(\hat{x})-W_1(\hat{y})|
                                   +|W_2(\hat{x})-W_2(\hat{y})|]+\frac{2^5}{\beta}+2^6\Upsilon(\hat{x},\hat{y})\nonumber\\
                        &\leq& \hat{C}|\hat{x}-\hat{y}|_{C}
                                   +\frac{2^5}{\beta}+{2^{8}}|\hat{x}-\hat{y}|_C^6
                                   \rightarrow0 \ 
                                   \mbox{as} \ \beta\rightarrow\infty.
 \end{eqnarray}
   Multiply the leftmost and rightmost sides of inequality (\ref{5.10112345}) by $\beta^{\frac{1}{6}}$, we obtain that
   \begin{eqnarray}\label{5.10123445567890}
                      \beta^{\frac{1}{2}}|\hat{x}(0)-\hat{y}(0)|^2
                     \leq
                               \hat{C}\beta^{\frac{1}{6}}|\hat{x}-\hat{y}|_C
                                   +\frac{2^5}{\beta^{\frac{5}{6}}}+{2^{8}}\beta^{\frac{1}{6}}|\hat{x}-\hat{y}|_C^6.
 \end{eqnarray}
By also (\ref{5.10112345}), the right side of above inequality converges to 0 as $\beta\rightarrow\infty$. Then we have that
 \begin{eqnarray*}
                      \beta^{\frac{1}{2}}|\hat{{\gamma}}_{{\hat{t}}}(\hat{t})-\hat{{\eta}}_{{\hat{t}}}(\hat{t})|^2
                     \rightarrow0 \ 
                                   \mbox{as} \ \beta\rightarrow\infty.
 \end{eqnarray*}
Combining with (\ref{5.10112345}), we get (\ref{5.10}) holds true.
\par
 $Step\ 3.$    Maximum principle.
\par
We put, for $(t,x), (t,y)\in \Lambda^{\hat{t}}$, 
\begin{eqnarray}\label{06091}
                             w_{1}(t,x)=W_1(x)-2^5\beta \Upsilon(t,x,\hat{t},\hat{\xi})-\varepsilon\Upsilon^{1,3}(x)
                 -\varepsilon \overline{\Upsilon}(\hat{t},\hat{x},t,x)
                -\sum_{i=0}^{\infty}
        \frac{1}{2^i}\overline{\Upsilon}(t_i,x^i,t,x),
        \end{eqnarray}
        \begin{eqnarray}\label{06092}
                             w_{2}(t,y)=-W_2(y)-2^5\beta \Upsilon(t,y,\hat{t},\hat{\xi})-\varepsilon\Upsilon^{1,3}(y)
                 -\varepsilon \overline{\Upsilon}(\hat{t},\hat{y},t,y)
                 -\sum_{i=0}^{\infty}
        \frac{1}{2^i}{\Upsilon}^3(t_i,y^i,t,y),
\end{eqnarray}
where $\hat{\xi}=\frac{\hat{x}+\hat{y}}{2}$.   We  note that $w_1,w_2$ depend on $\hat{\xi}_{{\hat{t}}}$, and thus on $\beta$ and
              $\varepsilon$.
By the following Lemma \ref{0611a}, $w_1$ and $w_2$ satisfy the conditions of  Theorem \ref{theorem0513}.
Then by Theorem \ref{theorem0513},  there exist the
sequences  $(l_k,\check{x}^k), (s_{k},\check{y}^k)\in \Lambda^{\hat{t}}$ and
 the sequences of functionals $\varphi_k\in C_p^{1,2}(\Lambda^{l_k}),\psi_k\in C_p^{1,2}(\Lambda^{s_k})$ 
   such that $\varphi_k$, $\partial_t\varphi_k$, $\partial_x\varphi_k$, $\partial_{xx}\varphi_k$, $\psi_k,\partial_t\psi_k$, $\partial_x\psi_k,\partial_{xx}\psi_k$ are bounded and uniformly continuous, and such that
\begin{eqnarray}\label{0609a}
 w_{1}(t,x)-\varphi_k(t,x)
\end{eqnarray}
has a strict global maximum $0$ at  $(l_k,\check{x}^k)$ over $\Lambda^{l_k}$,
\begin{eqnarray}\label{0609b}
w_{2}(t,y)-\psi_k(t,y)
\end{eqnarray}
has a strict global maximum $0$ at $(s_k,\check{y}^k)$ over $\Lambda^{s_k}$, and
\begin{eqnarray}\label{0608v1}
       &&\left(l_k, \check{x}^k(0), w_1(l_k,\check{x}^k),\partial_t\varphi_k(l_k,\check{x}^k),\partial_x\varphi_k(l_k,\check{x}^k),\partial_{xx}\varphi_k(l_k,\check{x}^k)\right)\nonumber\\
       &&\underrightarrow{k\rightarrow\infty}\left({\hat{t}},\hat{x}(0), w_1(\hat{t},\hat{x}),b_1, \nabla_{x_1}\varphi(\hat{x}(0),\hat{y}(0)), X\right),
\end{eqnarray}
\begin{eqnarray}\label{0608vw1}
       &&\left(s_{k}, \check{y}^k(0), w_2(s_k,\check{y}^k),\partial_t\psi_k(s_k,\check{y}^k),\partial_x\psi_k(s_k,\check{y}^k),\partial_{xx}\psi_k(s_k,\check{y}^k)\right)\nonumber\\
       &&\underrightarrow{k\rightarrow\infty}\left({\hat{t}},\hat{y}(0), w_2(\hat{t}, \hat{y}),b_2, \nabla_{x_2}\varphi(\hat{x}(0),\hat{y}(0)), Y\right),
\end{eqnarray}
 where $b_{1}+b_{2}=0$ and $X,Y\in \mathcal{S}(\mathbb{R}^{d})$ satisfy the following inequality:
\begin{eqnarray}\label{II}
                              {-6\beta^{\frac{1}{3}}}\left(\begin{array}{cc}
                                    I&0\\
                                    0&I
                                    \end{array}\right)\leq \left(\begin{array}{cc}
                                    X&0\\
                                    0&Y
                                    \end{array}\right)\leq  6\beta^{\frac{1}{3}} \left(\begin{array}{cc}
                                    I&-I\\
                                    -I&I
                                    \end{array}\right).
\end{eqnarray}
 We note that  (\ref{II}) follows from (\ref{II0615}) choosing $\kappa=\beta^{-\frac{1}{3}}$, and sequence  $(\check{x}^{k},\check{y}^{k},l_{k},s_{k},\varphi_k,\psi_k)$ and $b_{1},b_{2},X,Y$  depend on  $\beta$ and
              $\varepsilon$.   By the following Lemma \ref{lemma4.4}, we have
\begin{eqnarray}\label{4.23}
\lim_{k\rightarrow\infty}[d_\infty(l_k,\check{x}^{k},\hat{t},\hat{x})
+d_\infty(s_k,\check{\eta}^{k},\hat{t},\hat{y})]=0.
\end{eqnarray}
 For every $(t,x),(s,y)\in {{\Lambda}^{T-\bar{a}}}$, let
\begin{eqnarray*}
\chi^{k}(t,x)
        :=\varepsilon\Upsilon^{1,3}(x)+
               \varepsilon \overline{\Upsilon}(t,x,\hat{t},\hat{x})
                +\sum_{i=0}^{\infty}
        \frac{1}{2^i}\overline{\Upsilon}(t_i,x^i,t,x)+2^5\beta\Upsilon(t,x,\hat{t},\hat{\xi})
                +\varphi_k(t,x),
  \end{eqnarray*}
  \begin{eqnarray*}
\hbar^{k}(s,y)
        :=\varepsilon\Upsilon^{1,3}(y)+\varepsilon \overline{\Upsilon}(s,y,\hat{t},\hat{y})
                 +\sum_{i=0}^{\infty}
        \frac{1}{2^i}{\Upsilon}(t_i,y^i,s,y)+2^5\beta\Upsilon(s,y,\hat{t},\hat{\xi})
        +\psi_k(s,y).
  \end{eqnarray*}
  It is clear that  $\chi^{k}(\cdot)\in C^{1,2}_p(\Lambda^{l_k}),\hbar^{k}(\cdot)\in C^{1,2}_p(\Lambda^{s_k})$. 
  Moreover, by  (\ref{0609a}), (\ref{0609b}) and definitions of $w_1$ and $w_2$,
  $$
                         (W_1-\chi^{k})(l_k,\check{x}^{k})=\sup_{(t,x)\in \Lambda^{l_k}}
                         (W_1-\chi^{k})(t,x),
$$
$$
                         (W_2+\hbar^{k})(s_k,\check{y}^{k})=\inf_{(s,y)\in \Lambda^{s_k}}
                         (W_2+\hbar^{k})(s,y).
$$
Now, for every $\beta>0$ and $k>0$, from the definition of viscosity solutions it follows that
  \begin{eqnarray}\label{vis1}
                      -\lambda W_1(\check{x}^{k})+\partial_t\chi^{k}(l_{k},\check{x}^{k})
  +{\mathbf{H}}{(}l_{k},\check{x}^{k},  \partial_x\chi^{k}(l_{k},\check{x}^{k}),
                                       \partial_{xx}\chi^{k}(l_{k},\check{x}^{k})
                                    {)}\geq 0,
 \end{eqnarray}
 and
  \begin{eqnarray}\label{vis2}
                      -\lambda W_2(\check{y}^{k})-\partial_t\hbar^{k}(s_{k},\check{y}^{k})+{\mathbf{H}}{(}s_{k},\check{y}^{k},
                     -\partial_x\hbar^{k}(s_{k},\check{y}^{k}),-\partial_{xx}\hbar^{k}(s_{k},\check{y}^k){)}\leq0,
  \end{eqnarray}
 where, for every $(t,x)\in {{\Lambda}^{l_k}}$ and  $ (s,y)\in {{\Lambda}^{s_k}}$, from 
     Remark \ref{remarks} (i),
  \begin{eqnarray*}
\partial_t\chi^{k}(t,x)=
                       \partial_t\varphi_k(t,x)
                     +2\varepsilon({t}-{\hat{t}})+2\sum_{i=0}^{\infty}\frac{1}{2^i}(t-t_{i}),
  \end{eqnarray*}
  \begin{eqnarray*}
\partial_x\chi^{k}(t,x)&=&\varepsilon\partial_x\Upsilon^{1,3}(x)+\partial_{x}\varphi_k(t,x)
                      +\varepsilon\partial_x\Upsilon(x-\hat{x}_{t-\hat{t}})
+2^5\beta\partial_x\Upsilon(x-\hat{\xi}_{t-\hat{t}})\\
&&+\partial_x\left[\sum_{i=0}^{\infty}\frac{1}{2^i}
                      \Upsilon(x-\gamma^{i}_{t-t_{i}})
                     \right],
  \end{eqnarray*}
   \begin{eqnarray*}
\partial_{xx}\chi^{k}(t,x)&=&\varepsilon\partial_{xx}\Upsilon^{1,3}(x)+\partial_{x}\varphi_k(t,x)
                      +\varepsilon\partial_{xx}\Upsilon(x-\hat{x}_{t-\hat{t}})
+2^5\beta\partial_{xx}\Upsilon(x-\hat{\xi}_{t-\hat{t}})\\
&&+\partial_{xx}\left[\sum_{i=0}^{\infty}\frac{1}{2^i}
                      \Upsilon(x-\gamma^{i}_{t-t_{i}})
                     \right],
  \end{eqnarray*}
   \begin{eqnarray*}
\partial_t\hbar^{k}(s,y)=
                      \partial_t\psi_k(s,y)+2\varepsilon({s}-{\hat{t}}),
  \end{eqnarray*}
  \begin{eqnarray*}
\partial_x\hbar^{k}(s,y)
                         & =&\varepsilon\partial_x\Upsilon^{1,3}(y)+ \partial_{x}\psi_k(s,y)
                      +\varepsilon\partial_x\Upsilon(y-\hat{y}_{s-\hat{t}})
                     +2^5\beta\partial_x\Upsilon(y-\hat{\xi}_{s-\hat{t}})\\
                      &&+\partial_x\left[\sum_{i=0}^{\infty}\frac{1}{2^i}
                      \Upsilon(y-{y}^{i}_{s-{t}_{i}})
                      \right],
  \end{eqnarray*}
   \begin{eqnarray*}
\partial_{xx}\hbar^{k}(s,y)
                          &=&\varepsilon\partial_{xx}\Upsilon^{1,3}(y)+ \partial_{xx}\psi_k(s,y)
                      +\varepsilon\partial_{xx}\Upsilon(y-\hat{y}_{s-\hat{t}})
                     +2^5\beta\partial_{xx}\Upsilon(y-\hat{\xi}_{s-\hat{t}})\\
                     &&
                      +\partial_{xx}\left[\sum_{i=0}^{\infty}\frac{1}{2^i}
                      \Upsilon(y-{y}^{i}_{s-{t}_{i}})
                      \right].
  \end{eqnarray*}
 \par
 $Step\ 4.$    Calculation and completion of the proof.
 \par
  We notice that,  by   (\ref{0528a}), (\ref{0528b}), (\ref{0612b}), (\ref{0612c}) and the definition of $\Upsilon$,
   there exists a generic constant $C>0$ such that
     \begin{eqnarray*}
    |\partial_x\Upsilon(\check{x}^{k}-\hat{x}_{l_k-\hat{t}})|
    +|\partial_x\Upsilon(\check{y}^{k}-\hat{y}_{s_k-\hat{t}})|\leq C|\hat{x}(0)-\check{x}^{k}(0)|^5
                      +C|\hat{y}(0)-\check{y}^{k}(0)|^5;
    \end{eqnarray*}
    \begin{eqnarray*}
     |\partial_{xx}\Upsilon(\check{x}^{k}-\hat{x}_{l_k-\hat{t}})|
    +|\partial_{xx}\Upsilon(\check{y}^{k}-\hat{y}_{s_k-\hat{t}})|\leq C|\hat{x}(0)-\check{x}^{k}(0)|^4
                      +C|\hat{y}(0)-\check{y}^{k}(0)|^4.
    \end{eqnarray*}
  Letting  $k\rightarrow\infty$ in (\ref{vis1}) and (\ref{vis2}), and using (\ref{0608v1}), (\ref{0608vw1}) and (\ref{4.23}),    we obtain
    \begin{eqnarray}\label{03103}
                   -\lambda W_1(\hat{x})+ b_1+2\sum_{i=0}^{\infty}\frac{1}{2^i}(\hat{t}-t_i)
                     +{\mathbf{H}}(\hat{x},
                    \partial_x\chi(\hat{t},\hat{x}),\partial_{xx}\chi(\hat{t},\hat{x}))
                     \geq 0;
\end{eqnarray}
and
 \begin{eqnarray}\label{03104}
                     -\lambda W_2(\hat{y}) -b_2+{\mathbf{H}}(\hat{y},-\partial_x\hbar(\hat{t},\hat{y}),-\partial_{xx}\hbar(\hat{t},\hat{y}))
                     \leq0,
\end{eqnarray}
  where
  \begin{eqnarray*}
\partial_x\chi(\hat{t},\hat{x})
                                     :=
2\beta^{\frac{1}{3}} (\hat{x}(0)-\hat{y}(0))
                                       +2^5\beta\partial_x\Upsilon(\hat{x}-\hat{\xi})
                                      +\varepsilon\partial_x\Upsilon^{1,3}(\hat{x})
                                     +\partial_x\sum_{i=0}^{\infty}\frac{1}{2^i}\Upsilon(\hat{x}- x^i_{\hat{t}-t_i}),
  \end{eqnarray*}
   \begin{eqnarray*}
\partial_{xx}\chi(\hat{t},\hat{x})
:=X+2^5\beta\partial_{xx}\Upsilon(\hat{x}-\hat{\xi})
                                    +\varepsilon\partial_{xx}\Upsilon^{1,3}(\hat{x})
                                     +\partial_{xx}\sum_{i=0}^{\infty}\frac{1}{2^i}\Upsilon(\hat{x}- x^i_{\hat{t}-t_i}),
  \end{eqnarray*}
  \begin{eqnarray*}
\partial_x\hbar(\hat{t},\hat{y})
                                     := -2\beta^{\frac{1}{3}} (\hat{x}(0)-\hat{y}(0))
                      +2^5\beta\partial_x\Upsilon(\hat{y}-\hat{\xi})+\varepsilon\partial_x\Upsilon^{1,3}(\hat{y})
                                     +\partial_x\sum_{i=0}^{\infty}\frac{1}{2^i}
                                     \Upsilon(\hat{y}-y^i_{\hat{t}-t_i})
  \end{eqnarray*}
  and
   \begin{eqnarray*}
\partial_{xx}\hbar(\hat{t},\hat{y}):=
                      Y+2^5\beta\partial_{xx}\Upsilon(\hat{y}-\hat{\xi})+\varepsilon\partial_{xx}\Upsilon^{1,3}(\hat{y})+\partial_{xx}\sum_{i=0}^{\infty}\frac{1}{2^i}
                      \Upsilon(\hat{y}-y^i_{\hat{t}-t_i}).
  \end{eqnarray*}
  Notice that $b_1+b_2=0$ and $\hat{\xi}=\frac{\hat{x}+\hat{y}}{2}$,
combining  (\ref{03103}) and (\ref{03104}),  we have
 \begin{eqnarray}\label{vis112}
                    &&  \lambda(W_1(\hat{x})-W_2(\hat{y}))-2\sum_{i=0}^{\infty}\frac{1}{2^i}(\hat{t}-t_i)\nonumber\\
                     &\leq&{\mathbf{H}}(\hat{x},
                    \partial_x\chi(\hat{t},\hat{x}),\partial_{xx}\chi(\hat{t},\hat{x}))
                     -{\mathbf{H}}(\hat{y},-\partial_x\hbar(\hat{t},\hat{y}),-\partial_{xx}\hbar(\hat{t},\hat{y})).
\end{eqnarray}
 On the other hand, via a simple calculation we obtain
 \begin{eqnarray}\label{v4}
                &&{\mathbf{H}}(\hat{x},
                    \partial_x\chi(\hat{t},\hat{x}),\partial_{xx}\chi(\hat{t},\hat{x}))
                     -{\mathbf{H}}(\hat{y},-\partial_x\hbar(\hat{t},\hat{y}),-\partial_{xx}\hbar(\hat{t},\hat{y}))\nonumber\\
                &\leq&\sup_{u\in U}(J_{1}+J_{2}+J_{3}),
\end{eqnarray}
 where
\begin{eqnarray}\label{j1}
                               J_{1}&=&\bigg{(} {b}(\hat{x},u),2\beta^{\frac{1}{3}} (\hat{x}(0)-\hat{y}(0))
                                       +2^5\beta\partial_x\Upsilon(\hat{x}-\hat{\xi})
                                      +\varepsilon\partial_x\Upsilon^{1,3}(\hat{x})
                                     \nonumber\\
                                     &&
                                     +\partial_x\sum_{i=0}^{\infty}\frac{1}{2^i}\Upsilon(\hat{x}- x^i_{\hat{t}-t_i})\bigg{)}_{\mathbb{R}^d}  -\bigg{(} {b}(\hat{y},u),2\beta^{\frac{1}{3}} (\hat{x}(0)-\hat{y}(0))
                      -2^5\beta\partial_x\Upsilon(\hat{y}-\hat{\xi})\nonumber\\
                      &&-\varepsilon\partial_x\Upsilon^{1,3}(\hat{y})
                                     -\partial_x\sum_{i=0}^{\infty}\frac{1}{2^i}
                                     \Upsilon(\hat{y}-y^i_{\hat{t}-t_i})\bigg{)}_{\mathbb{R}^d}\nonumber\\
                                  &\leq&2\beta^{\frac{1}{3}}{L}|\hat{x}(0)-\hat{y}(0)|
                                 |\hat{x}-\hat{y}|_C
                                 +18\beta L|\hat{x}(0)-\hat{y}(0)|^5L|\hat{x}-\hat{y}|_C
                                \nonumber\\
                                            &&
                                            +18L\sum_{i=0}^{\infty}\frac{1}{2^i}[|x^i(0)-\hat{x}(0)|^5
                                            +|y^i(0)-\hat{y}(0)|^5]
                                            (1+|\hat{x}|_C+|\hat{y}|_C)\nonumber\\
                                           && +12\varepsilon L(1+|\hat{x}|^2_C+|\hat{y}|^2_C);
\end{eqnarray}
\begin{eqnarray}\label{j2}
                               J_{2}&=&\frac{1}{2}\mbox{tr}\left[ \left(X+2^5\beta\partial_{xx}\Upsilon(\hat{x}-\hat{\xi})
                                    +\varepsilon\partial_{xx}\Upsilon^{1,3}(\hat{x})
                                     +\partial_{xx}\sum_{i=0}^{\infty}\frac{1}{2^i}\Upsilon(\hat{x}- x^i_{\hat{t}-t_i})\right){\sigma}(\hat{x},u)
                                        {\sigma}^\top(\hat{x},u)\right]\nonumber\\
                                        &&-\frac{1}{2}\mbox{tr}\left[ \left(-Y-2^5\beta\partial_{xx}\Upsilon(\hat{y}-\hat{\xi})-\varepsilon\partial_{xx}\Upsilon^{1,3}(\hat{y})-\partial_{xx}\sum_{i=0}^{\infty}\frac{1}{2^i}
                      \Upsilon(\hat{y}-y^i_{\hat{t}-t_i})\right)
                                        \sigma(\hat{y},u)\sigma^\top(\hat{y},u)\right]\nonumber\\
                               &\leq&3
                               \beta^{\frac{1}{3}}|{\sigma}(\hat{x},u)-\sigma(\hat{y},u)|_2^2
                                      +306\beta|\hat{x}(0)-\hat{y}(0)|^4
                                      (|{\sigma}(\hat{x},u)|_2^2+
                                      |\sigma(\hat{y},u)|_2^2)\nonumber\\
                                      &&+15\varepsilon(|{\sigma}(\hat{x},u)|_2^2+|\sigma(\hat{y},u)|_2^2)
                                      +153\sum_{i=0}^{\infty}\frac{1}{2^i}|x^i(0)-\hat{x}(0)|^4
                                        |\sigma(\hat{x},u)|_2^2\nonumber\\
                                            &&
                                        +153\sum_{i=0}^{\infty}\frac{1}{2^i}|y^i(0)-\hat{y}(0)|^4
                                        |\sigma(\hat{y},u)|_2^2
                                        \nonumber\\
                               &\leq&
                            3
                               \beta^{\frac{1}{3}}L^2|\hat{x}-\hat{y}|_C^2
                                     + 306\beta|\hat{x}(0)-\hat{y}(0)|^4L^2
                                     (2+|\hat{x}|_C^2
                                             +|\hat{y}|_C^2
                                             )+15\varepsilon L^2
                                     (2+|\hat{x}|_C^2
                                             +|\hat{y}|_C^2
                                             )\nonumber\\
                                     &&
                                     +153\sum_{i=0}^{\infty}\frac{1}{2^i}[|x^i(0)-\hat{x}(0)|^4
                                     +|y^i(0)-\hat{y}(0)|^4])L^2
                                     (1+|\hat{x}|_C^2
                                             +|\hat{y}|_C^2
                                             );
\end{eqnarray}
\begin{eqnarray}\label{j3}
                                 J_{3}=q(\hat{x}, u)
                               -
                                 q(\hat{y},u)
                                     \leq
                                 L|\hat{x}-\hat{y}|_C;
\end{eqnarray}
 We notice that, by  the property (i) of $(\hat{t},\hat{x},\hat{y})$,
  \begin{eqnarray*}
2\sum_{i=0}^{\infty}\frac{1}{2^i}(\hat{t}-t_i)
  \leq2\sum_{i=0}^{\infty}\frac{1}{2^i}\bigg{(}\frac{1}{2^i\beta}\bigg{)}^{\frac{1}{2}}\leq 4{\bigg{(}\frac{1}{{\beta}}\bigg{)}}^{\frac{1}{2}},
    \end{eqnarray*}
     \begin{eqnarray*}
  \sum_{i=0}^{\infty}\frac{1}{2^i}[|x^i(0)-\hat{x}(0)|^5
                                            +|y^i(0)-\hat{y}(0)|^5]
                      \leq 2\sum_{i=0}^{\infty}\frac{1}{2^i}\bigg{(}\frac{1}{2^i\beta}\bigg{)}^{\frac{5}{6}}\leq 4{\bigg{(}\frac{1}{{\beta}}\bigg{)}}^{\frac{5}{6}},
                        \end{eqnarray*}
                        and
                         \begin{eqnarray*}
  \sum_{i=0}^{\infty}\frac{1}{2^i}[|x^i(0)-\hat{x}(0)|^4
                                     +|y^i(0)-\hat{y}(0)|^4]
                      \leq 2\sum_{i=0}^{\infty}\frac{1}{2^i}\bigg{(}\frac{1}{2^i\beta}\bigg{)}^{\frac{2}{3}}\leq 4{\bigg{(}\frac{1}{{\beta}}\bigg{)}}^{\frac{2}{3}}.
  \end{eqnarray*}
 Combining (\ref{vis112})-(\ref{j3}), and  by (\ref{5.10jiajiaaaa}) and (\ref{5.10}) we can let $\beta>0$ be large enough such that,
 \begin{eqnarray}\label{vis122}
                      \lambda(W_1(\hat{x})-W_2(\hat{y}))
                                             \leq
                      (12 +15 L)\varepsilon L(2+|\hat{x}|^2_C+|\hat{y}|^2_C)+(12 +15 L)L\frac{\tilde{m}}{8}.
\end{eqnarray}
Recalling $
\frac{\tilde{m}}{2}
\leq\Psi(\hat{t}, \hat{x}, \hat{y})\leq W_1(\hat{x})-W_2(\hat{y})-\varepsilon(\Upsilon^{1,3}(\hat{x})+\Upsilon^{1,3}(\hat{y}))$ and $\lambda\geq (12+15L)L$, by (\ref{s0}) and (\ref{5.3}), the following contradiction is induced:
\begin{eqnarray*}\label{vis122}
                  \frac{\tilde{m}}{2}\leq  \varepsilon(2+|\hat{x}|^2_C+|\hat{y}|^2_C)+\frac{\tilde{m}}{8}-\varepsilon(\Upsilon^{1,3}(\hat{x})+\Upsilon^{1,3}(\hat{y}))\leq {2\varepsilon}+\frac{\tilde{m}}{8}\leq \frac{\tilde{m}}{4}.
\end{eqnarray*}
 The proof is now complete.
 \ \ $\Box$
  \par
 To complete the previous proof, it remains to state and prove the following lemmas. In the following Lemmas of this subsection, let $\tilde{w}_{1}^{\hat{t}}, \tilde{w}_{1}^{\hat{t},*}$ and $\tilde{w}_{2}^{\hat{t}}, \tilde{w}_{2}^{\hat{t},*}$ be the definitions in Definition \ref{definition0607} with respect to $w_1$ defined by (\ref{06091}) and  $w_2$ defined by  (\ref{06092}), respectively.
 \begin{lemma}\label{0611a}\ \
 The functionals $w_1$ and $w_2$ defined by (\ref{06091}) and (\ref{06092}) satisfy the conditons of  Theorem \ref{theorem0513}.
\end{lemma}
\par
   {\bf  Proof  }. \ \
From (\ref{s0}) and (\ref{w}),  $w_1$ and $w_2$ are upper semicontinuous functions bounded from above and satisfy (\ref{05131}).
By the following Lemmas \ref{lemma4.3} and \ref{lemma4.344}, $w_1$ and $w_2$ satisfy condition  (\ref{0608a}), and  $\tilde{w}_{1}^{\hat{t},*}\in \Phi(\hat{t},\hat{x}(0),T)$ and $\tilde{w}_{2}^{\hat{t},*}\in \Phi(\hat{t},\hat{y}(0),T)$ for some $T\in(\hat{t},\infty)$.
Moreover,  by Lemma \ref{theoremS000} and (\ref{iii4}) we obtain that, for all 
   $(t,x,y)\in [\hat{t},\infty)\times {\cal{C}}_0\times  {\cal{C}}_0$,
\begin{eqnarray}\label{wv}
                         &&w_{1}(t,x)+w_{2}(t,y)-\beta^{\frac{1}{3}}|x(0)-y(0)|^2\nonumber   \\
                          &\leq&
                   {\Psi}_1(t,x,y)\leq\Psi_1(\hat{t},\hat{x},\hat{y})=w_{1}(\hat{t},\hat{x})+w_{2}(\hat{t},\hat{y})
                          -\beta^{\frac{1}{3}}|\hat{x}(0)-\hat{y}(0)|^2,
\end{eqnarray}
                       where the last inequality becomes equality if and only if $t={\hat{t}}$, $x=\hat{x}, y=\hat{y}$.
                          Then we obtain that $
                w_1(t,x)+w_2(t,y)-\beta^{\frac{1}{3}}|x(0)-y(0)|^2
$
has a 
 maximum over $[\hat{t},\infty)\times{\cal{C}}_0\times {\cal{C}}_0$ at a point $(\hat{t},\hat{x},\hat{y})$ with $\hat{t}>0$. Thus $w_1$ and $w_2$ satisfy the conditons of  Theorem \ref{theorem0513}.  \ \ $\Box$
 \begin{lemma}\label{lemma4.3}\ \
  For every fixed $T>\hat{t}$, there exists a local modulus of continuity  $\rho_1$ 
  such that  the functionals $w_1$ and $w_2$ defined by (\ref{06091}) and (\ref{06092}) satisfy condition (\ref{0608a}).
\end{lemma}
\par
   {\bf  Proof  }. \ \
From (\ref {w12072}) and the definition of $w_{1}$, we have that,
for every $\hat{t}\leq t\leq s\leq T$ and $x\in {\cal{C}}_0$,
\begin{eqnarray*}
&&w_1(t,x)-w_1(s,x_{s-t})\\&=&
W_1(x)-2^5\beta \Upsilon(t,x,\hat{t},\hat{\xi})-\varepsilon\Upsilon^{1,3}(x)
                 -\varepsilon \overline{\Upsilon}(\hat{t},\hat{x},t,x)
                -\sum_{i=0}^{\infty}
        \frac{1}{2^i}\overline{\Upsilon}(t_i,x^i,t,x)\\
        &&-W_1(x_{s-t})+2^5\beta \Upsilon(s,x_{s-t},\hat{t},\hat{\xi})+\varepsilon\Upsilon^{1,3}(x_{s-t})
                 +\varepsilon \overline{\Upsilon}(\hat{t},\hat{x},s,x_{s-t})
                +\sum_{i=0}^{\infty}
        \frac{1}{2^i}\overline{\Upsilon}(t_i,x^i,s,x_{s-t})\\
        &=&W_1(x)-W_1(x_{s-t})+\varepsilon((s-\hat{t})^2-(t-\hat{t})^2)+\sum_{i=0}^{\infty}\frac{1}{2^i}((s-t_i)^2-(t-t_i)^2)\\
        &\leq&
                                   \hat{C}(1+|x|_C)((s-t)+(s-t)^{\frac{1}{2}}+1-e^{-\lambda (s-t)})+(2\varepsilon+4)T(s-t).
\end{eqnarray*}
Taking $\rho_1(h,z)=\hat{C}(1+z)(h+h^{\frac{1}{2}}+1-e^{-\lambda h})+(2\varepsilon+4)Th$, 
  $(h,z)\in[0,\infty)\times [0,\infty)$, it is clear that $\rho_1$ is a local modulus of continuity 
   and $w_1$  satisfies condition (\ref{0608a}) with it.  In a similar way, we show that  $w_2$  satisfies condition (\ref{0608a}) with this $\rho_1$.
The proof is now complete. \ \ $\Box$
 \begin{lemma}\label{lemma4.344}\ \   $\tilde{w}_{1}^{\hat{t},*}\in \Phi(\hat{t},\hat{x}(0),T)$ and $\tilde{w}_{2}^{\hat{t},*}\in \Phi(\hat{t},\hat{y}(0),T)$ for every $T\in (\hat{t},\infty)$. 
\end{lemma}
\par
   {\bf  Proof  }. \ \  We only prove $\tilde{w}_{1}^{\hat{t},*}\in \Phi(\hat{t},\hat{x}(0),T)$ for every $T\in (\hat{t},\infty)$. 
                           $\tilde{w}_{2}^{\hat{t},*}\in \Phi(\hat{t},\hat{y}(0),T)$ can be obtained by a symmetric  way.
        Set $r=1
 $, for a given $L_1>0$, let $\varphi\in C^{1,2}([0,T]\times \mathbb{R}^d)$ be a function such that
            $\tilde{w}_{1}^{\hat{t},*}(t,x_0)-\varphi(t,x_0)$  has a  maximum at $(\bar{{t}},\bar{x}_{0})\in (0, \infty)\times \mathbb{R}^d$, moreover, the following inequalities hold true:
\begin{eqnarray*}
                    &&|\bar{{t}}-{\hat{t}}|+|\bar{x}_{0}-\hat{x}(0)|<r,\\
                    &&|\tilde{w}_{1}^{\hat{t},*}(\bar{{t}},\bar{x}_{0})|+|\nabla_x\varphi(\bar{{t}},\bar{{x}}_{0})|
                    +|\nabla_{xx}\varphi(\bar{{t}},\bar{{x}}_{0})|\leq L_1.
\end{eqnarray*}
We can modify $\varphi$  such that $\varphi\in C^{1,2}([0,\infty)\times \mathbb{R}^d)$ , ${\varphi}$, $\nabla_{x}{\varphi}$ and $\nabla_{xx}{\varphi}$  grow  in a polynomial way, $\tilde{w}_{1}^{\hat{t},*}(t,x_0)-\varphi(t,x_0)$  has a strict   maximum at $(\bar{{t}},\bar{x}_0)\in (0, \infty)\times \mathbb{R}^d$ and the above two inequalities hold true.
If $\bar{{t}}<\hat{t}$, we have $ b=\varphi_{t}(\bar{{t}},\bar{x}_{0})=\frac{1}{2}(\hat{t}-\bar{{t}})^{-\frac{1}{2}} \geq 0$.
If $\bar{{t}}\geq \hat{t}$, we consider the functional
$$
                \Gamma(t,x)= w_{1}(t,x)
                 -\varphi(t,x(0)),\ (t,x)\in \Lambda^{\hat{t}}.
$$
We may assume that $\varphi$ grow quadratically at $\infty$. By (\ref{s0}) and (\ref{w}), it is clear that $\Gamma$ is bounded from above on ${\Lambda}^{\hat{t}}$. Moreover,  by Lemma \ref{theoremS}, $\Gamma$ is an  upper semicontinuous  functional.
 Define a sequence of positive numbers $\{\delta_i\}_{i\geq0}$  by 
        $\delta_i=\frac{1}{2^i}$ for all $i\geq0$.  For every  $0<\delta<1$,
 by Lemma \ref{theoremleft} we have that,
 for every  $(\breve{t}_0,\breve{x}^0)\in \Lambda^{\bar{t}}$ satisfying
\begin{eqnarray}\label{0615a}
\Gamma(\breve{t}_0,\breve{x}^0)\geq \sup_{(s,x)\in \Lambda^{\bar{t}}}\Gamma(s,x)-\delta,
\end{eqnarray}
  there exist $(\breve{t},\breve{x})\in \Lambda^{\bar{t}}$ and a sequence $\{(\breve{t}_i,\breve{x}^i)\}_{i\geq1}\subset
  \Lambda^{\bar{t}}$ such that
  \begin{description}
        \item{(i)} $\overline{\Upsilon}(\breve{t}_0,\breve{x}^0,\breve{t},\breve{x})\leq \delta$,
         $\overline{\Upsilon}(\breve{t}_i,\breve{x}^i,\breve{t},\breve{x})\leq \frac{\delta}{2^i}$ and $t_i\uparrow \breve{t}$ as $i\rightarrow\infty$,
        \item{(ii)}  $\Gamma(\breve{t},\breve{x})
            -\sum_{i=0}^{\infty}\frac{1}{2^i}\overline{\Upsilon}(\breve{t}_i,\breve{x}^i,\breve{t}, \breve{x})
        \geq \Gamma(\breve{t}_0,\breve{x}^0)$, and
        \item{(iii)}    for all $(s,x)\in \Lambda^{\breve{t}}\setminus \{(\breve{t},\breve{x})\}$,
        \begin{eqnarray*}
        \Gamma(s,x)
        -\sum_{i=0}^{\infty}
        \frac{1}{2^i}\overline{\Upsilon}(\breve{t}_i,\breve{x}^i,s,x)
            <\Gamma(\breve{t},\breve{x})
            -\sum_{i=0}^{\infty}\frac{1}{2^i}\overline{\Upsilon}(\breve{t}_i,\breve{x}^i,\breve{t},\breve{x}).
        \end{eqnarray*}
        \end{description}
   We should note that the point
             $(\breve{t},\breve{x})$ depends on $\delta$.
        By Lemma \ref{lemma4.3}, $w_1$   satisfies condition (\ref{0608a}). Then, by the definitions of $\tilde{w}^{\hat{t}}_1$ and $\tilde{w}_{1}^{\hat{t},*}$, we have
 \begin{eqnarray*}
        &&\tilde{w}_{1}^{\hat{t},*}(\bar{{t}},\bar{x})-\varphi(\bar{{t}},\bar{x})\\
        &=&\limsup_{s\geq \hat{t},(s,y)\rightarrow(\bar{{t}},\bar{x})}(\tilde{w}^{\hat{t}}_1(s,y)-\varphi(s,y))
        =\limsup_{s\geq \hat{t}, (s,y)\rightarrow(\bar{{t}},\bar{x})}\left(\sup_{\gamma\in {\cal{C}}_0,\gamma(0)=y}
                             {[}w_1(s,\gamma){]}-\varphi(s,y)\right)\\
                             &=&\limsup_{s\geq \bar{t}, (s,y)\rightarrow(\bar{{t}},\bar{x})}\sup_{\gamma\in {\cal{C}}_0,\gamma(0)=y}
                             {[}w_1(s,\gamma)-\varphi(s,\gamma(0)){]}\leq\sup_{(s,x)\in  \Lambda^{\bar{t}}}\Gamma(s,x).
 \end{eqnarray*}
 Combining with (\ref{0615a}),
 $$
 \Gamma(\breve{t}_0,\breve{x}^0)\geq \sup_{(s,x)\in  \Lambda^{\bar{t}}}\Gamma(s,x)-\delta
\geq\tilde{w}_{1}^{\hat{t},*}(\bar{{t}},\bar{x}_0)-\varphi(\bar{{t}},\bar{x}_0)-\delta.
 $$
  Recall that $\tilde{w}_{1}^{\hat{t},*}\geq\tilde{w}^{\hat{t}}_1$.  Then, by 
   the definition of $\tilde{w}^{\hat{t}}_1$ and the property (ii) of $(\breve{t},\breve{x})$,
\begin{eqnarray}\label{20210509}
                 \tilde{w}_{1}^{\hat{t},*}(\breve{t},\breve{x}(0))
                   -\varphi(\breve{t},\breve{x}(0))
                   &\geq&\tilde{w}^{\hat{t}}_1(\breve{t},\breve{x}(0))
                   -\varphi(\breve{t},\breve{x}(0))
                 \geq w_1(\breve{t},\breve{x})
                -\varphi(\breve{t},\breve{x}(0))\nonumber\\
                  &\geq& \Gamma(\breve{t}_0,\breve{x}^0)
                 \geq\tilde{w}_{1}^{\hat{t},*}(\bar{{t}},\bar{x}_0)-\varphi(\bar{{t}},\bar{x}_0)-\delta.
\end{eqnarray}
Noting $\varepsilon$ is independent of  $\delta$ and $\varphi$ grows quadratically at $\infty$, by the definition of  $\Gamma$,
 there exists a constant  ${M}_3>0$  independent of $\delta$ that is sufficiently  large   that
 $
           \Gamma(t,x)<\sup_{(s,x)\in  \Lambda^{\bar{t}}}\Gamma(s,x)-1
           $ for all $t+|x|_C\geq M_3$. Thus, we have $\breve{t}\vee|\breve{x}|_C\vee
           \breve{t}_{0}\vee|\breve{x}^{0}|_C<M_3$. In particular, $|\breve{x}(0)|<M_3$.  
Letting $\delta\rightarrow0$, by  the similar proof procedure of (\ref{4.22}), we obtain
\begin{eqnarray}\label{delta0}
       && \breve{t}\rightarrow \bar{t},\ \breve{x}(0)\rightarrow \bar{x}_0 \ \mbox{as}\ \delta\rightarrow0.
       \end{eqnarray}
 Thus,  the definition of the viscosity subsolution can be used to obtain the following result:
 \begin{eqnarray}\label{5.15}
                      -\lambda W_1(\check{x})+\partial_t\Im(\breve{t},\breve{x})
                     +{\mathbf{H}}{(}\breve{x},   \partial_x\Im(\breve{t},\breve{x}),\partial_{xx}\Im(\breve{t},\breve{x}){)}\geq 0.
 \end{eqnarray}
  where, for every $(t,x)\in {{\Lambda}}^{\breve{t}}$,
  \begin{eqnarray*}
\Im(t,x)
        &:=&\varepsilon\Upsilon^{1,3}(x)
                +\varepsilon \overline{\Upsilon}(t,x,\hat{t},\hat{x})
                +\sum_{i=0}^{\infty}
        \frac{1}{2^i}\overline{\Upsilon}(t_i,x^i,t,x)+2^5\beta\Upsilon(t,x,\hat{t},\hat{\xi})\\
               && +\sum_{i=0}^{\infty}
        \frac{1}{2^i}\overline{\Upsilon}(\breve{t}_{i},\breve{x}^i,t,x)+\varphi(t,x(0)),
  \end{eqnarray*}
 \begin{eqnarray*}
\partial_t\Im(t,x):=
                     2\varepsilon({t}-{\hat{t}})+2\sum_{i=0}^{\infty}\frac{1}{2^i}[(t-t_{i})+(t-\breve{t}_i)]+\varphi_t({t},x(0)),
  \end{eqnarray*}
  \begin{eqnarray*}
\partial_x\Im(t,x)&:=&
                     \varepsilon\partial_x\Upsilon^{1,3}(x) +\varepsilon\partial_x\Upsilon(x-\hat{x}_{t-\hat{t}})
+2^5\beta\partial_x\Upsilon(x-\hat{\xi}_{t-\hat{t}})
\\
 &&+\partial_x\left[\sum_{i=0}^{\infty}\frac{1}{2^i}
                      \Upsilon(x-x^{i}_{t-t_{i}})
                      +\sum_{i=0}^{\infty}\frac{1}{2^i}\Upsilon(x-\breve{x}^i_{t-\breve{t}_i})\right]+\nabla_{x}\varphi({t}, x(0)),
  \end{eqnarray*}
   \begin{eqnarray*}
\partial_{xx}\Im(t,x)&:=&
                     \varepsilon\partial_{xx}\Upsilon^{1,3}(x) +\varepsilon\partial_{xx}\Upsilon(x-\hat{x}_{t-\hat{t}})
+2^5\beta\partial_{xx}\Upsilon(x-\hat{\xi}_{t-\hat{t}})
\\
 &&+\partial_{xx}\left[\sum_{i=0}^{\infty}\frac{1}{2^i}
                      \Upsilon(x-x^{i}_{t-t_{i}})
                      +\sum_{i=0}^{\infty}\frac{1}{2^i}\Upsilon(x-\breve{x}^i_{t-\breve{t}_i})\right]+\nabla_{x}\varphi({t}, x(0)).
  \end{eqnarray*}
  We notice that  $\breve{t}\vee|\breve{x}|_C\leq M_3$.
Then letting $\delta\rightarrow0$ in (\ref{5.15}),   by the definition of ${\mathbf{H}}$, it follows that there exists a constant $C$ such that $ b=\varphi_{t}(\bar{t},\bar{x}_0) \geq C$.
The proof is now complete. \ \ $\Box$

\begin{lemma}\label{lemma4.4}\ \ The maximum points $(l_{k},\check{x}^{k}, s_{k},\check{y}^{k})$
 satisfy  condition (\ref{4.23}).
\end{lemma}
\par
   {\bf  Proof  }. \ \
 Without loss of generality, we may assume $s_{k}\leq l_{k}$, by  (\ref{up}), (\ref{w12072}), (\ref{iii4}) and the definitions of $w_1$ and $w_2$, we have that
\begin{eqnarray}\label{06101}
                      &&w_1(l_{k},\check{x}^{k})+w_2(s_{k},\check{y}^{k})
               -\beta^{\frac{1}{3}} |\check{x}^{k}(0)-\check{y}^{k}(0)|^2\nonumber\\
                      &\leq&\Psi_1(l_{k},\check{x}^{k},\check{y}^k_{l_k-s_k})  -W_2(\check{y}^{k})
                      +W_2(\check{y}^k_{l_k-s_k})
                       -\varepsilon[\overline{\Upsilon}(l_{k},\check{x}^{k},\hat{t},\hat{x})
+\overline{\Upsilon}(s_k,\check{y}^k,\hat{t},\hat{y})
]\nonumber\\
                       &\leq&\Psi_1(\hat{t},\hat{x},\hat{y})
                    +\hat{C}(1+|\check{y}^k|_C)((l_{k}- s_{k})+(l_{k}- s_{k})^{\frac{1}{2}}+1-e^{-\lambda(l_{k}- s_{k})})
                    \nonumber\\
                       &&
                      -\varepsilon[\overline{\Upsilon}(l_{k},\check{x}^{k},\hat{t},\hat{x})
+\overline{\Upsilon}(s_k,\check{y}^k,\hat{t},\hat{y})
].
\end{eqnarray}
By (\ref{0608vw1}), $w_2(s_k,\check{y}^k)\rightarrow w_2(\hat{t},\hat{y})$ as $k\rightarrow\infty$. Then by  that $w_2$ satisfies condition (\ref{05131}),  
 there exists a constant  ${M}_4>0$  that is sufficiently  large   that
$$
 |\check{y}^k|_C\leq M_4,\ \mbox{for all} \ k>0.
$$
Letting $k\rightarrow\infty$ in (\ref{06101}), by (\ref{0608v1}) and (\ref{0608vw1}) we have that
\begin{eqnarray*}
                       \Psi_1(\hat{t},\hat{x},\hat{y})&=&
                       w_1(\hat{t},\hat{x})+w_2(\hat{t},\hat{y})
               -\beta^{\frac{1}{3}} |\hat{x}(0)-\hat{y}(0)|^2\\
                      &\leq& \Psi_1(\hat{t},\hat{x},\hat{y})
                            -\varepsilon\limsup_{k\rightarrow\infty}[\overline{\Upsilon}(l_{k},\check{x}^{k},\hat{t},\hat{x})
+\overline{\Upsilon}(s_k,\check{y}^k,\hat{t},\hat{y})
].
\end{eqnarray*}
Thus,
$$
\lim_{k\rightarrow\infty}[\overline{\Upsilon}(l_{k},\check{x}^{k},\hat{t},\hat{x})
+\overline{\Upsilon}(s_k,\check{y}^k,\hat{t},\hat{y})
]=0.
$$
Then by (\ref{s0}) we get
(\ref{4.23}) holds true.
  The proof is now complete. \ \ $\Box$

\appendix

 \section{ Existence and consistency for viscosity solutions.}

\setcounter{equation}{0}
\renewcommand{\theequation}{A.\arabic{equation}}

\par
  {\bf  Proof (of Lemma \ref{theoremvexist}) }. \ \
 First, let  $\varphi\in {\cal{A}}^+({t}, {x},V)$
                  with
                   $({t}, {x})\in \Lambda$, then  for fixed  $u\in U$,  by the DPP (Theorem \ref{theoremvalue}), we obtain the following result:
\begin{eqnarray*}
                 \varphi(t,x)&=& V(x)
                 \leq \int_{0}^{s}e^{-\lambda l}{\mathbb{E}}q(X^{x,u}_l,u)dl+e^{-\lambda s}{\mathbb{E}}V(X^{x,u}_s)\\
                 &\leq&\int_{0}^{s}e^{-\lambda l}{\mathbb{E}}q(X^{x,u}_l,u)dl+e^{-\lambda s}{\mathbb{E}}\varphi(t+s,X^{x,u}_s),\ \ s\geq 0.
\end{eqnarray*}
                    Thus,
  \begin{eqnarray*}
                 0  \leq\frac{1}{s}\int_{0}^{s}e^{-\lambda l}{\mathbb{E}}q(X^{x,u}_l,u)dl
                            +\frac{1}{s}{[}e^{-{\lambda} s}{\mathbb{E}}\varphi(t+s,X^{x,u}_s)-\varphi(t,x){]}.
\end{eqnarray*}
             Now,  applying   functional It\^{o} formula (\ref{statesop0})  to
                         $e^{-\lambda s}\varphi(t+s,X^{x,u}_{s})$,  we have that
  \begin{eqnarray*}
                              0\leq\frac{1}{s}\int_{0}^{s}e^{-\lambda l}[{{\mathbb{E}}}q(X^{x,u}_l,u)-\lambda\mathbb{E}\varphi(t+l,X^{x,u}_l)
                             +\mathbb{E}({\cal{L}}\varphi)(t+l,X^{x,u}_l,u)]dl,
\end{eqnarray*}
where, for every $(s,x,u)\in  \Lambda^t\times U$ and $\varphi\in C_p^{1,2}(\Lambda^t)$,
\begin{eqnarray*}
                       ({\cal{L}}{\varphi})(s,x,u)
                       =\partial_t\varphi(s,x)+
                                         (\partial_x {\varphi}(s,x),b(x,u))_{\mathbb{R}^d}+\frac{1}{2}\mbox{tr}[\partial_{xx}{\varphi}(s,x)\sigma(x,u)\sigma^{\top}(x,u)].
\end{eqnarray*}
Letting $s\rightarrow0$,
\begin{eqnarray*}
                              0\leq q(x,u)-\lambda V(x)
                             +({\cal{L}}{\varphi})(t,x,u).
\end{eqnarray*}
Taking the infimum over $u\in U$, we see that $V$ is a viscosity subsolution of (\ref{hjb1}).
 \par
 Let  $\varphi\in {\cal{A}}^-({t}, {x},V)$
                  with
                   $({t}, {x})\in \Lambda$, then any $\varepsilon>0$ and   $s>0$,  by the DPP (Theorem \ref{theoremvalue}), one can find a control    ${u}^{\varepsilon}(\cdot)\equiv u^{{\varepsilon},s}(\cdot)\in {\cal{U}}_0$ such
   that the following result holds:
 \begin{eqnarray*}\label{jiajiajia}
\varepsilon s &\geq& \int_{0}^{s}e^{-\lambda l}{{\mathbb{E}}}q(X^{x,{u}^{\varepsilon}}_l,{u}^{\varepsilon}(l))dl
                   +e^{-\lambda s}{{\mathbb{E}}}V(X^{x,{u}^{\varepsilon}}_s)- V(x)\nonumber\\
                   &\geq&\int_{0}^{s}e^{-\lambda l}{{\mathbb{E}}}q(X^{x,{u}^{\varepsilon}}_l,{u}^{\varepsilon}(l))dl
                   -e^{-\lambda s}{{\mathbb{E}}}\varphi(t+s,X^{x,{u}^{\varepsilon}}_s)+\varphi(t,x).
\end{eqnarray*}
Applying also (\ref{statesop0}) to
                         $e^{-\lambda s}\varphi(t+s,X^{x,{{u}^{\varepsilon}}}_{s})$,
 we have that
  \begin{eqnarray*}\label{bsde4.2144}
                            \varepsilon &\geq&\frac{1}{s}\int_{0}^{s}e^{-\lambda l}[{{\mathbb{E}}}q(X^{x,{u}^{\varepsilon}}_l,{u}^{\varepsilon}(l))+\lambda \mathbb{E}\varphi(t+l,X^{x,{u}^{\varepsilon}}_l)
                             -\mathbb{E}({\cal{L}}\varphi)(t+l,X^{x,{u}^{\varepsilon}}_l,{u}^{\varepsilon}(l))]dl\nonumber\\
                             &=&-\lambda V(x)+\frac{1}{s}\int_{0}^{s}e^{-\lambda l}[{{\mathbb{E}}}q(x,{u}^{\varepsilon}(l)-\mathbb{E}({\cal{L}}\varphi)(t+l,x,{u}^{\varepsilon}(l))]dl+o(1)\nonumber\\
                             &\geq&-\lambda V(x)+[-\partial_t \varphi(t,x)
                           +{\mathbf{H}}(x,-\partial_x\varphi(t,x),
                           -\partial_{xx}\varphi(t,x))]\frac{1}{s}\int_{0}^{s}e^{-\lambda l}dl+o(1).
\end{eqnarray*}
Letting $s\rightarrow0^+$,
 we obtain the following inequality:
\begin{eqnarray*}
                       \varepsilon\geq -\lambda V(x)-\partial_t\varphi(x)+{\mathbf{H}}(x,-\partial_x\varphi(t,x),
                           -\partial_{xx}\varphi(t,x)).
\end{eqnarray*}
By the arbitrariness of $\varepsilon$, we show
 $V$ is  a viscosity subsolution to (\ref{hjb1}). This step completes the proof.
 \ \ $\Box$
\par
  {\bf  Proof (of Lemma \ref{theorem3.2}) }. \ \
    Assume $v$ is a viscosity solution. For any $x\in {\cal{C}}_0$, since $v\in C_p^{1,2}({\cal{C}}_0)$, by definition of viscosity solutions we see that $$
-\lambda v(x)+\partial_tv(x)+{\mathbf{H}}(x,\partial_xv(x),\partial_{xx}v(x))= 0.
$$
On the other hand, assume $v$ is a  classical solution. For every $({t}, {x})\in \Lambda$, let  $\varphi\in {\cal{A}}^+({t}, {x},v)$. For every $\alpha\in \mathbb{R}^d$ and $\gamma\in \mathbb{R}^{d\times n}$, let $\tau=t$, $\vartheta(\cdot)\equiv\alpha$ and $\varpi(\cdot)\equiv\gamma$  in (\ref{formular1}), applying functional formula (\ref{statesop0}) and noticing that $v(x)-\varphi(t,x)=0$, we have, for every $\delta>0$,
 \begin{eqnarray}\label{1224}
                            0\leq\mathbb{E}(\varphi-v)(X_{t+\delta})
                            &=&\mathbb{E}\int^{t+\delta}_{t}[\partial_t\varphi(l,X_l)-\partial_tv(X_l)
                 +(\partial_x\varphi(l,X_l)-\partial_xv(X_l),\alpha)_{\mathbb{R}^d}]dl\nonumber\\
                 &&+\mathbb{E}\int^{t+\delta}_{t}\frac{1}{2}\mbox{tr}((\partial_{xx}\varphi(l,X_l)-\partial_{xx}v(X_l))\gamma\gamma^\top)dl\nonumber\\
                 &=&
                           \mathbb{E}\int^{t+\delta}_{t}\widetilde{{\mathcal{H}}}(l,X_l)
                              dl,
\end{eqnarray}
where $$\widetilde{{\mathcal{H}}}(s,y)=\partial_t\varphi(s,y)-\partial_tv(y)
                 +(\partial_x\varphi(s,y)-\partial_xv(y),\alpha)_{\mathbb{R}^d}
                 +\frac{1}{2}\mbox{tr}((\partial_{xx}\varphi(s,y)-\partial_{xx}v(y))\gamma\gamma^\top), \ \ (s,y)\in  \Lambda. $$
Letting $\delta\rightarrow0$,
\begin{eqnarray}\label{0713a}
\widetilde{{\mathcal{H}}}(t,x)\geq0.
\end{eqnarray}
Let $\gamma=\mathbf{0}$, by the arbitrariness of  $\alpha$,
$$
\partial_t\varphi(t,x)\geq\partial_tv(x), \ \ \partial_x\varphi(t,x)=\partial_xv(x). 
$$
For every $u\in U$, let $\gamma=\sigma(x,u)$ in (\ref{0713a}),
$$
\partial_t\varphi(t,x)
                 +\frac{1}{2}\mbox{tr}(\partial_{xx}\varphi(t,x)\sigma(x,u)\sigma(x,u)^\top)
                 \geq \partial_tv(x)+\frac{1}{2}\mbox{tr}(\partial_{xx}v(x)\sigma(x,u)\sigma^\top(x,u)).
$$
Noting that $\varphi(t,x)=v(x)$ and $\partial_x\varphi(t,x)=\partial_xv(x)$, we show that
\begin{eqnarray*}
 &&-\lambda\varphi(t,x)+\partial_t\varphi(t,x)+(\partial_x\varphi(t,x),b(x,u))_{\mathbb{R}^d}
                 +\frac{1}{2}\mbox{tr}(\partial_{xx}\varphi(t,x)\sigma(x,u)\sigma^\top(x,u))+q(x,u)\\
                 &\geq& -\lambda v(x)+\partial_tv(x)+(\partial_xv(x),b(x,u))_{\mathbb{R}^d}
                 +\frac{1}{2}\mbox{tr}(\partial_{xx}v(x)\sigma(x,u)\sigma^\top(x,u))+q(x,u),
\end{eqnarray*}
Taking the infimum over $u\in U$, we see that
\begin{eqnarray*}
 -\lambda\varphi(t,x)+\partial_t\varphi(t,x)+{\mathbf{H}}(x,\partial_x\varphi(t,x),\partial_{xx}\varphi(t,x))
                 \geq -\lambda v(x)+\partial_tv(x)+{\mathbf{H}}(x,\partial_xv(x),\partial_{xx}v(x)).
\end{eqnarray*}
Note that $-\lambda v(x)+\partial_tv(x)
                           +{\mathbf{H}}(x,\partial_xv(x),\partial_{xx}v(x))=0$. Thus,
\begin{eqnarray*}\label{12241}
                            -\lambda\varphi(t,x)+\partial_t\varphi(t,x)+{\mathbf{H}}(x,\partial_x\varphi(t,x),\partial_{xx}\varphi(t,x))\geq0.
\end{eqnarray*}
   We have that $v$ is a viscosity subsolution of (\ref{hjb1}). In a symmetric way, we show that $v$ is also a viscosity supersolution to equation (\ref{hjb1}). \ \ $\Box$

\end{document}